\pdfoutput=1

\documentclass[a4paper,12pt]{article}

\usepackage{mathtools, bbm, bm, array, amsmath, amsthm, amsfonts, amssymb, mathrsfs}
\usepackage{tikz}
\usepackage[margin=1cm]{caption}
\usepackage{subcaption}
\usepackage{environ}
\usepackage{graphicx}
\usepackage[
  linkcolor=blue,
  colorlinks=true,
  citecolor=blue,
  urlcolor=black,
  filecolor=black%
]{hyperref}
\usepackage[
  a4paper,
  total={170mm,237mm},
  left=20mm,
  top=30mm%
]{geometry}
\usepackage{datetime}
\usepackage[T1]{fontenc}
\usepackage{lmodern}
\usepackage[normalem]{ulem}
\usepackage{extarrows}
\usepackage[inline]{enumitem}
\usepackage[numbers, sort, compress]{natbib}
\usepackage{setspace}
\usepackage{changepage} % adjust margins locally
\usepackage[ragged, hang]{footmisc}

\usepackage{empheq}
\usepackage{xcolor}

\newtheorem{theorem}{Theorem}% 
\newtheorem{lemma}[theorem]{Lemma}% 
\newtheorem{corollary}[theorem]{Corollary}% 
\newtheorem{definition}[theorem]{Definition}%

\DeclareUnicodeCharacter{025B}{\ensuremath{\varepsilon}}

\newcommand{\injop}[2]{\text{id}_{#1, #2}}
\newcommand{\kminus}{\underline{\kappa}}
\newcommand{\kplus}{\overline{\kappa}}
\newcommand{\cfracminus}{\underline{\mathfrak{c}}}
\newcommand{\cfracplus}{\overline{\mathfrak{c}}}

\let\originalleft\left
\let\originalright\right
\renewcommand{\left}{\mathopen{}\mathclose\bgroup\originalleft}
\renewcommand{\right}{\aftergroup\egroup\originalright}

\usetikzlibrary{shapes}
\usetikzlibrary{decorations.markings}
\usetikzlibrary{patterns}

\newdateformat{monthyeardate}{%
  \monthname[\THEMONTH] \THEYEAR} % just month and year

\colorlet{eqframe}{black!35}   
\newcommand\blfootnote[1]{%
  \begingroup
  \renewcommand\thefootnote{}\footnote{#1}%
  \addtocounter{footnote}{-1}%
  \endgroup
}

\title{Metric Entropy of Analytic Function Classes via Ellipsoidal Methods}
\date{}
\author{Thomas Allard \\ tallard@stanford.edu  \and Helmut Bölcskei \\ hboelcskei@ethz.ch}

\begin{document}

\maketitle

\blfootnote{\label{ack}The authors gratefully acknowledge support by the Lagrange Mathematics and Computing Research Center, Paris, France.}

\abstract{\noindent
We present a systematic methodology for characterizing the metric entropy of 
infinite-dimensional ellipsoids with exponentially decaying semi-axes. The 
approach does not rely on the explicit construction of coverings or packings and 
yields a unified framework for deriving sharp entropy estimates for a wide range 
of analytic function classes, including periodic functions analytic on a strip, 
analytic functions bounded on a disk, and functions of exponential type. In each 
of these cases, our results improve upon the best known bounds in the literature. 
% From a broader perspective, our framework can be seen as a step toward using 
% metric entropy not only at an order-of-magnitude level but also in a more 
% quantitatively refined way for complexity estimation.
}

%We present a \textcolor{blue}{systematic} methodology for the characterization of the metric entropy of infinite-dimensional ellipsoids with exponentially %decaying semi-axes. 
%This procedure does not rely on the explicit construction of coverings or packings and
%provides a unified framework
%for the derivation of the metric entropy of a wide variety of analytic function classes, such as periodic functions analytic on a strip, analytic functions %bounded on a disk, and functions of exponential type. In each of these cases, our results improve upon the best known results in the literature.}

\section{Introduction}\label{introduction}

The concept of metric entropy has traditionally played a significant role in various domains of mathematics such as non-linear approximation theory \cite{lorentzMetricEntropyApproximation1966,lorentzApproximationFunctions1966,carlEntropyCompactnessApproximation1990}, non-parametric estimation \cite{haussler1997mutual,Pinskerpaper,donohoCountingBitsKolmogorov2000},
statistical learning theory \cite{hausslerDecisionTheoreticGeneralizations1992,wainwrightHighDimensionalStatistics2019,shalev-shwartzUnderstandingMachineLearning2014}, and
empirical process theory \cite{dudleyEcoleEteProbabilites1984,pollardEmpiricalProcessesTheory1990}.
Recent advances in machine learning theory, more specifically in deep learning, have led to renewed interest in methods for metric entropy computation. 
Indeed, metric entropy is at the heart of the approximation-theoretic characterization of deep neural networks \cite{elbrachterDeepNeuralNetwork2021,shalev-shwartzUnderstandingMachineLearning2014,devore2021neural}.
However, computing the precise value of the metric entropy of a given function class turns out to be notoriously difficult in general; exact expressions are available only in very few simple cases.
It has therefore become common practice to resort to characterizations of the asymptotic behavior as the covering ball radius approaches zero. 
Even this more modest endeavor has turned out daunting in most cases.
A sizeable body of corresponding research exists in the literature \cite{kolmogorovCertainAsymptoticCharacteristics1956,mityaginApproximateDimensionBases1961,vitushkinTheoryTransmissionProcessing1961}.
The work of Donoho \cite{donohoCountingBitsKolmogorov2000} constitutes a canonical example of the type of asymptotic results sought; 
specifically, it provides the exact expression for the leading term in the asymptotic expansion of the metric entropy of unit balls in Sobolev spaces.

The methods for characterizing the asymptotic behavior of metric entropy available in the literature are usually highly specific to the function class under consideration.
The survey \cite[Chapter 7]{shiryayevSelectedWorksKolmogorov1993} illustrates this point in the context of complex-analytic functions. 
An in-depth study of the variety of methods available in the literature leads to the insight that the underlying ellipsoidal structure of the function classes considered can be exploited to formulate a comprehensive methodology and en passant improve many of the best known results.
This will, in fact, be the main goal of the present paper.
A first step toward such a general method was made in 
\cite{luschgySharpAsymptoticsKolmogorov2004, grafSharpAsymptoticsMetric2004,luschgyFunctionalQuantizationGaussian2002,secondpaper,thirdpaper,Pinskerpaper} by computing the metric entropy of infinite-dimensional ellipsoids with semi-axes of regularly varying (typically going to zero as the inverse of a polynomial) or slowly varying (typically going to zero as the inverse of a logarithm) lengths.
This approach was applied in \cite{luschgySharpAsymptoticsKolmogorov2004,secondpaper,thirdpaper,Pinskerpaper} to recover and improve on the above-mentioned result by Donoho.

% A comparable systematic treatment for ellipsoids with exponentially 
% decaying semi-axes does not appear to be available in the literature. 
% Here we address this gap and develop a general 
% % three-step 
% procedure for 
% determining the asymptotic behavior of the metric entropy of a wide 
% variety of function classes. The versatility of our approach is illustrated by applying it to classes of analytic functions on a strip, analytic functions bounded on a disk, and entire functions of exponential type; in each of these cases, we improve upon the best known results in the literature.
Three central classes of analytic functions---periodic functions analytic on a strip, analytic functions bounded on a disk, and entire functions of exponential type---share the property that their metric entropy is governed by that of ellipsoids with exponentially decaying semi-axes.
Metric entropy results for these classes exist in the literature, obtained through explicit, case-by-case constructions of coverings and packings, but no unified treatment exploiting this common ellipsoidal structure appears to be available.
Here we address this gap and establish sharp asymptotic metric entropy results for each of these three function classes, improving upon the best known results in the literature in each case.
    These function classes arise naturally in harmonic analysis, sampling
    theory, and time-frequency analysis. 
    The last of the three---entire functions of exponential type---is, by the
    Paley--Wiener theorem \cite[Theorem 19.3]{rudinRealComplexAnalysis1987}, precisely the class of band-limited signals of classical sampling
    theory; 
    the Paley--Wiener and Bernstein spaces are canonical function spaces in this setting \cite{higgins1996sampling}.
    Metric entropy quantifies the effective number of degrees of freedom; 
    for band-limited signals of bandwidth $W$, the metric entropy rate is proportional to $ W \log_2(1/\varepsilon)$ bits per unit time, where $\varepsilon > 0$ is
    the prescribed approximation accuracy \cite{daubechiesConversionImperfectQuantizers2006}. 
    Sharp asymptotics for the fixed-bandwidth Paley--Wiener case were established
    in \cite{thirdpaper}; the present paper treats the broader class of entire functions of exponential type.

Classically, metric entropy has been used primarily to quantify the massiveness 
of sets in function spaces at an order-of-magnitude level and to characterize 
notions such as compactness or minimax rates. In many modern applications, 
however, sharp characterizations of metric entropy are conceptually important, 
as such refinements play a role in concrete quantitative questions. This 
perspective appears, for example, in nonparametric estimation, where refined 
entropy bounds lead to sharper sample-complexity estimates~\cite{Pinskerpaper}. 
In optimal A/D conversion, the required sampling rate at accuracy~$\varepsilon$ 
is governed directly by $H(\varepsilon)$, so improved asymptotics yield more 
accurate bit-rate budgets~\cite{daubechiesConversionImperfectQuantizers2006}. In machine learning, the precision 
of entropy estimates affects the sharpness of lower bounds on the required size of neural 
networks approximating elements of a given function class~\cite{elbrachterDeepNeuralNetwork2021}. 

We conclude this section by defining notation and terminology.
$\mathbb{N}$ and $\mathbb{N}^*$ stand for the set of natural numbers including and, respectively, excluding zero, 
$\mathbb{Z}$ is the set of integers.
$\mathbb{R}$ and $\mathbb{C}$ denote the fields of real and, respectively, complex numbers. We use the generic notation $\mathbb{K}$ to mean that a statement applies both for $\mathbb{K}=\mathbb{R}$ and $\mathbb{K}=\mathbb{C}$.
It will further be convenient to introduce the notation
\begin{equation}
\sigma_{\mathbb{K}}(d) \coloneqq 
\begin{cases}
d, \quad &\text{if } \mathbb{K}=\mathbb{R}, \\
2d, \quad &\text{if } \mathbb{K}=\mathbb{C}.
\end{cases}
\end{equation}
For $d \in \mathbb{N}^*$, we denote the geometric mean of the set $\{\mu_1, \dots, \mu_d\}$ of positive real numbers by
\begin{equation*}
\bar \mu_d 
\coloneqq \prod_{n=1}^d \mu_n^{1/d}.
\end{equation*}
We refer to the closed ball with center $x$, of radius $r$, and with respect to the metric $\rho$ by $B_\rho(x ; r)$.
The subscript $\rho$ will be omitted when there is no room for confusion.
When the metric $\rho$ is induced by the usual $p$-norm, we simply write $\mathcal{B}_p$ for the unit ball $B_\rho(0 ; 1)$.
$\log (\cdot)$ stands for the logarithm to base $2$ and $\ln(\cdot)$ is the natural logarithm.
For $k\in \mathbb{N}^*$, we denote the $k$-fold iterated logarithm by
\begin{equation*}
\log^{(k)}(\cdot) 
\coloneqq \underbrace{\log \circ \cdots \circ \log}_{k \text{ times}} \, (\cdot).
\end{equation*}
Finally, when comparing the asymptotic behavior of the functions $f$ and $g$ as $x \to \ell$, with $\ell \in \mathbb{R}\cup\{-\infty, \infty\}$, we use the notation 
\begin{equation*}
f = o_{x \to \ell}(g) \iff \lim_{x \to \ell} \frac{f(x)}{g(x)} =0
\quad \text{and} \quad 
f = \mathcal{O}_{x \to \ell}(g) \iff \lim_{x \to \ell} \left\lvert \frac{f(x)}{g(x)}\right\rvert \leq C,
\end{equation*}
for some constant $C>0$.
We further indicate asymptotic equivalence according to
\begin{equation*}
f \sim_{x \to \ell} g \iff \lim_{x \to \ell} \frac{f(x)}{g(x)} =1.
\end{equation*}

\section{Metric Entropy of Ellipsoids}\label{sec:meellipsoidssec}

We start with the formal definition of metric entropy.

\begin{definition}[Metric entropy]
Let $(\mathcal{X}, d)$ be a metric space and $\mathcal{K}\subseteq \mathcal{X}$ a compact set. 
An $\varepsilon$-covering of $\mathcal{K}$ with respect to the metric $d$ is a set $\{x_1,...\, ,x_N\} \subseteq \mathcal{X}$  such that for each $x \in \mathcal{K}$, there exists an $i\in \{1, \dots, N\}$ so that $d(x,x_i)\leq \varepsilon$. The $\varepsilon$-covering number $N(\varepsilon ; \mathcal{K}, d)$ is the cardinality of a smallest such $\varepsilon$-covering.
The {metric entropy} of $\mathcal{K}$ is 
\begin{equation*}
H \left(\varepsilon ; \mathcal{K}, d\right)
\coloneqq \log_2 N\left(\varepsilon ; \mathcal{K}, d\right).
\end{equation*}
\end{definition}

\noindent
We shall characterize the metric entropy of infinite-dimensional ellipsoids, 
which we next formally introduce.

%\textcolor{blue}{
\begin{definition}\label{def: Infinite dimensional ellipsoid}
Let $p \in [1, \infty]$.
To a given bounded sequence $\{\mu_n\}_{n \in \mathbb{N}^*}$ of %strictly 
non-negative real numbers, we associate the {ellipsoid norm} $\|\cdot\|_{p, \mu}$ on $\ell^p(\mathbb{N^*} ; \mathbb{K})$, defined by 
\begin{equation*}
\|\cdot\|_{p, \mu} \colon x \in \ell^p(\mathbb{N^*} ; \mathbb{K}) \mapsto 
\begin{dcases}
\left(\sum_{n \in \mathbb{N}^*}  \left\lvert x_n/\mu_n \right\rvert^p\right)^{1/p}, &\text{if } 1 \leq p < \infty, \\[.5cm]
\sup_{n \in \mathbb{N}^*} \, \left\lvert x_n/\mu_n \right\rvert,  &\text{if } p= \infty,
\end{dcases}
\end{equation*}
with the convention that $0/0 = 0$.
The $p$-ellipsoid is the unit ball in $\ell^p(\mathbb{N^*} ; \mathbb{K})$ with respect to this norm %$\|\cdot\|_{p, \mu}$ 
and is denoted by
\begin{equation*}\label{eq: Infinite dimensional ellipsoid definition}
\mathcal{E}_{p}(\{\mu_n\}_{n\in \mathbb{N}^*}) 
\coloneqq 
\left\{ x \in \ell^p(\mathbb{N^*} ; \mathbb{K}) \mid \|x\|_{p, \mu} \leq 1 \right\}.
\end{equation*}

\end{definition}

\noindent
The elements of $\{\mu_n\}_{n \in \mathbb{N}^*}$ are called the {semi-axes} of $\mathcal{E}_{p}(\{\mu_n\}_{n\in \mathbb{N}^*})$.
We simply write $\mathcal{E}_p$ when the choice of semi-axes is clear from the context.
For simplicity of exposition, we assume throughout that the semi-axes $\{\mu_n\}_{n\in \mathbb{N}^*}$ are arranged in non-increasing order, that is, 
$\mu_1 \geq \dots \geq \mu_d \geq \dots \geq  0$.
As announced in the introduction, we restrict ourselves to ellipsoids with semi-axes of exponential decay formalized as follows.

\begin{definition}\label{def:expdecellips}
We say that a smooth (i.e., infinitely differentiable) function
$\psi \colon (0, \infty) \to \mathbb{R}$
is a {decay rate function} with parameter $t^* \geq 0$ if 
\begin{equation}\label{eq:assumptionspsidecayfct}
\psi((t^*, \infty))= (0, \infty)
\quad \text{and} \quad 
t \mapsto \frac{\psi(t)}{t}
\, \text{ is non-decreasing on } (t^*,\infty).
\end{equation}
Further, the ellipsoid $\mathcal{E}_{p}(\{\mu_n\}_{n\in \mathbb{N}^*})$ has {$\psi$-exponentially decaying semi-axes} if 
\begin{equation}\label{eq: that's some big boy exp}
\mu_n
\coloneqq c_0\exp \left\{-\ln (2) \, \psi(n)\right\}, 
\quad \text{for all } n \in \mathbb{N}^* \text{ and some } c_0>0.
\end{equation}
\end{definition}

\noindent
Concrete examples of decay rate functions $\psi$ will be considered in Corollaries~\ref{cor: Complex ellipsoid with exp0}--\ref{cor: Complex ellipsoid with exp2}.
We note that every decay rate function $\psi$ must be  invertible on $(t^*, \infty)$; see Lemma \ref{lem:invertdecaratefct} in Appendix~\ref{sec:proofqqqszracxref}.
%  for a formal statement and proof of this fact.
The inverse
$\psi^{(-1)}\colon (0,\infty) \to (t^*, \infty)$
of $\psi$ plays a central role in our main result Theorem~\ref{thm: scaling metric entropy infinite ellipsoids exp}.
Before stating this result, we need to introduce two auxiliary functions.

\begin{definition}\label{def:deltapsiexpdec}
Let $\psi$ be a decay rate function.
We define the {$\psi$-average function} according to 
\begin{equation*}
\delta(d)
\coloneqq \frac{1}{d}\sum_{n=1}^{d} (\psi(d) -\psi(n)), \quad \text{for all } d \geq 1,
\end{equation*}
and the {$\psi$-difference function} as
\begin{equation*}
\zeta (d)
\coloneqq \psi(d)  - \psi(d-1), \quad \text{for all } d \geq 2.
\end{equation*}
\end{definition}

\noindent
The main properties of the $\psi$-average and the $\psi$-difference function are summarized in Appendix~\ref{sec:proofqqqszracxref}.
The next theorem constitutes the main result of the paper and characterizes the asymptotic behavior of the metric entropy of infinite-dimensional ellipsoids with $\psi$-exponentially decaying semi-axes.

\begin{theorem}\label{thm: scaling metric entropy infinite ellipsoids exp}
Let $p, q \in [1, \infty]$ and let $\psi$ be a decay rate function. 
There exists an integer-valued function $d_{\varepsilon}$ which can be written in the form
\begin{equation}\label{eq:ppppaaatutyyrieie2}
d_{\varepsilon}
=\psi^{(-1)} \left[\log \left(\varepsilon^{-1}\right) + \log (c_0) \right] + O_{\varepsilon \to 0}(1),
\end{equation}
with $c_0$ as per (\ref{eq: that's some big boy exp}),
such that the  metric entropy of the infinite-dimensional ellipsoid $\mathcal{E}_p$ with $\psi$-exponentially decaying semi-axes $\{\mu_n\}_{n \in \mathbb{N}^*}$ satisfies
\begin{equation}\label{eq:ppppaaatutyyrieie}
H \left(\varepsilon ; \mathcal{E}_p, \|\cdot\|_q \right) 
= \, \sigma_{\mathbb{K}}\left(d_\varepsilon\right)  
\left\{\delta\left(d_\varepsilon\right)  
+ \left(\frac{1}{q}-\frac{1}{p}\right) \log \left( \sigma_{\mathbb{K}}\left(d_\varepsilon\right) \right)
 + O_{d_\varepsilon\to \infty}\left(\zeta\left(d_\varepsilon\right)\right) \right\},
\end{equation}
where $\delta(\cdot)$ and $\zeta (\cdot)$ are the $\psi$-average and the $\psi$-difference function, respectively.
\end{theorem}

\begin{proof}[Proof.]
See Appendix~\ref{sec:proof-main}.
\end{proof}

\noindent
% We note that the metric entropy of ellipsoids with exponentially decaying semi-axes does not seem to have been characterized before in the literature. 
The proof of Theorem~\ref{thm: scaling metric entropy infinite ellipsoids exp} 
reduces the problem 
% $\mathcal{E}_{p}(\{\mu_n\}_{n\in \mathbb{N}^*})$  by balls of radius $\varepsilon>0$ 
to that of covering a finite-dimensional ellipsoid of dimension $d_{\varepsilon}$ satisfying \eqref{eq:ppppaaatutyyrieie2}.
This truncation step relies on Lemmata~\ref{lemlkjhgfdaa}, \ref{lemlkjhgfdaa2}, and \ref{lem:lemboundtriangleineqlike}; see also \cite[Eq. (10)]{FISCHER2020105343} for a weaker variant of Lemma~\ref{lem:lemboundtriangleineqlike} that applies to non-necessarily exponentially-decaying semi-axes.
One then applies results on finite-dimensional ellipsoids to this $d_{\varepsilon}$-dimensional ellipsoid to obtain (\ref{eq:ppppaaatutyyrieie});
we discuss finite-dimensional ellipsoids in Appendix~\ref{sec:finite-dim-elli}
and provide a version of such results adapted to our setting in Theorems \ref{thm: metric entropy of finite ellipsoidsLB} and \ref{thm: metric entropy of finite ellipsoids}.
The truncation approach itself---reducing the infinite-dimensional covering problem to a finite-dimensional one---appears in the literature on entropy numbers of
ellipsoids and diagonal operators \cite{FISCHER2020105343,grafSharpAsymptoticsMetric2004}. What is specific to the present exponential-decay setting is the specific
$\varepsilon$-dependent choice of truncation dimension $d_\varepsilon$ provided in \eqref{eq:ppppaaatutyyrieie2}: it must be calibrated so that the truncation error and the
volume bounds of  Theorems \ref{thm: metric entropy of finite ellipsoidsLB} and \ref{thm: metric entropy of finite ellipsoids} contribute at the same asymptotic order, which is what ultimately yields the
sharp expression in Theorem~\ref{thm: scaling metric entropy infinite ellipsoids exp}.
The choice of $d_{\varepsilon}$ is informed by the idea that semi-axes with indices beyond $d_{\varepsilon}$ should be asymptotically negligible at scale $\varepsilon$, i.e., 
$\mu_{d_{\varepsilon}} \sim_{\varepsilon \to 0} \varepsilon$,
which via $\mu_{d_{\varepsilon}} = c_0\exp \left\{-\ln(2) \,  \psi({d_{\varepsilon}})\right\}$ yields (\ref{eq:ppppaaatutyyrieie2}).

For later reference, we next collect consequences of Theorem~\ref{thm: scaling metric entropy infinite ellipsoids exp} for specific decay rate functions $\psi$.
The first one, Corollary~\ref{cor: Complex ellipsoid with exp0}, treats the case of linear functions $\psi$; while Corollaries~\ref{cor: Complex ellipsoid with exp} and \ref{cor: Complex ellipsoid with exp2} pertain to $\psi$-functions growing super-linearly.
When $\psi $ is a linear function, 
% the analogy with binary expansions, studied in Section~\ref{sec:binaryexp}, suggests that 
the effective dimension $d_\varepsilon$ grows logarithmically in $\varepsilon^{-1}$.
Combining this observation with volume arguments (see (\ref{eq:aaakjhvaaa})),
we expect the leading term of the metric entropy expression to scale according to $\log^2 \left(\varepsilon^{-1}\right)$.
The next corollary formalizes this insight.

\begin{corollary}\label{cor: Complex ellipsoid with exp0}
    Let $p, q \in [1 , \infty]$, $c > 0$, and
        \begin{equation*}
        \psi \colon t \in (0, \infty) \mapsto c\, t \in \mathbb{R}.
    \end{equation*}
    The metric entropy of the infinite-dimensional ellipsoid $\mathcal{E}_{p}$ with $\psi$-exponentially decaying semi-axes satisfies
    \begin{equation*}
        H \left(\varepsilon ; \mathcal{E}_{p}, \|\cdot\|_q \right)
         = \frac{\alpha}{2c} \log^2 \left(\varepsilon^{-1}\right)  + \frac{\alpha}{c} \left(\frac{1}{q}-\frac{1}{p}\right) \log \left(\varepsilon^{-1}\right) \log^{(2)} \left(\varepsilon^{-1}\right) + O_{\varepsilon\to 0}\left(\log \left(\varepsilon^{-1}\right)\right),
    \end{equation*}
    where $\alpha=1$ if $\mathbb{K}=\mathbb{R}$ and $\alpha=2$ if $\mathbb{K}=\mathbb{C}$.
\end{corollary}

\begin{proof}[Proof.]
See Appendix~\ref{sec:proofscor67}.
\end{proof}

\noindent
We now turn to decay rate functions of super-linear growth, whose corresponding metric entropy asymptotics will be expressed in terms of the Lambert $W$-function, the main properties of which are recalled in Appendix~\ref{sec:WLambertappdix}.

\begin{corollary}\label{cor: Complex ellipsoid with exp}
Let $p, q \in [1, \infty]$, $c,c' > 0$, and
\begin{equation*}
    \psi \colon t \in [1, \infty) \mapsto c\, t \, (\log (t)-c')\in \mathbb{R}.
\end{equation*}
The metric entropy of the infinite-dimensional ellipsoid $\mathcal{E}_{p}$ with $\psi$-exponentially decaying semi-axes satisfies
\begin{align*}
H \left(\varepsilon ; \mathcal{E}_{p}, \|\cdot\|_q \right)
=&\, \frac{\alpha \, c \, 2^{2c'-1}}{\ln(2)}  \exp\left\{2 \beta(\varepsilon) \right\} \left(\beta(\varepsilon)+ \frac{1}{2}\right) \left(1
+ O_{\beta(\varepsilon)\to \infty}\left(\exp\left\{- \beta(\varepsilon) \right\}\right)\right), 
\end{align*}
where $\alpha=1$ if $\mathbb{K}=\mathbb{R}$, $\alpha=2$ if $\mathbb{K}=\mathbb{C}$, and
\begin{equation*}
\beta(\varepsilon) 
\coloneqq W\left( \frac{\ln \left(\varepsilon^{-1}\right)}{2^{c'}c}\right),
\end{equation*}
with $W(\cdot)$ denoting the Lambert $W$-function.
\end{corollary}

\begin{proof}[Proof.]
See Appendix~\ref{sec:proofcor89}.
\end{proof}

\noindent
The next result exploits properties of the Lambert $W$-function to find the explicit asymptotic behavior of the metric entropy in Corollary~\ref{cor: Complex ellipsoid with exp} up to second order.
We remark that inspection of the proof reveals that it is actually possible to obtain the asymptotic behavior up to arbitrary order. Concretely, this can be done by continuing the asymptotic expansion of $\psi^{(-1)}$ in Lemma \ref{lem:lkjodihuzehchehjz} using more terms from Lemma \ref{lem:lbtseriesexp}.
For clarity of exposition, however, we decided to limit the formal statement of the result to second order.

\begin{corollary}\label{cor: Complex ellipsoid with exp2}
Let $p, q \in [1, \infty]$, $c,c' > 0$, and
\begin{equation*}
    \psi \colon t \in [1, \infty) \mapsto c\, t \, (\log (t)-c')\in \mathbb{R}.
\end{equation*}
The metric entropy of the infinite-dimensional ellipsoid  $\mathcal{E}_{p}$ with $\psi$-exponentially decaying semi-axes satisfies
\begin{align*}
H \left(\varepsilon ; \mathcal{E}_{p}, \|\cdot\|_q \right)
= \frac{\alpha \, \log^{2}\left(\varepsilon^{-1}\right)}{2\, c\, {\log^{(2)} \left(\varepsilon^{-1}\right)}} 
\left(1+ \frac{\log^{(3)} \left(\varepsilon^{-1}\right)}{\log^{(2)} \left(\varepsilon^{-1}\right)}
+ o_{\varepsilon\to 0 }\left(\frac{\log^{(3)} \left(\varepsilon^{-1}\right)}{\log^{(2)} \left(\varepsilon^{-1}\right)}\right)\right).
\end{align*}
\end{corollary}

\begin{proof}[Proof.]
See Appendix~\ref{sec:proofcor89}.
\end{proof}

\noindent
We conclude by pointing out that, both in Corollary~\ref{cor: Complex ellipsoid with exp} and Corollary~\ref{cor: Complex ellipsoid with exp2}, the 
asymptotic behavior of metric entropy does not depend on
the parameters $p$ and $q$.

\section{Applications to Complex Analytic Functions}\label{sec:complexanalyticapplications}

% We now turn to the application of our general results to classes of analytic functions whose asymptotic metric entropy behavior has
% been characterized before in the literature.
% In each of these cases, we improve upon the best known results, in some cases significantly so.
% All of these best known results are based on evaluating the cardinalities of explicitly constructed coverings and packings. The corresponding
% proofs are hence often tedious and highly specific to the function class under consideration.

    The three function classes considered in this section (periodic functions analytic on a strip, analytic functions bounded on a disk, and entire functions of exponential type)
    are central objects in harmonic analysis, complex analysis, and sampling theory. 
    Their asymptotic metric entropy behavior has been
    characterized before in the literature by evaluating the cardinalities of explicitly constructed coverings and
    packings, 
    leading to proofs that are invariably tedious and highly specific.
    We show that these results are covered by our general theory and improve upon all of them.

\subsection{Periodic Functions Analytic on a Strip}\label{sec:analyticstripperiodic}

We first consider a class of functions that are periodic on the real line and can be analytically continued to a strip in the complex plane.
The formal definition is as follows.

\begin{definition}[Periodic functions analytic on a strip]\label{def:perfctonstrip}
Let $M$ and $s$ be positive real numbers.
We denote by $\mathcal{A}_s(M)$ the class of functions 
$f \colon \mathbb{R} \to \mathbb{C}$
which are $2\pi$-periodic and can be analytically continued to the domain 
\begin{equation*}
\mathcal{S}
\coloneqq \left\{z=x+iy\in \mathbb{C} \mid x\in \mathbb{R} \text{ and } \lvert y \rvert < s \right\},
\end{equation*}
such that
\begin{equation*}
\sup_{z\in \mathcal{S}} \,  \lvert f(z) \rvert \leq M.
\end{equation*}
\end{definition}

\noindent
Note that, by the identity theorem (see e.g. \cite[Corollary 10.18]{rudinRealComplexAnalysis1987}), $f(z)$ in Definition \ref{def:perfctonstrip} is
$2\pi$-periodic on the entire strip $\mathcal{S}$, i.e., $f(z+2\pi)=f(z), z \in \mathcal{S}$.

The metric entropy of the class $\mathcal{A}_s(M)$ endowed with
the metric
\begin{equation*}
d_{2 \pi} \colon (f_1, f_2) \longmapsto \sup_{x \in [0, 2\pi]} \, \lvert f_1(x) - f_2(x) \rvert
\end{equation*}
was characterized in \cite[Chapter 7, Section 2.4]{shiryayevSelectedWorksKolmogorov1993} according to
\begin{equation}\label{eq:pokjhbc}
H \left(\varepsilon ; \mathcal{A}_s(M), d_{2 \pi}\right)
= \frac{2 \ln (2)}{s} \log^2\left(\varepsilon^{-1}\right)  + O_{\varepsilon\to 0}\left(\log \left(\varepsilon^{-1}\right) \log^{(2)} \left(\varepsilon^{-1}\right)\right).
\end{equation}
The standard technique for deriving this result is to construct covering and packing elements via an adequate choice of Fourier series coefficients.
We next demonstrate that our general approach improves upon (\ref{eq:pokjhbc}) by providing a more precise characterization of the $O_{\varepsilon\to 0}(\cdot)$ term.

\begin{theorem}\label{thm:analyticstrip}
Let $M$ and $s$ be positive real numbers.
The metric entropy of the class $\mathcal{A}_s(M)$ equipped with the metric $d_{2 \pi}$ can be expressed as
\begin{equation*}
H \left(\varepsilon ; \mathcal{A}_s(M), d_{2 \pi}\right)
= \frac{2 \ln (2)}{s} \left[\log^2\left(\varepsilon^{-1}\right)  + \log \left(\varepsilon^{-1}\right) \log^{(2)} \left(\varepsilon^{-1}\right)\left(\gamma(\varepsilon) + o_{\varepsilon\to 0}\left(1 \right)\right)\right] ,
\end{equation*}
with $\gamma(\cdot)$ a function satisfying $\lvert\gamma(\varepsilon)\rvert \leq 1$, for all $\varepsilon >0$.
\end{theorem}

\noindent
Theorem~\ref{thm:analyticstrip} improves upon (\ref{eq:pokjhbc}) as follows. While (\ref{eq:pokjhbc}) states that there exists $K>0$, possibly depending on $s$ and $M$, such that
\begin{equation*}
\lim_{\varepsilon \to 0} \, \left\lvert\frac{\frac{s}{2 \ln (2)}H \left(\varepsilon ; \mathcal{A}_s(M), d_{2 \pi}\right) - \log^2\left(\varepsilon^{-1}\right)}{\log \left(\varepsilon^{-1}\right) \log^{(2)}\left(\varepsilon^{-1}\right)}\right\rvert
\leq K,
\end{equation*}
Theorem~\ref{thm:analyticstrip} establishes that the constant $K$ can be taken to be equal to $1$, independently of $s$ and $M$.

\begin{proof}[Proof.]
The elements of $\mathcal{A}_s(M)$ are $2\pi$-periodic functions and can hence be represented in terms of Fourier series according to 
\begin{equation*}
f(x)
= \sum_{k=-\infty}^{\infty} a_k e^{ikx}, 
\quad \text{with } 
a_k \coloneqq \frac{1}{2\pi}  \int_{0}^{2\pi} f(x) \, e^{-ikx} \, dx,
\quad \text{for all } k \in \mathbb{Z}.
\end{equation*}
We shall use the shorthand $a \coloneqq \{a_k\}_{k \in \mathbb{Z}}$.
Following our three-step procedure, we first establish that the body formed by the Fourier series coefficients is of ellipsoidal structure, more specifically it can be inscribed and circumscribed with ellipsoids. %we shall say sandwiched by ellipsoids.
There is a minor technical adjustment we need to make before we can proceed.
Specifically, so far we only considered one-sided ellipsoids, that is, ellipsoids indexed by $\mathbb{N}^*$ rather than $\mathbb{Z}$.
The Fourier series coefficients $\{{a_k}\}_{k \in \mathbb{Z}}$ are, however, two-sided sequences.
To consolidate this matter, we define the coefficients
\begin{equation}\label{eq:deftildea}
\tilde a_{1} \coloneqq  a_0, \quad
\tilde a_{2 n} \coloneqq  a_n, 
\quad \text{and} \quad 
\tilde a_{2 n + 1} \coloneqq  a_{-n},
\quad \text{for all } n\in \mathbb{N}^*,
\end{equation}
together with the map
\begin{equation*}
\iota \colon f \in \mathcal{A}_s(M) \mapsto \left\{ \tilde a_{n}\right\}_{n \in \mathbb{N}^*}.
\end{equation*}
The ellipsoidal structure of the body formed by the Fourier series coefficients is formalized in the following lemma.

\begin{lemma}[Ellipsoidal structure of $\iota(\mathcal{A}_s(M))$]\label{lem:lemmainclusionofellipsoidsfouryeah}
Let $M$ and $s$ be positive real numbers and set 
\begin{equation}\label{eq:defmu1and2expfour}
\mu_{n}^{(1)} 
\coloneqq M e^{-\frac{s n}{2}}
\quad \text{and} \quad 
\mu_{n}^{(2)} 
\coloneqq \sqrt{2} M e^{-\frac{s (n-1)}{2}},
\quad \text{for all } n\in \mathbb{N}^*.
\end{equation}
Then, we have
\begin{equation*}
\mathcal{E}_{1}\left(\left\{\mu_{n}^{(1)}\right\}_{n \in\mathbb{N}^*}\right)
\subseteq \iota(\mathcal{A}_s(M))
\subseteq \mathcal{E}_{2}\left(\left\{\mu_{n}^{(2)}\right\}_{n \in\mathbb{N}^*}\right).
\end{equation*}
\end{lemma}

\begin{proof}[Proof.]
See Appendix~\ref{sec:prooflemmainclusionofellipsoidsfouryeah}.
\end{proof}

\noindent
We next relate the metric $d_{2 \pi}$ to $\ell^q$-metrics in sequence spaces.

\begin{lemma}\label{lem:poiuhgfyyyyyyy}
Let $f_1$ and $f_2$ be $2\pi$-periodic functions with Fourier series coefficients $a^{(1)}$ and $a^{(2)}$, respectively, such that
\begin{equation*}
f_1(x) = \sum_{k=-\infty}^{\infty} a^{(1)}_k \, e^{ikx}
\quad \text{and} \quad 
f_2(x) =  \sum_{k=-\infty}^{\infty} a^{(2)}_k \, e^{ikx} .
\end{equation*}
Then, we have 
\begin{equation*}
\left\| a^{(1)}-a^{(2)}\right\|_{\ell^2(\mathbb{Z})}
\leq d_{2 \pi}(f_1, f_2)
\leq \left\| a^{(1)}-a^{(2)} \right\|_{\ell^1(\mathbb{Z})}.
\end{equation*}
\end{lemma}

\begin{proof}[Proof.]
See Appendix~\ref{sec:prooffouriercoefs}.
\end{proof}

\noindent
Combining Lemma \ref{lem:lemmainclusionofellipsoidsfouryeah} with Lemma \ref{lem:poiuhgfyyyyyyy}, we obtain
\begin{align}
    H \left(\varepsilon ; \mathcal{E}_{1}\left(\left\{\mu_{n}^{(1)}\right\}_{n \in\mathbb{N}^*}\right), \|\cdot\|_2 \right)
    &\leq H \left(\varepsilon ; \mathcal{A}_s(M), d_{2 \pi} \right) \label{eq: ljsjgirlssdnf1bist}\\
    &\leq H \left(\varepsilon ; \mathcal{E}_{2}\left(\left\{\mu_{n}^{(2)}\right\}_{n \in\mathbb{N}^*}\right), \|\cdot\|_1 \right), \label{eq: ljsjgirlssdnf2bist}
\end{align}
where $\{\mu_{n}^{(1)}\}_{n \in\mathbb{N}^*}$ and $\{\mu_{n}^{(2)}\}_{n \in\mathbb{N}^*}$ are according to (\ref{eq:defmu1and2expfour}).
We now wish to apply Corollary~\ref{cor: Complex ellipsoid with exp0} as a proxy for the second and third steps of our procedure.
To this end, first observe that, with $\psi(t)= \frac{s}{2 \ln (2)} t$, we have
\begin{equation*}
\mu_n^{(1)} = M \exp \left\{-\ln (2) \, \psi(n)\right\}
\quad \text{and} \quad 
\mu_n^{(2)} = \sqrt{2} M e^{s/2}  \exp \left\{-\ln (2) \, \psi(n)\right\},
\end{equation*}
for all $n \in \mathbb{N}^*$, showing that the semi-axes $\{\mu_n^{(1)}\}_{n\in\mathbb{N}^*}$ and $\{\mu_n^{(2)}\}_{n\in\mathbb{N}^*}$ are $\psi$-exponentially decaying.
We are now in a position to apply Corollary~\ref{cor: Complex ellipsoid with exp0} with $c=\frac{s}{2 \ln (2)}$ and $\alpha=2$ to get
\begin{align}
&H \left(\varepsilon ; \mathcal{E}_{2}\left(\left\{\mu_{n}^{(2)}\right\}_{n \in\mathbb{N}^*}\right), \|\cdot\|_1 \right)  \label{eq: htjfkldoifguirhefdkjl21bist} \\
&= \,  \frac{2 \ln (2)}{s} \log^2\left(\varepsilon^{-1}\right)  +  \frac{2 \ln (2)}{s}  \log \left(\varepsilon^{-1}\right) \log^{(2)} \left(\varepsilon^{-1}\right) \left(1+ o_{\varepsilon\to 0}\left(1 \right)\right) \label{eq: htjfkldoifguirhefdkjl22bist}
\end{align}
and 
\begin{align}
&H \left(\varepsilon ; \mathcal{E}_{1}\left(\left\{\mu_{n}^{(1)}\right\}_{n \in\mathbb{N}^*}\right), \|\cdot\|_2 \right) \label{eq: htjfkldoifguirhefdkjl31bist}\\
&= \,  \frac{2 \ln (2)}{s} \log^2\left(\varepsilon^{-1}\right)  +  \frac{2 \ln (2)}{s}  \log \left(\varepsilon^{-1}\right) \log^{(2)} \left(\varepsilon^{-1}\right) \left(-1+ o_{\varepsilon\to 0}\left(1 \right)\right).\label{eq: htjfkldoifguirhefdkjl32bist}
\end{align}
Inserting (\ref{eq: htjfkldoifguirhefdkjl21bist})--(\ref{eq: htjfkldoifguirhefdkjl22bist}) and (\ref{eq: htjfkldoifguirhefdkjl31bist})--(\ref{eq: htjfkldoifguirhefdkjl32bist}) in (\ref{eq: ljsjgirlssdnf1bist})--(\ref{eq: ljsjgirlssdnf2bist}) yields the desired result
\begin{equation*}
H \left(\varepsilon ; \mathcal{A}_s(M), d_{2 \pi}\right)
= \frac{2 \ln (2)}{s} \left[\log^2\left(\varepsilon^{-1}\right)  + \log \left(\varepsilon^{-1}\right) \log^{(2)} \left(\varepsilon^{-1}\right)\left(\gamma(\varepsilon) + o_{\varepsilon\to 0}\left(1 \right)\right)\right] ,
\end{equation*}
with $\lvert \gamma(\varepsilon)\rvert \leq 1$, for all $\varepsilon >0$.
\end{proof}

\noindent
We conclude by noting that thresholding infinite-dimensional ellipsoids in the Fourier series sequence space and covering through the resulting finite-dimensional ellipsoids implicitly amounts to covering the class $\mathcal{A}_s(M)$ by finite Fourier series expansions, i.e., by trigonometric polynomials.

\subsection{Functions Bounded on a Disk}\label{sec: ana 2 bounded disk}

We next consider functions that are analytic and bounded on a disk, formally defined as follows.

\begin{definition}[Analytic functions bounded on a disk]
Let $M$ and $r'$ be positive real numbers. 
We denote by $\mathcal{A}(r' ; M)$ the class of functions $f$ that are analytic on the open disk $D(0 ; r')$ in the complex plane centered at $0$ and of radius $r'$, and satisfy
\begin{equation*}
\sup_{z\in D(0 ; r')} \lvert f(z) \rvert \leq M.
\end{equation*}
\end{definition}

\noindent
Based on our three-step procedure, we now characterize the asymptotic behavior of the metric entropy of the class $\mathcal{A}(r' ; M)$ endowed with the metric
\begin{equation}\label{eq: fnbhgvgvfd}
 d_{r} \colon (f_1, f_2) \longmapsto \sup_{z \in D(0;r)} \, \lvert f_1(z)- f_2(z)\rvert,
\end{equation}
where we will assume throughout that $0<r<r'$.

\noindent

The metric entropy of this function class is of relevance, inter alia, in control theory and signal processing \cite{zamesMetricComplexityCausal1977,zamesNoteMetricDimension1993} as well as in neural network theory \cite{hutterMetricEntropyLimits2022},
with the best-known result due to 
Vitu\v{s}kin~\cite[Chapter 10, Theorem 4]{lorentzApproximationFunctions1966} given by
\begin{equation}\label{eq:pokjhbc211}
    H \left(\varepsilon ; \mathcal{A}(r' ; M), d_{ r}\right)
    = \frac{\log^2\left(\varepsilon^{-1}\right)}{\log(r'/r)}   + \mathcal{O}_{\varepsilon\to 0}\left(\log \left(\varepsilon^{-1}\right) \log^{(2)} \left(\varepsilon^{-1}\right) \right).
\end{equation}
The standard technique for deriving this result is based on explicit constructions of coverings and packings via adequate choices of Taylor series coefficients of functions in $\mathcal{A}(r' ; M)$.
Again, we shall show that our general approach improves upon (\ref{eq:pokjhbc211}) by providing a more precise characterization of the second-order term,
and does so without resorting to the explicit construction of coverings and packings.

\begin{theorem}\label{thm: second look analytic bounded}
Let $r$, $r'$, and $M$ be positive real numbers such that $r'>r$.
The metric entropy of the class $\mathcal{A}(r' ; M)$ equipped with the metric $d_{r}$ can be expressed as 
\begin{equation*}
H \left(\varepsilon ; \mathcal{A}(r' ; M), d_{r}\right)
= \frac{\log^2\left(\varepsilon^{-1}\right)}{\log (r'/r)} + \frac{\log \left(\varepsilon^{-1}\right)}{\log (r'/r)}  \log^{(2)} \left(\varepsilon^{-1}\right) \left(\gamma(\varepsilon) + o_{\varepsilon\to 0}\left(1 \right)\right),
\end{equation*}
with $\gamma(\cdot)$ a function satisfying $\lvert\gamma(\varepsilon)\rvert \leq 1$, for all $\varepsilon >0$.
\end{theorem}

\noindent
Theorem~\ref{thm: second look analytic bounded} improves upon (\ref{eq:pokjhbc211}) as follows. While (\ref{eq:pokjhbc211}) states that there exists $K > 0$, possibly depending on $r, r'$, and $M$, such that
\begin{equation*}
\lim_{\varepsilon \to 0} \,  \left\lvert\frac{\log (r'/r)\, H \left(\varepsilon ; \mathcal{A}(r' ; M), d_{r}\right) - \log^2\left(\varepsilon^{-1}\right)}{\log \left(\varepsilon^{-1}\right) \log^{(2)}\left(\varepsilon^{-1}\right)}\right\rvert
\leq K,
\end{equation*}
Theorem~\ref{thm: second look analytic bounded} establishes that the constant $K$ can be taken to be equal to $1$, independently of $r, r'$, and $M$.

\begin{proof}[Proof.]
Let $\{a_k\}_{k\in\mathbb{N}}$ be the sequence of Taylor series coefficients of $f$, which exists by analyticity, i.e., 
\begin{equation*}
f(z)
=\sum_{k=0}^\infty a_k z^k,
\quad \text{for all } z\in D(0 ; r').
\end{equation*}
We start by defining the embedding
\begin{equation}\label{eq:defiotatheembedd}
\iota \colon f \in \mathcal{A}(r' ; M) \mapsto \{\tilde a_k\}_{k\in\mathbb{N}} \in \ell^{\infty}(\mathbb{N}),
\quad \text{with } \tilde a_k \coloneqq a_k r^k,
\end{equation}
where $\{\tilde a_k\}_{k\in\mathbb{N}}\in \ell^{\infty}(\mathbb{N})$ is a direct consequence of Cauchy's estimate (\ref{eq: bound an coefs}).
As in the proof of Theorem~\ref{thm:analyticstrip}, we next inscribe and circumscribe $\iota(\mathcal{A}(r' ; M))$ with ellipsoids.

\begin{lemma}[Ellipsoidal structure of $\mathcal{A}(r' ; M)$]\label{lem:Comparison with ellipsoidbddiskup}
Let $r$, $r'$, and $M$ be positive real numbers such that $r'>r$, and let 
\begin{equation}\label{eq:defsemaxespokji}
\mu_k \coloneqq M \exp\left\{-k \ln (r'/r) \right\},
\quad \text{for all } k\in \mathbb{N}.
\end{equation}
Then, we have
\begin{equation*}
\mathcal{E}_{1}(\{\mu_k\}_{k\in\mathbb{N}})
\subseteq \iota \left(\mathcal{A}(r' ; M) \right) 
\subseteq \mathcal{E}_{2}(\{\mu_k\}_{k\in\mathbb{N}}).
\end{equation*}
\end{lemma}

\begin{proof}[Proof.]
See Appendix~\ref{sec:proofComparison with ellipsoidbddiskup}.
\end{proof}

\noindent
We next relate the metric $ d_{r}$ to $\ell^q$-metrics in sequence spaces.

\begin{lemma}\label{lem: Bounds on analytic norm}
Let $r$ be a positive real number and let $f_1, f_2$ be functions that are analytic on $D(0; r)$, with respective Taylor series expansions
\begin{equation*}
f_1(z) = \sum_{k=0}^\infty a_{k}^{(1)} z ^{k}
\quad \text{and} \quad 
f_2(z) = \sum_{k=0}^\infty a_{k}^{(2)} z ^{k}.
\end{equation*}
Then, we have
\begin{equation}\label{eq: mlkjhgfcz}
\left\|\tilde a^{(1)} - \tilde a^{(2)}\right\|_{\ell^2}
\leq d_{r} (f_1,f_2)
\leq\left\|\tilde a^{(1)}-\tilde a^{(2)}\right\|_{\ell^1},
\end{equation}
where 
\begin{equation*}
\tilde a^{(1)} \coloneqq \left\{a_{k}^{(1)} r^{k}\right\}_{k \in \mathbb{N}}
\quad \text{and} \quad 
\tilde a^{(2)} \coloneqq \left\{a_{k}^{(2)} r^{k}\right\}_{k \in \mathbb{N}}.
\end{equation*}
\end{lemma}

\begin{proof}[Proof.]
See Appendix~\ref{sec:prooftaylorcoefs}.
\end{proof}

\noindent
Combining Lemma \ref{lem:Comparison with ellipsoidbddiskup} with Lemma \ref{lem: Bounds on analytic norm}, we obtain
\begin{equation}\label{eq: htjfkldoifguirhefdkjl}
H \left(\varepsilon ; \mathcal{E}_1(\{\mu_k\}_{k\in\mathbb{N}}), \|\cdot\|_2 \right)
\leq H \left(\varepsilon ; \mathcal{A}(r' ; M), d_{ r}\right)
\leq H \left(\varepsilon ; \mathcal{E}_2(\{\mu_k\}_{k\in\mathbb{N}}), \|\cdot\|_1 \right),
\end{equation}
where $\{\mu_k\}_{k\in\mathbb{N}}$ is as per (\ref{eq:defsemaxespokji}).
We can now apply Corollary~\ref{cor: Complex ellipsoid with exp0} with $c= \log (r'/r)$ to obtain
\begin{align}
H \left(\varepsilon ; \mathcal{E}_2(\{\mu_k\}_{k\in\mathbb{N}}), \|\cdot\|_1 \right)
=& \,  \frac{\log^2\left(\varepsilon^{-1}\right)}{\log (r'/r)}  \label{eq: htjfkldoifguirhefdkjl21}\\
& \, +  \frac{2}{\log (r'/r)} \log \left(\varepsilon^{-1}\right) \log^{(2)} \left(\varepsilon^{-1}\right) \left(\frac{1}{2}+ o_{\varepsilon\to 0}\left(1 \right)\right) \label{eq: htjfkldoifguirhefdkjl22}
\end{align}
and 
\begin{align}
H \left(\varepsilon ; \mathcal{E}_1(\{\mu_k\}_{k\in\mathbb{N}}), \|\cdot\|_2 \right)
=& \,  \frac{\log^2\left(\varepsilon^{-1}\right)}{\log (r'/r)}  \label{eq: htjfkldoifguirhefdkjl31}\\
& \, +  \frac{2}{\log (r'/r)} \log \left(\varepsilon^{-1}\right) \log^{(2)} \left(\varepsilon^{-1}\right) \left(-\frac{1}{2}+ o_{\varepsilon\to 0}\left(1 \right)\right).\label{eq: htjfkldoifguirhefdkjl32}
\end{align}
Inserting (\ref{eq: htjfkldoifguirhefdkjl21})--(\ref{eq: htjfkldoifguirhefdkjl22}) and (\ref{eq: htjfkldoifguirhefdkjl31})--(\ref{eq: htjfkldoifguirhefdkjl32}) in (\ref{eq: htjfkldoifguirhefdkjl}) yields the desired result
\begin{equation*}
H \left(\varepsilon ; \mathcal{A}(r' ; M), d_{r}\right)
= \frac{\log^2\left(\varepsilon^{-1}\right)}{\log (r'/r)} + \frac{\log \left(\varepsilon^{-1}\right)}{\log (r'/r)}  \log^{(2)} \left(\varepsilon^{-1}\right) \left(\gamma(\varepsilon) + o_{\varepsilon\to 0}\left(1 \right)\right) ,
\end{equation*}
with $\lvert \gamma(\varepsilon)\rvert \leq 1$, for all $\varepsilon >0$.
\end{proof}

\noindent
We conclude by noting that thresholding infinite-dimensional ellipsoids in the Taylor series sequence space and covering through the resulting finite-dimensional ellipsoids implicitly amounts to covering the class $ \mathcal{A}(r' ; M)$ by finite Taylor series expansions, i.e., by polynomials, in a complex variable.

\subsection{Functions of Exponential Type}\label{sec: ana 2 exp type}

We finally consider functions of exponential type \cite[Chapter 19]{rudinRealComplexAnalysis1987} defined as follows.

\begin{definition}[Functions of exponential type]\label{def: exp type analytic}
    An entire function $f$ is said to be of {exponential type} if there exist positive real constants $A$ and $C$ such that, for all $z\in \mathbb{C}$, 
    \begin{equation}\label{eq: popoloihj}
        \lvert f(z) \rvert \leq C e^{A\lvert z \rvert}.
    \end{equation}
    We write $\mathcal{F}_{\text{exp}}^{A,C}$ for the class of entire functions of exponential type with constants $A$ and $C$.
\end{definition}

\noindent
Functions of exponential type appear naturally in many practical applications, mainly as they can be identified, through the Paley--Wiener theorem \cite[Theorem 19.3]{rudinRealComplexAnalysis1987}, with band-limited functions.
For instance, in \cite{daubechiesConversionImperfectQuantizers2006} the metric entropy (rate) of band-limited signals plays a fundamental role in assessing the ultimate performance limits of
analog-to-digital converters.

The best known result on the metric entropy of $\mathcal{F}_{\text{exp}}^{A,C}$ is \cite[Chapter 7, Theorem XX]{shiryayevSelectedWorksKolmogorov1993}
\begin{equation}\label{eq:prevratekolmexptype}
H \left(\varepsilon ; \mathcal{F}_{\text{exp}}^{A,C}, d_{1} \right)
\sim_{\varepsilon \to 0} \frac{\log^{2}\left(\varepsilon^{-1}\right)}{{\log^{(2)} \left(\varepsilon^{-1}\right)}},
\end{equation}
where $d_1$ is the metric defined in (\ref{eq: fnbhgvgvfd}) under the choice $r=1$.
We improve upon 
\eqref{eq:prevratekolmexptype} by providing a compact functional expression for the asymptotic behavior of 
$H(\varepsilon;\mathcal{F}^{A,C}_{\exp},d_{1})$, from which not only the 
second-order term but, if desired, all higher-order terms can be obtained in a systematic manner. 
Our main structural bounds are given in the following theorem.
%We improve upon (\ref{eq:prevratekolmexptype}) by providing a compact functional expression for the asymptotic behavior of $H \left(\varepsilon ; %\mathcal{F}_{\text{exp}}^{A,C}, d_{1} \right)$ as follows.

\begin{theorem}\label{thm:exptypecompanlyticfctme}
Let $A$ and $C$ be positive real constants.
The metric entropy of the class $\mathcal{F}_{\text{exp}}^{A,C}$ (as per Definition~\ref{def: exp type analytic})
equipped with the metric $d_{1}$ satisfies
\begin{align*}
    & \left(\frac{e A}{2} \right)^2\,\frac{\exp\left\{2 \beta_1(\varepsilon) \right\}}{\ln(2)}  \left[\beta_1(\varepsilon)+ \frac{1}{2}\right] \left[1
+ O_{\beta_1(\varepsilon)\to \infty}\left(\exp\left\{- \beta_1(\varepsilon) \right\}\right)\right] \\
    & \leq  H \left(\varepsilon ; \mathcal{F}_{\text{exp}}^{A,C}, d_{1} \right)  \\
    & \leq \left(e  A\right)^2\,\frac{\exp\left\{2 \beta_2(\varepsilon) \right\}}{\ln(2)} \left[\beta_2(\varepsilon) + \frac{1}{2}\right] \left[1
+ O_{\beta_2(\varepsilon)\to \infty}\left(\exp\left\{- \beta_2(\varepsilon) \right\}\right)\right],
\end{align*}
where 
\begin{equation*}
\beta_1(\varepsilon) 
\coloneqq  W\left( \frac{2 \ln \left(\varepsilon^{-1}\right)}{e \,  A}\right)
\quad \text{and} \quad
\beta_2(\varepsilon) 
\coloneqq  W\left( \frac{\ln \left(\varepsilon^{-1}\right)}{e \,   A}\right),
\end{equation*}
with $W$ denoting the Lambert $W$-function (cf. Appendix~\ref{sec:WLambertappdix}).
\end{theorem}

\begin{proof}[Proof.]
See Appendix~\ref{sec:proofthm14cor15ss}.
\end{proof}

\noindent
It is not immediately clear from Theorem~\ref{thm:exptypecompanlyticfctme} 
that the characterization it provides constitutes an improvement over the best 
known bound~(\ref{eq:prevratekolmexptype}). We therefore specialize the bounds of 
Theorem~\ref{thm:exptypecompanlyticfctme} via 
Corollary~\ref{cor: Complex ellipsoid with exp2} to make the second-order term in 
the asymptotic expansion of the metric entropy of 
$\mathcal{F}_{\text{exp}}^{A,C}$ explicit.

\iffalse 
It is not immediate that the characterization provided in Theorem~\ref{thm:exptypecompanlyticfctme} indeed constitutes an improvement over the best known result (\ref{eq:prevratekolmexptype}).
We therefore specialize the bounds of Theorem~\ref{thm:exptypecompanlyticfctme}
via Corollary~\ref{cor: Complex ellipsoid with exp2} to make the second-order term in the asymptotic expansion of the metric entropy of $\mathcal{F}_{\text{exp}}^{A,C}$ explicit.
\fi 
%, we now specialize the bounds of Theorem~\ref{thm:exptypecompanlyticfctme}
%via Corollary~\ref{cor: Complex ellipsoid with exp2}.
%to make the second order term in the asymptotic expansion of the metric entropy of $\mathcal{F}_{\text{exp}}^{A,C}$ explicit.

\begin{corollary}\label{cor:exptypecompanlyticfctme}
Let $A$ and $C$ be positive real constants.
The metric entropy of the class $\mathcal{F}_{\text{exp}}^{A,C}$ (as per Definition~\ref{def: exp type analytic}) equipped with the metric $d_{1}$ satisfies
        \begin{equation*}
        H \left(\varepsilon ; \mathcal{F}_{\text{exp}}^{A,C}, d_{1} \right)
= \frac{\log^{2}\left(\varepsilon^{-1}\right)}{{\log^{(2)} \left(\varepsilon^{-1}\right)}} \left(1+ \frac{\log^{(3)} \left(\varepsilon^{-1}\right)}{\log^{(2)} \left(\varepsilon^{-1}\right)} + o_{\varepsilon\to 0}\left(\frac{\log^{(3)} \left(\varepsilon^{-1}\right)}{\log^{(2)} \left(\varepsilon^{-1}\right)}\right)\right) .
    \end{equation*}
\end{corollary}

\begin{proof}[Proof.]
See Appendix~\ref{sec:proofthm14cor15ss}.
\end{proof}

\paragraph{Acknowledgments.}
The authors would like to thank V. Abadie for drawing their attention to the Lambert $W$-function.

% Print Bibliography
\bibliography{main}

@article{secondpaper,
  author = {Allard, Thomas and Bölcskei, Helmut},
  title = {Metric Entropy of Ellipsoids in {Banach} Spaces: Techniques and Precise Asymptotics}, 
  journal={Available online: https://arxiv.org/abs/2504.18321},
  volume={},
  number={},
  pages={},
  year={2025},
  url={},
  publisher={}
}

@article{thirdpaper,
title = {Entropy of compact operators with applications to {Landau-Pollak-Slepian} theory and {Sobolev} spaces},
journal = {Applied and Computational Harmonic Analysis},
volume = {77},
number = {101762},
year = {2025},
author = {Thomas Allard and Helmut Bölcskei},
}

@article{Pinskerpaper,
  author = {Allard, Thomas},
  title = {Metric Entropy and Minimax Risk of Ellipsoids with an Application to {Pinsker}'s Theorem}, 
  journal={Available online: https://arxiv.org/abs/2510.22441},
  volume={},
  number={},
  pages={},
  year={2025},
  url={},
  publisher={}
}

@article{carl1981entropy,
  title={Entropy numbers of diagonal operators with an application to eigenvalue problems},
  author={Carl, Bernd},
  journal={Journal of Approximation Theory},
  volume={32},
  number={2},
  pages={135--150},
  year={1981}
}

@book{de1981asymptotic,
  title={Asymptotic Methods in Analysis},
  author={de Bruijn, Nicolaas Govert},
  volume={4},
  year={1981},
  publisher={North-Holland},
  series={Bibliotheca Mathematica},
  edition={2}
}

@book{higgins1996sampling,
  title={Sampling Theory in Fourier and Signal Analysis: Foundations},
  author={Higgins, John Rowland},
  year={1996},
  publisher={Oxford University Press}
}

@article{devore2021neural,
  title={Neural network approximation},
  author={DeVore, Ronald and Hanin, Boris and Petrova, Guergana},
  journal={Acta Numerica},
  volume={30},
  pages={327--444},
  year={2021},
  publisher={Cambridge University Press}
}

@book{carlEntropyCompactnessApproximation1990,
  title = {Entropy, Compactness and the Approximation of Operators},
  author = {Carl, Bernd and Stephani, Irmtraud},
  year = {1990},
  edition = {1},
  publisher = {Cambridge University Press},
  series={Cambridge Tracts in Mathematics}
}

@article{corlessLambertFunction1996,
  title={On the {Lambert W} function},
  author={Corless, Robert M and Gonnet, Gaston H and Hare, David EG and Jeffrey, David J and Knuth, Donald E},
  journal={Advances in Computational Mathematics},
  volume={5},
  pages={329--359},
  year={1996},
  publisher={Springer}
}

@article{daubechiesConversionImperfectQuantizers2006,
  title={{A/D} conversion with imperfect quantizers},
  author={Daubechies, Ingrid and DeVore, Ronald A and G\"{u}nt\"{u}rk, C Sinan and Vaishampayan, Vinay A},
  journal={IEEE Transactions on Information Theory},
  volume={52},
  number={3},
  pages={874--885},
  year={2006},
  publisher={IEEE}
}

@techreport{donohoCountingBitsKolmogorov2000,
  title = {Counting bits with {Kolmogorov} and {Shannon}},
  author = {Donoho, David L.},
  year = {2000},
  institution = {Stanford University},
  number={2000-38}
}

@book{dudleyEcoleEteProbabilites1984,
  title = {École d'été de probabilités de {Saint-Flour XII} - 1982},
  author = {Dudley, R. M. and Kunita, H. and Ledrappier, F.},
  editor = {Hennequin, P. L.},
  year = {1984},
  series = {Lecture Notes in Mathematics},
  volume = {1097},
  publisher = {Springer Berlin Heidelberg}
}

@book{edmundsFunctionSpacesEntropy1996,
  title = {Function {Spaces}, {Entropy} {Numbers}, {Differential} {Operators}},
  author = {Edmunds, D. E. and Triebel, H.},
  year = {1996},
  edition = {1},
  publisher = {{Cambridge University Press}}
}

@article{elbrachterDeepNeuralNetwork2021,
  title = {Deep neural network approximation theory},
  author = {Elbrächter, Dennis and Perekrestenko, Dmytro and Grohs, Philipp and Bölcskei, Helmut},
  year = {2021},
  journal = {IEEE Transactions on Information Theory},
  volume = {67},
  number = {5},
  pages = {2581--2623}}

@article{grafSharpAsymptoticsMetric2004,
  title = {Sharp {asymptotics} of the {metric} {entropy} for {ellipsoids}},
  author = {Graf, Siegfried and Luschgy, Harald},
  year = {2004},
  journal= {Journal of Complexity},
  volume = {20},
  number = {6},
  pages = {876--882}}

@article{hausslerDecisionTheoreticGeneralizations1992,
  title = {Decision {theoretic} {generalizations} of the PAC {model} for {neural} {net} and {other} {learning} {applications}},
  author = {Haussler, David},
  year = {1992},
  journal = {Information and computation},
  volume = {100},
  number = {1},
  pages = {78--150}}

@article{haussler1997mutual,
  title={Mutual information, metric entropy and cumulative relative entropy risk},
  author={Haussler, David and Opper, Manfred},
  journal={The Annals of Statistics},
  volume={25},
  number={6},
  pages={2451--2492},
  year={1997}
}

@article{hutterMetricEntropyLimits2022,
  title = {Metric {entropy} {limits} on {recurrent} {neural} {network} {learning} of {linear} {dynamical} {systems}},
  author = {Hutter, Clemens and Gül, Recep and Bölcskei, Helmut},
  year = {2022},
  journal = {Applied and Computational Harmonic Analysis},
  volume = {59},
  pages = {198--223}}

@article{kempkaVolumesUnitBalls2017,
  title={Volumes of unit balls of mixed sequence spaces},
  author={Kempka, Henning and Vyb{\'\i}ral, Jan},
  journal={Mathematische Nachrichten},
  volume={290},
  number={8-9},
  pages={1317--1327},
  year={2017}
}

@article{kolmogorovCertainAsymptoticCharacteristics1956,
  title = {On certain asymptotic characteristics of completely bounded metric spaces},
  author = {Kolmogorov, A. N.},
  year = {1956},
  journal = {Dokl. Akad. Nauk},
  volume = {108},
  pages = {385--388}
}

@book{lorentzApproximationFunctions1966,
  title = {Approximation of {{Functions}}},
  author = {Lorentz, G. G.},
  year = {1966},
  series = {Athena Series},
  volume = {48},
  publisher={Holt, Rinehart and Winston}
}

@article{lorentzMetricEntropyApproximation1966,
  title = {Metric entropy and approximation},
  author = {Lorentz, G. G.},
  year = {1966},
  journal = {Bulletin of the American Mathematical Society},
  volume = {72},
  number = {6},
  pages = {903--937}}

@book{lorentzConstructiveApproximationAdvanced1996,
  title = {Constructive {Approximation}: {Advanced} {Problems}},
  author = {Lorentz, G. G. and von Golitschek, Manfred and Makovoz, Yuly},
  year = {1996},
  series = {Grundlehren der mathematischen {{Wissenschaften}}},
  number = {304},
  publisher = {Springer Berlin}}

@article{luschgyFunctionalQuantizationGaussian2002,
  title={Functional quantization of {Gaussian} processes},
  author={Luschgy, Harald and Pag{\`e}s, Gilles},
  journal={Journal of Functional Analysis},
  volume={196},
  number={2},
  pages={486--531},
  year={2002}
}

@article{luschgySharpAsymptoticsKolmogorov2004,
  title = {Sharp asymptotics of the {{Kolmogorov}} entropy for {{Gaussian}} measures},
  author = {Luschgy, Harald and Pag{\`e}s, Gilles},
  year = {2004},
  journal = {Journal of Functional Analysis},
  volume = {212},
  number = {1},
  pages = {89--120}
}

@article{mityaginApproximateDimensionBases1961,
  author = {B. S. Mityagin},
  journal = {Russian Mathematical Surveys},
  number = {4},
  title = {Approximate dimension and bases in nuclear spaces},
  volume = {16},
  year = {1961},
  pages={59-127}
}

@book{pollardEmpiricalProcessesTheory1990,
  title = {Empirical Processes: Theory and Applications},
  author = {Pollard, David},
  year = {1990},
  series = {NSF-CBMS Regional Conference Series in Probability and Statistics},
  volume = {2},
  publisher = {{Institute of Mathematical Statistics and American Statistical Association}}}

@book{rudin1953principles,
  title={Principles of mathematical analysis},
  author={Rudin, Walter},
  year={1976},
  series={International Series in Pure and Applied Mathematics},
  edition={3},
  publisher = {{McGraw-Hill}}
}

@book{rudinRealComplexAnalysis1987,
  title = {Real and Complex Analysis},
  author = {Rudin, Walter},
  year = {1987},
  edition = {3},
  publisher = {{McGraw-Hill}}}

@article{schuttEntropyNumbersDiagonal1984,
  title = {Entropy numbers of diagonal operators between symmetric {Banach} spaces},
  author = {Schütt, Carsten},
  year = {1984},
  journal = {Journal of Approximation Theory},
  volume = {40},
  number = {2},
  pages = {121--128}}

@book{shalev-shwartzUnderstandingMachineLearning2014,
title={Understanding Machine Learning},
subtitle={From Theory to Algorithms},
  author = {Shalev-Shwartz, Shai and Ben-David, Shai},
  year = {2014},
  publisher = {Cambridge University Press}}

@book{shiryayevSelectedWorksKolmogorov1993,
  title = {Selected {{works}} of {{A}}. {{N}}. {{Kolmogorov}}},
  subtitle={Volume III: Information Theory and the Theory of Algorithms},
  editor = {Shiryayev, Albert N.},
  year = {1993},
  series = {Mathematics and Its Applications},
  volume = {27},
  publisher = {{Springer Netherlands}}
}

@book{vitushkinTheoryTransmissionProcessing1961,
  title = {Theory of the {{Transmission}} and {{Processing}} of {{Information}}},
  author = {Vitushkin, A.G.},
  year = {1961}}

@book{wainwrightHighDimensionalStatistics2019,
  title = {High-Dimensional Statistics},
  subtitle={A Non-Asymptotic Viewpoint},
  author = {Wainwright, Martin J.},
  year = {2019},
  series = {Cambridge Series in Statistical and Probabilistic Mathematics},
  publisher = {Cambridge University Press},
  volume={48}
}

@article{zamesMetricComplexityCausal1977,
  title = {On the {{metric complexity}} of {{causal linear systems}}: ɛ-{{entropy}} and ɛ-{{dimension}} for {{continuous time}}},
  author = {Zames, George},
  year = {1979},
  pages = {222--230},
  journal = {IEEE Transactions on Automatic Control},
  volume={24},
  number={2}
}

@article{zamesNoteMetricDimension1993,
  title = {A {{note}} on {{metric dimension}} and {{feedback}} in {{discrete time}}},
  author = {Zames, George and Owen, J. G.},
  year = {1993},
  journal = {IEEE Transactions on Automatic Control},
  volume = {38},
  number = {4},
  pages = {664--667}
}

@article{gordon1987geometric,
  title={Geometric and probabilistic estimates for entropy and approximation numbers of operators},
  author={Gordon, Y and K{\"o}nig, Hermann and Sch{\"u}tt, Carsten},
  journal={Journal of Approximation Theory},
  volume={49},
  number={3},
  pages={219--239},
  year={1987},
  publisher={Elsevier}
}

@article{kossaczka2020entropy,
  title={Entropy numbers of finite-dimensional embeddings},
  author={Kossaczk{\'a}, Marta and Vyb{\'\i}ral, Jan},
  journal={Expositiones Mathematicae},
  volume={38},
  number={3},
  pages={319--336},
  year={2020},
  publisher={Elsevier}
}

@article{FISCHER2020105343,
title = {Some new bounds on the entropy numbers of diagonal operators},
journal = {Journal of Approximation Theory},
volume = {251},
number = {105343},
year = {2020},
author = {Simon Fischer}
}

\appendix

\section{Properties of Decay Rate Functions}\label{sec:proofqqqszracxref}

\begin{lemma}\label{lem:invertdecaratefct}
Let $\psi$ be a decay rate function with parameter $t^*$.
Then, $\psi$ is strictly increasing and invertible on $(t^*, \infty)$.
\end{lemma}

\begin{proof}[Proof.]
Fix $t_1, t_2 \in (t^*, \infty)$ such that $t_1<t_2$. 
We then have 
\begin{equation}\label{eq:decratestrctincr}
\psi(t_2) - \psi(t_1)
= t_2 \left[ \frac{\psi(t_2)}{t_2} - \frac{\psi(t_1)}{t_2}  \right]
>t_2 \left[ \frac{\psi(t_2)}{t_2} - \frac{\psi(t_1)}{t_1}  \right]
\geq 0.
\end{equation}
The inequality (\ref{eq:decratestrctincr}) directly implies that $\psi$ is strictly increasing on $(t^*, \infty)$.
Combined with the continuity of $\psi$ on $(t^*, \infty)$, which follows by definition, this allows us to conclude that $\psi$ is invertible on $(t^*, \infty)$.
\end{proof}

\begin{lemma}\label{lem:qqqszracxref}
Let $\psi$ be a decay rate function.
The $\psi$-average function $\delta$ and the $\psi$-difference function $\zeta$ are related according to
\begin{equation*}
\frac{\sigma_{\mathbb{K}}(d )\, \delta(d ) 
- \sigma_{\mathbb{K}}(d -1)\, \delta(d-1)}{\sigma_{\mathbb{K}}(d )}
= \zeta(d)\, (1-1/d), 
\quad \text{for all } d\geq 2.
\end{equation*}
\end{lemma}

\begin{proof}[Proof.]
We fix $d\geq 2$ and note the following chain of identities
\begin{align*}
d\, \delta(d) 
- (d-1)\,  \delta(d-1) 
=& \, \sum_{n=1}^d \left(\psi(d)-\psi(n)\right) - \sum_{n=1}^{d-1} \left(\psi(d-1)-\psi(n)\right) \\
=& \, d\, \psi(d) - (d-1)\, \psi(d-1) + \sum_{n=1}^{d-1} \psi(n) -  \sum_{n=1}^d \psi(n) \\
=& \, d \, \psi(d) - (d-1) \, \psi(d-1) - \psi(d) \\
=& \, (d-1)(\psi(d)-\psi(d-1))\\
=& \, (d-1)\, \zeta(d) = d\, \zeta(d)\, (1-1/d).
\end{align*}
Rewriting this relation according to
\begin{equation*}
\frac{\sigma_{\mathbb{K}}(d) \,  \delta(d ) 
- \sigma_{\mathbb{K}}(d -1) \, \delta(d-1)}{\sigma_{\mathbb{K}}(d)}
= \zeta(d)\, (1-1/d)
\end{equation*}
finishes the proof.
\end{proof}

\begin{lemma}\label{lem:zetaisod}
Let $\psi$ be a decay rate function with parameter $t^*$.
There exists a positive real number $\kappa$ such that the $\psi$-difference function $\zeta$ satisfies
\begin{equation*}
\zeta(d)\geq \kappa,
\quad \text{for all } \, d > t^*+2.
\end{equation*}
\end{lemma}

\begin{proof}[Proof.]
We fix $d > t^*+2$ and start by rewriting $\zeta$ according to
\begin{align}
\zeta(d)
=& \, \psi(d)-\psi(d-1) \label{eq:pppoetttrref}\\
=& \,  d \, \left(\frac{\psi(d)}{d} - \frac{\psi(d-1)}{d-1}\left(1-\frac{1}{d}\right)\right) \nonumber\\
=& \, d \, \left(\frac{\psi(d)}{d} - \frac{\psi(d-1)}{d-1}\right) + \frac{\psi(d-1)}{d-1}.\label{eq:pppoetttrref2}
\end{align}
As the function 
\begin{equation*}
t \in (t^*, \infty) \mapsto \frac{\psi(t)}{t} \in (0, \infty)
\end{equation*}
is non-decreasing by definition, we have  
\begin{equation*}
\frac{\psi(d)}{d} - \frac{\psi(d-1)}{d-1} 
\geq 0
\quad \text{and} \quad 
\frac{\psi(d-1)}{d-1} 
\geq \frac{\psi\left(\left\lfloor t^*+1\right\rfloor\right)}{\left\lfloor t^*+1\right\rfloor} .
\end{equation*}
Putting everything together, we obtain
\begin{align*}
\zeta(d)
= d \, \left(\frac{\psi(d)}{d} - \frac{\psi(d-1)}{d-1}\right) + \frac{\psi(d-1)}{d-1}
\geq \frac{\psi\left(\left\lfloor t^*+1\right\rfloor\right)}{\left\lfloor t^*+1\right\rfloor},
\end{align*}
which establishes the desired result upon setting
\begin{equation*}
\kappa \coloneqq \frac{\psi\left(\left\lfloor t^*+1\right\rfloor\right)}{\left\lfloor t^*+1\right\rfloor}.
\end{equation*}
\end{proof}

\begin{lemma}[Subadditivity of $\psi^{(-1)}$]\label{lem:subadditivitypsi}
The inverse $\psi^{(-1)}$ of the decay rate function $\psi$ is subadditive, i.e., for all $a, b > 0$, we have
\begin{equation*}
\psi^{(-1)}(a+b)
\leq \psi^{(-1)}(a) + \psi^{(-1)}(b).
\end{equation*}
\end{lemma}

\begin{proof}[Proof.]
We proceed in two steps, the first of which analyzes the variations of the function 
\begin{equation*}
u \mapsto \frac{\psi^{(-1)}(u)}{u}.
\end{equation*}
To this end, we compute
\begin{align}
\left[\frac{\psi^{(-1)}(u)}{u}\right]'
=& \, \left[\frac{\left[\psi'\left(\psi^{(-1)}(u)\right)\right]^{-1}}{u}- \frac{\psi^{(-1)}(u)}{u^2}\right] \label{eq:jhghguhjklkll}\\
=& \, \frac{\left[\psi^{(-1)}(u)\right]^2}{u^2\, \psi'\left(\psi^{(-1)}(u)\right)}
\left[\frac{u}{\left[\psi^{(-1)}(u)\right]^2} - \frac{\psi'\left(\psi^{(-1)}(u)\right)}{\psi^{(-1)}(u)} \right].\label{eq:jhghguhjklkll2}
\end{align}
Changing variables according to
\begin{equation*}
t
= \psi^{(-1)}(u),
\end{equation*}
(\ref{eq:jhghguhjklkll})--(\ref{eq:jhghguhjklkll2}) can be rewritten as
\begin{align}
\left[\frac{\psi^{(-1)}(u)}{u}\right]' 
=& \, \frac{t^2}{\left[\psi(t)\right]^2\, \psi'\left(t\right)}
\left[\frac{\psi(t)}{t^2} -\frac{\psi'\left(t\right)}{t} \right] \label{eq:jhghguhjklkllb1}\\
=& \, - \frac{t^2}{\left[\psi(t)\right]^2\, \psi'\left(t\right)}
\left[\frac{\psi(t)}{t}\right]'
\leq 0, \label{eq:jhghguhjklkllb2}
\end{align}
where the inequality holds for all $t \in (t^* , \infty)$ as a direct consequence of $\psi$ being strictly increasing (cf. Lemma \ref{lem:invertdecaratefct}) and satisfying the properties (\ref{eq:assumptionspsidecayfct}) in Definition \ref{def:expdecellips}.
It hence follows from (\ref{eq:jhghguhjklkllb1})--(\ref{eq:jhghguhjklkllb2}) that 
\begin{equation}\label{eq:subadddresfrststp}
u \mapsto \frac{\psi^{(-1)}(u)}{u}
\end{equation}
is non-increasing on $(0, \infty)$.

For the second step, fix $a, b > 0$ and observe that
\begin{equation}\label{eq:lkozkpseppezjzeh}
a\, \frac{\psi^{(-1)}(a+b)}{a+b}
\leq \psi^{(-1)}(a)
\quad \text{and} \quad 
b\, \frac{\psi^{(-1)}(a+b)}{a+b}
\leq \psi^{(-1)}(b).
\end{equation}
The desired result now follows by application of these two inequalities according to
\begin{equation*}
\psi^{(-1)}(a+b)
= a\, \frac{\psi^{(-1)}(a+b)}{a+b} + b\, \frac{\psi^{(-1)}(a+b)}{a+b}
\stackrel{(\ref{eq:lkozkpseppezjzeh})}{\leq} \psi^{(-1)}(a) + \psi^{(-1)}(b).
\end{equation*}
\end{proof}

\section{Finite-dimensional ellipsoids}\label{sec:finite-dim-elli}

    When only finitely many semi-axes are non-zero in Definition~\ref{def: Infinite dimensional ellipsoid}, the monotone ordering implies 
    that these are associated with the initial indices. 
    In this case we say that the 
    ellipsoid is finite-dimensional and write $\mathcal{E}_p^d$, where $d$ denotes 
    the number of non-zero semi-axes.
    The proof of the main result Theorem~\ref{thm: scaling metric entropy infinite ellipsoids exp} reduces the problem to that of covering a finite-dimensional ellipsoid of dimension $d_\varepsilon$ satisfying \eqref{eq:ppppaaatutyyrieie2}.
    The present appendix provides the necessary tools to estimate the metric entropy of finite-dimensional ellipsoids.

Our approach is based on volume arguments.
Specifically, we shall exploit the fact that the total volume occupied by the covering balls of a given set is larger than or equal to the volume of the set.
Conversely, when packing balls into a set, the total volume occupied by these balls must be less than or equal to the volume of a slightly augmented version of the set.
This intuition is formalized in \cite[Lemma 5.7]{wainwrightHighDimensionalStatistics2019} for $x \in \mathbb{R}^{d}$ (the case $x \in \mathbb{C}^{d}$ can be handled similarly by identifying $\mathbb{C}$ with $\mathbb{R}^2$), which we restate here for completeness.

\begin{lemma}\label{lem: volume estimates}
Let $d \in \mathbb{N}^*$ and fix $\varepsilon>0$.
Consider the norms $\|\cdot\|$ and $\|\cdot\|'$ on $\mathbb{K}^d$
and let $\mathcal{B}$ and $\mathcal{B}'$ be their respective unit balls.
Then, the $\varepsilon$-covering number $N (\varepsilon ; \mathcal{B}, \|\cdot\|' ) $ satisfies
\begin{equation}\label{eq: volume estimates}
 \frac{\text{\normalfont{vol}}_{\mathbb{K}}(\mathcal{B})}{\text{\normalfont{vol}}_{\mathbb{K}}(\mathcal{B}')}
\leq N \left(\varepsilon ; \mathcal{B}, \|\cdot\|' \right) \, \varepsilon^{\sigma_{\mathbb{K}}(d)}
\leq 2^{\sigma_{\mathbb{K}}(d)} \frac{\text{\normalfont{vol}}_{\mathbb{K}}\left(\mathcal{B} + \frac{\varepsilon}{2} \mathcal{B}' \right)}{\text{\normalfont{vol}}_{\mathbb{K}}(\mathcal{B}')}.
\end{equation}
\end{lemma}

\noindent
Volume ratios akin to those in (\ref{eq: volume estimates}) will appear regularly in our analyses, in particular with $\|\cdot\|$ and $\|\cdot\|'$ given by $p-$ and $q-$norms, for some $p,q \in [1, \infty]$.
It is therefore convenient to introduce the notation 
\begin{equation}\label{eq:defvolratio}
V_{p,q}^{{\mathbb{K}}, d}
\coloneqq  \frac{\text{vol}_{\mathbb{K}}(\mathcal{B}_p)}{\text{vol}_{\mathbb{K}}(\mathcal{B}_q)}.
\end{equation}
Such volume ratios have been studied extensively in the literature; see, for example, \cite{kempkaVolumesUnitBalls2017,kossaczka2020entropy} and \cite[Chapter 3]{edmundsFunctionSpacesEntropy1996} for representative references in the context of metric entropy.

Now, observe that, by taking logarithms in (\ref{eq: volume estimates}), one can readily deduce the metric entropy scaling behavior of a $d$-dimensional unit ball $\mathcal{B}$ according to
\begin{equation}\label{eq:aaakjhvaaa}
H \left(\varepsilon ; \mathcal{B}, \|\cdot\|' \right) 
\sim_{\varepsilon \to 0} {\sigma_{\mathbb{K}}(d)} \log \left(\varepsilon^{-1}\right).
\end{equation}
Intuitively, this scaling quantifies that, in order to encode an element of $\mathcal{B}$ with precision $\varepsilon$, one needs to quantize each of its components using  $\log (\varepsilon^{-1})$ bits.

Recalling that finite-dimensional $p$-ellipsoids are unit balls with respect to $\|\cdot\|_{p, \mu}$-norms, it follows that their metric entropy behavior is, in principle, characterized by (\ref{eq:aaakjhvaaa}).
However, bringing out the dependency on the semi-axes $\{\mu_1, \dots, \mu_d\}$ and getting good constants in lower and upper bounds on metric entropy, ideally even sharp results, requires significantly more work. 
The following two theorems, whose proofs have been relegated to Appendix~\ref{sec:lakaplzksazaza}, address these issues and constitute the main results in this context.

\begin{theorem}\label{thm: metric entropy of finite ellipsoidsLB}
Let $d\in \mathbb{N}^*$ and $p,q\in [1, \infty]$.
Then, there exists a positive real constant $\kminus$ independent of $d$ such that, for all $\varepsilon>0$, the covering number of the ellipsoid $\mathcal{E}^d_p$ in $\|\cdot\|_q$-norm satisfies
\begin{equation*}\label{eq: metric entropy of finite ellipsoidsLB}
N \left(\varepsilon ; \mathcal{E}_p^d, \|\cdot\|_q \right)^{\frac{1}{{\sigma_{\mathbb{K}}(d)}}} \, \varepsilon
\geq \kminus \,  {\sigma_{\mathbb{K}}(d)}^{(\frac{1}{q}-\frac{1}{p})} \, \bar \mu_d,
\end{equation*}
where $\bar \mu_d$ is the geometric mean of the semi-axes of $\mathcal{E}^d_p$.
\end{theorem}

\noindent 
In the real case, 
Theorem~\ref{thm: metric entropy of finite ellipsoidsLB} is the metric-entropy analogue of 
\cite[Lemma~5]{FISCHER2020105343} for entropy numbers of diagonal operators.

\begin{theorem}\label{thm: metric entropy of finite ellipsoids}
Let $d\in \mathbb{N}^*$, let $p,q\in [1, \infty]$,  and assume that 
\begin{equation}\label{eq:thisisassuptionond}
H \left(\varepsilon ; \mathcal{E}^d_p, \|\cdot\|_q \right) \geq 2 \, {\sigma_{\mathbb{K}}(d)},
\quad \text{for all } \varepsilon\in [0, 2 \mu_d],
\end{equation}
where $\mu_d$ is the smallest semi-axis of $\mathcal{E}^d_p$.
Then, there exists a positive real constant $\kplus$ independent of $d$ such that, for all $\varepsilon\in [0 , 2 \mu_d]$, the covering number of the ellipsoid $\mathcal{E}^d_p$ in $\|\cdot\|_q$-norm satisfies
\begin{equation*}\label{eq: metric entropy of finite ellipsoids}
N \left(\varepsilon ; \mathcal{E}_p^d, \|\cdot\|_q \right)^{\frac{1}{{\sigma_{\mathbb{K}}(d)}}} \, \varepsilon
\leq  \kplus \,  {\sigma_{\mathbb{K}}(d)}^{(\frac{1}{q}-\frac{1}{p})} \, \bar \mu_d,
\end{equation*}
where $\bar \mu_d$ is the geometric mean of the semi-axes of $\mathcal{E}^d_p$.
\end{theorem}

\noindent

\iffalse 
\textcolor{blue}{
Again, there exists results upper-bounding the entropy numbers of diagonal operators such as \cite[Lemma~4]{FISCHER2020105343}.
However, the existing literature (such as \cite[Lemma~4]{FISCHER2020105343}) fails to obtain an upper bound that matches the lower bound of Theorem~\ref{thm: metric entropy of finite ellipsoidsLB}.
By contrast, setting the additional assumption \eqref{eq:thisisassuptionond}, the upper bound we obtain in Theorem~\ref{thm: metric entropy of finite ellipsoids} matches (up to a multiplicative constant independent of $d$) the lower bound of Theorem~\ref{thm: metric entropy of finite ellipsoidsLB}.
This is a key step in establishing sharp asymptotics when the dimension $d$ tends to infinity (cf the proof of Theorem~\ref{thm: scaling metric entropy infinite ellipsoids exp}).
}
\fi 

\noindent
Obviously, we must have
\begin{equation}\label{eq:obviousbutneeded}
\kplus \geq \kminus.
\end{equation}
Note that the upper bound in Theorem~\ref{thm: metric entropy of finite ellipsoids} requires the additional assumption $\varepsilon\in [0, 2 \mu_d]$.
This can be explained as follows.
If  $\varepsilon$ is large compared to $\mu_d$ (recall that the semi-axes are ordered), the problem of covering the $d$-dimensional ellipsoid can be reduced to that of covering a lower-dimensional version thereof.
This insight is illustrated, for $d=2$, in Figure \ref{fig: proof outline}.

\begin{center}
\begin{tikzpicture}[scale=1]
\draw[help lines, color=gray!30, dashed] (-7.9,-2.5) grid (-.6,2.5);

\draw[->, color=black!50] (-7,0)--(-1,0);
\draw[->, color=black!50] (-4,-2)--(-4,2);

\node[ellipse,
    draw,
    fill = gray!10,
    fill opacity=.1,
    minimum width = 5.33cm, 
    minimum height = .65cm] (e) at (-4,0) {};

\filldraw[color=black!60!green, fill=green!10, fill opacity=.5, draw opacity = .5, dash pattern=on \pgflinewidth off 1pt]
(-6.3,0) circle (.75) node[color=black!60!green, scale=.3, draw opacity=1] {x};

\node[scale=.4, color=black!60!green] (a1) at (-6.3,-.1) {$x_1$};

\filldraw[color=black!60!green, fill=green!10, fill opacity=.5, draw opacity = .5, dash pattern=on \pgflinewidth off 1pt]
(-1.8,0) circle (.75) node[color=black!60!green, scale=.3, draw opacity=1] {x};

\node[scale=.4, color=black!60!green] (b1) at (-1.8,-.1) {$x_5$};

\draw [stealth-stealth, color=black!60!green] (-1.8,0) -- node [right,midway, scale=.4] {$\varepsilon$}  (-1.8,0.75) ;

\filldraw[color=black!60!green, fill=green!10, fill opacity=.5, draw opacity = .5, dash pattern=on \pgflinewidth off 1pt]
(-2.925,0) circle (.75) node[color=black!60!green, scale=.3, draw opacity=1] {x};

\node[scale=.4, color=black!60!green] (e1) at (-2.925,-.1) {$x_4$};

\filldraw[color=black!60!green, fill=green!10, fill opacity=.5, draw opacity = .5, dash pattern=on \pgflinewidth off 1pt]
(-4.05,0) circle (.75) node[color=black!60!green, scale=.3, draw opacity=1] {x};

\node[scale=.4, color=black!60!green] (f1) at (-4.05,-.1) {$x_3$};

\filldraw[color=black!60!green, fill=green!10, fill opacity=.5, draw opacity = .5, dash pattern=on \pgflinewidth off 1pt]
(-5.175,0) circle (.75) node[color=black!60!green, scale=.3, draw opacity=1] {x};

\node[scale=.4, color=black!60!green] (g1) at (-5.175,-.1) {$x_2$};

\node[scale=.75] (g) at (-3, 1.5) {$\mathcal{E}^2_2$};

\draw [stealth-stealth, color=blue] (-4,-1) -- node [below,midway] {$\mu_1$}  (-6.68,-1) ;
\draw [-, dashed, color=blue!40] (-6.68,-1) -- (-6.68,0) ;

\draw [stealth-stealth, color=red] (-6.9,0) -- node [left,midway] {$\mu_2$}  (-6.9,.335) ;
\draw [-, dashed, color=red!40] (-6.9,.335) -- (-4,.335) ;

\end{tikzpicture}

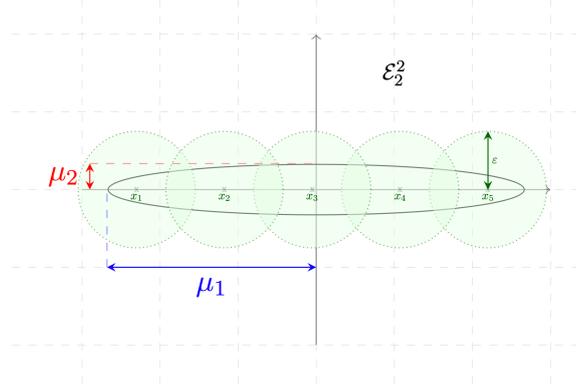
\captionof{figure}{\small Reduction of a two-dimensional covering problem to one dimension by lining up $\varepsilon$-balls, with $\varepsilon>2 \mu_2$, along the
$x$-axis.} 
\label{fig: proof outline}
\end{center}

When ellipsoids arise as images of unit balls under diagonal operators, their 
semi-axes correspond to the singular values of the underlying operator. It is 
therefore natural to compare our setting with what is available in that context 
(see, for example, 
\cite{kossaczka2020entropy,carl1981entropy,schuttEntropyNumbersDiagonal1984,
gordon1987geometric,FISCHER2020105343}). Existing bounds for diagonal operators, 
notably those in \cite{FISCHER2020105343}, yield expressions of the form 
$\prod_{n=1}^{d}\bigl(\|\mathrm{id}^{d}_{p,q}\| + \kappa_{q}\mu_{n}/\varepsilon\bigr)$, 
which differ from the multiplicative form 
$\prod_{n=1}^{d}(\mu_{n}/\varepsilon)$ that arises for ellipsoids after 
$\varepsilon$-dependent truncation. Because of this structural difference, the 
available diagonal-operator bounds do not provide the matching lower and 
upper bounds needed here. Theorems~\ref{thm: metric entropy of finite ellipsoidsLB} 
and~\ref{thm: metric entropy of finite ellipsoids} deliver these tight 
finite-dimensional bounds (up to the additional assumption~\eqref{eq:thisisassuptionond}), 
and, when combined with the correct $\varepsilon$-dependent truncation, yield the sharp asymptotics for 
infinite-dimensional ellipsoids developed below.

\section{Proofs}

\subsection{Proofs of Theorems \ref{thm: metric entropy of finite ellipsoidsLB} and \ref{thm: metric entropy of finite ellipsoids}}\label{sec:lakaplzksazaza}

We start with preparatory material needed in both proofs. First, observe that $\mathcal{E}^d_p$ is the image of $\mathcal{B}_p$ under the diagonal matrix
\begin{equation}\label{eq:diagmatrixeigenvalaxis}
A_\mu = 
\begin{pmatrix}
\mu_1 & 0 & \dots & 0 \\
0 & \mu_2 & \dots & 0 \\
\vdots & \vdots & \ddots & \vdots \\
0 & 0 & \dots & \mu_d 
\end{pmatrix}.
\end{equation}
Indeed, we have 
\begin{align*}
\mathcal{E}^d_p 
&= \left\{ x \in \mathbb{K}^d \mid \|x\|_{p, \mu} \leq 1 \right\} \\
&= \left\{ x \in \mathbb{K}^d \mid \|z\|_{p}\leq 1, \text{ such that  } x_n = \mu_n z_n, \text{ for all } n \in \{1, \dots, d\} \right\} \\
&= \left\{ A_\mu z \mid  z \in \mathbb{K}^d \text{ and } \|z\|_{p} \leq 1 \right\} 
= A_\mu \mathcal{B}_p.
\end{align*}
Next, note that
\begin{equation}\label{eq:lkhugyfdsssskbbjb}
\text{det}(A_\mu) 
=  \bar \mu_d^{\sigma_{\mathbb{K}}(d)},
\end{equation}
where we identified $\mathbb{C}$ with $\mathbb{R}^2$.
It now follows directly from $\mathcal{E}^d_p = A_\mu \mathcal{B}_p$ that the volumes of the ellipsoid $\mathcal{E}^d_p $ and the ball $\mathcal{B}_p$ are related according to 
\begin{equation}\label{eq: volume ellipsoid}
\text{vol}_{\mathbb{K}}(\mathcal{E}^d_p) 
=\text{vol}_{\mathbb{K}}(\mathcal{B}_p) \, \text{det}(A_\mu) 
\stackrel{(\ref{eq:lkhugyfdsssskbbjb})}{=}  \text{vol}_{\mathbb{K}}(\mathcal{B}_p) \, \bar \mu_d^{\sigma_{\mathbb{K}}(d)}.
\end{equation}
Asymptotic expressions for the volume of the ball $\mathcal{B}_p$ as the dimension 
tends to infinity are well known; see, for example, \cite[Eq.~(10)]{kossaczka2020entropy} 
for the real case. We conclude this preparatory discussion with a corresponding 
result on the asymptotics of the ratio of such volumes (recall 
\eqref{eq:defvolratio}).

%The volume of the ball $\mathcal{B}_p$ can be characterized asymptotically when the dimension tends to infinity, see, e.g., \cite[eq. (10)]%{kossaczka2020entropy} for the real case.
%We conclude this preparatory material with a comparable result on the asymptotics of the ratio of such volumes (recall  \eqref{eq:defvolratio}).

\begin{lemma}\label{lem: Asymptotic scaling of volume ratio}
Let $d\in\mathbb{N}^*$ and $p, q \in [1,\infty]$.
The volume ratio $V_{p,q}^{{\mathbb{K}}, d}$ satisfies
\begin{equation*}\label{eq: Asymptotic scaling of volume ratio}
\log\left\{ \left(V_{p,q}^{{\mathbb{K}}, d}\right)^{\frac{1}{{\sigma_{\mathbb{K}}(d)}}} \right\}
= \left(\frac{1}{q}-\frac{1}{p}\right) \log\left(\sigma_{\mathbb{K}}(d)\right)+O_{d\to\infty}\left(1\right).
\end{equation*}
\end{lemma}

\begin{proof}[Proof.]
See Appendix~\ref{sec:proofAsymptotic scaling of volume ratio}.
\end{proof}

\subsubsection*{Proof of Theorem~\ref{thm: metric entropy of finite ellipsoidsLB}}

The proof of Theorem~\ref{thm: metric entropy of finite ellipsoidsLB} is effected by applying Lemma \ref{lem: volume estimates} with $\mathcal{B}=\mathcal{E}^d_p$, $\mathcal{B}'=\mathcal{B}_q$, $\|\cdot\| = \|\cdot\|_{p, \mu}$, and $\|\cdot\|' = \|\cdot\|_q$, according to
\begin{equation}\label{eq: volume estimates finite dim ellipsoids}
N \left(\varepsilon ; \mathcal{E}^d_p, \|\cdot\|_q \right) \, \varepsilon^{\sigma_{\mathbb{K}}(d)}
\geq  \frac{\text{vol}_{\mathbb{K}}(\mathcal{E}^d_p)}{\text{vol}_{\mathbb{K}}(\mathcal{B}_q)}
\stackrel{(\ref{eq: volume ellipsoid})}{=} V_{p,q}^{{\mathbb{K}}, d}  \, \bar \mu_d^{\sigma_{\mathbb{K}}(d)},
\end{equation}
and observing that Lemma \ref{lem: Asymptotic scaling of volume ratio} then implies the existence of a constant $\kminus >0$ independent of $d$ such that 
\begin{equation}\label{eq: volume estimates finite dim ellipsoids2}
\left(V_{p,q}^{{\mathbb{K}}, d} \right)^{\frac{1}{{\sigma_{\mathbb{K}}(d)}}}
\geq \kminus \, {\sigma_{\mathbb{K}}(d)}^{\left(\frac{1}{q}-\frac{1}{p}\right)}.
\end{equation}
Using (\ref{eq: volume estimates finite dim ellipsoids2}) in (\ref{eq: volume estimates finite dim ellipsoids}) yields the desired result.
We additionally note that, in the case $p=q$, one can take $\kminus=1$. 
This observation is exploited in \cite[Theorem 5]{thirdpaper}.

\subsubsection*{Proof of Theorem~\ref{thm: metric entropy of finite ellipsoids}}

We split the proof of Theorem~\ref{thm: metric entropy of finite ellipsoids} into the cases $p\geq q$ and $p<q$.

\textit{Case $p\geq q$.}
With a view towards application of Lemma \ref{lem: volume estimates} with $\mathcal{B}=\mathcal{E}_p^d$ and $\mathcal{B}'=\mathcal{B}_q$, we first consider the set 
\begin{equation*}
\mathcal{E}^d_p + \frac{\varepsilon}{2} \mathcal{B}_q
\end{equation*}
and note that owing to $\mathcal{B}_q \subseteq \mathcal{B}_p$ and $\varepsilon \leq 2 \mu_d$, it holds that
\begin{equation}\label{eq:ploipol1}
\mathcal{E}^d_p + \frac{\varepsilon}{2} \mathcal{B}_q  
\subseteq \mathcal{E}^d_p + \frac{\varepsilon}{2} \mathcal{B}_p  
\subseteq \mathcal{E}^d_p + \mu_d \mathcal{B}_p.
\end{equation}
Next, as the semi-axes $\mu_1, \dots, \mu_d$ are non-increasing, we obtain 
\begin{equation}\label{eq:ploipol2}
\mathcal{E}^d_p + \mu_d \mathcal{B}_p  
\subseteq \mathcal{E}^d_p + \mathcal{E}^d_p 
= 2 \mathcal{E}^d_p,
\end{equation}
where the equality is thanks to the convexity of the ellipsoid $\mathcal{E}^d_p$ (recall that $p\geq 1$ by assumption).
Combining (\ref{eq:ploipol1}) and (\ref{eq:ploipol2}) now yields
\begin{equation}\label{eq: bound small balls}
\text{vol}_{\mathbb{K}}\left(\mathcal{E}^d_p + \frac{\varepsilon}{2} \mathcal{B}_q \right)
\leq 2^{\sigma_{\mathbb{K}}(d)} \, \text{vol}_{\mathbb{K}} \left(\mathcal{E}^d_p \right).
\end{equation}
Applying the upper bound from Lemma \ref{lem: volume estimates},  with $\mathcal{B}=\mathcal{E}^d_p$, $\mathcal{B}'=\mathcal{B}_q$, $\|\cdot\| = \|\cdot\|_{p, \mu}$, and $\|\cdot\|' = \|\cdot\|_q$, we get
\begin{align}
N \left(\varepsilon ; \mathcal{E}^d_p, \|\cdot\|_q \right) \, \varepsilon^{\sigma_{\mathbb{K}}(d)}
& \leq 2^{\sigma_{\mathbb{K}}(d)} \frac{\text{vol}_{\mathbb{K}}\left(\mathcal{E}_p^d + \frac{\varepsilon}{2} \mathcal{B}_q \right)}{\text{vol}_{\mathbb{K}}(\mathcal{B}_q)} \nonumber \\
& \stackrel{(\ref{eq: bound small balls})}{\leq} 4^{\sigma_{\mathbb{K}}(d)} \frac{\text{vol}_{\mathbb{K}}(\mathcal{E}^d_p)}{\text{vol}_{\mathbb{K}}(\mathcal{B}_q)} 
\stackrel{(\ref{eq: volume ellipsoid})}{=} 4^{{\sigma_{\mathbb{K}}(d)}} \, V_{p,q}^{{\mathbb{K}}, d} \, \bar \mu_d^{\sigma_{\mathbb{K}}(d)}. \label{eq:catapultebalancoire}
\end{align}
From Lemma \ref{lem: Asymptotic scaling of volume ratio} we know that there exists a constant $C>0$ that is independent of $d$ and such that 
\begin{equation*}
\left(V_{p,q}^{{\mathbb{K}}, d} \right)^{\frac{1}{{\sigma_{\mathbb{K}}(d)}}}
\leq C \, {\sigma_{\mathbb{K}}(d)}^{\left(\frac{1}{q}-\frac{1}{p}\right)},
\end{equation*}
which, when used in (\ref{eq:catapultebalancoire}), leads to
\begin{equation}\label{eq:lkojihuezdjksfpoezss}
N \left(\varepsilon ; \mathcal{E}_p^d, \|\cdot\|_q \right)^{\frac{1}{{\sigma_{\mathbb{K}}(d)}}} \, \varepsilon
\leq  4 \, C \,  {\sigma_{\mathbb{K}}(d)}^{\left(\frac{1}{q}-\frac{1}{p}\right)} \, \bar \mu_d.
\end{equation}
Setting $\kplus\coloneqq 4\, C$ yields the desired result.

\textit{Case $p < q$.}
The proof proceeds through arguments based on entropy numbers. % \cite{carlEntropyCompactnessApproximation1990}.
We first recall that, for $m\in\mathbb{N}^*$, the $m$-th entropy number $e_m(T)$ of a linear operator $T \colon (\mathbb{K}^d, \| \cdot \|) \to (\mathbb{K}^d, \| \cdot \|')$ is defined according to 
\begin{equation*}
e_m(T)
\coloneqq \inf \left\{ \varepsilon >0 \mid H \left(\varepsilon ; T(\mathcal{B}), \|\cdot\|' \right) \leq m  \right\},
\end{equation*}
where $\mathcal{B}$ denotes the unit ball in $\mathbb{K}^d$ w.r.t. the norm $\| \cdot \|$ (see \cite[Definition 1.3.1]{edmundsFunctionSpacesEntropy1996}).
We are particularly interested in the mapping 
\begin{equation}\label{eq:defoptmupq}
T_{p,q, \mu} \coloneqq \text{id} \colon (\mathbb{K}^d, \| \cdot \|_{p, \mu}) \to (\mathbb{K}^d, \| \cdot \|_q),
\end{equation}
that is, the identity operator between $\mathbb{K}^d$ equipped with the $p$-ellipsoid norm $\| \cdot \|_{p, \mu}$ and $\mathbb{K}^d$ equipped with the $q$-norm $\| \cdot \|_q$.
With this choice, we have
\begin{equation*}
e_m(T_{p,q, \mu})
\coloneqq \inf \left\{ \varepsilon >0 \mid H \left(\varepsilon ; \mathcal{E}^d_p, \|\cdot\|_q \right) \leq m  \right\}.
\end{equation*}
Additionally, we fix $m$ according to 
\begin{equation*}
m
\coloneqq \frac{1}{2}\lfloor H \left(\varepsilon ; \mathcal{E}^d_p, \|\cdot\|_q \right) \rfloor,
\end{equation*}
so that trivially
\begin{equation}\label{eq:verynicerelationsmmmhe}
\varepsilon \leq e_{2 m}\left(T_{p,q, \mu}\right)
\quad \text{and} \quad 
N \left(\varepsilon ; \mathcal{E}^d_p, \|\cdot\|_q \right) \leq 2^{2m+1}.
\end{equation}
Next, note that the operator $T_{p,q, \mu}$ can be factorized according to 
\begin{equation*}
T_{p,q, \mu}
= \injop{p}{q}\circ A_{\mu},
\end{equation*}

\noindent
where $\injop{p}{q}$ refers to the identity operator between $\mathbb{K}^d$ equipped with the $p$-norm $\| \cdot \|_{p}$ and $\mathbb{K}^d$ equipped with the $q$-norm $\| \cdot \|_q$, and $A_{\mu}$ %is the diagonal operator from $\mathbb{K}^d$ equipped with the $p$-norm $\| \cdot \|_{p}$ to itself 
was defined in (\ref{eq:diagmatrixeigenvalaxis}). Using the multiplicativity property of entropy numbers (see e.g. \cite[Equation 15.7.2]{lorentzConstructiveApproximationAdvanced1996}), we get
\begin{equation}\label{eq:multiplicatentnb}
e_{2m} \left(\injop{p}{q} \circ A_{\mu}  \right)
\leq e_m \left(\injop{p}{q} \right) \, e_m \left(A_{\mu} \right),
\end{equation}
which allows us to reduce the problem of upper-bounding the entropy number $e_{2m} \left(T_{p,q, \mu} \right)$ to that of upper-bounding the entropy numbers $e_m \left(\injop{p}{q} \right)$ and $e_m \left(A_{\mu} \right)$.
Under the assumption (\ref{eq:thisisassuptionond}) and using the fact that ${\sigma_{\mathbb{K}}(d)}$ is integer, one has $m\geq {\sigma_{\mathbb{K}}(d)}$ and hence the classical result \cite[Theorem 1]{schuttEntropyNumbersDiagonal1984} on the entropy number of diagonal operators applies, which can be reformulated in our setting according to
\begin{equation}\label{eq:thisisaneqlabel11}
e_m \left(\injop{p}{q} \right)
\leq C_1 2^{-\frac{m}{{\sigma_{\mathbb{K}}(d)}}} {\sigma_{\mathbb{K}}(d)}^{\left(\frac{1}{q}-\frac{1}{p}\right)},
\end{equation}
where $C_1$ is a positive numerical constant.
Moreover, by setting $p=q$ in (\ref{eq:lkojihuezdjksfpoezss}), we get
\begin{equation}\label{eq:thisisaneqlabel12}
e_m \left(A_{\mu} \right) 
\leq C_2 2^{-\frac{m}{{\sigma_{\mathbb{K}}(d)}}} \, \bar \mu_d,
\end{equation}

\noindent
where $C_2$ is a positive numerical constant.
Now combining (\ref{eq:thisisaneqlabel11}) and (\ref{eq:thisisaneqlabel12}) in (\ref{eq:multiplicatentnb}) yields
\begin{equation}\label{eq:boundonentnbfinal}
e_{2m} \left(T_{p,q, \mu} \right)
\leq C_0 2^{-\frac{2m}{{\sigma_{\mathbb{K}}(d)}}} {\sigma_{\mathbb{K}}(d)}^{\left(\frac{1}{q}-\frac{1}{p}\right)} \, \bar \mu_d,
\end{equation}
with $C_0$ a positive numerical constant.
Finally, using (\ref{eq:verynicerelationsmmmhe}) in combination with (\ref{eq:boundonentnbfinal}) results in the desired bound according to
\begin{align*}
N \left(\varepsilon ; \mathcal{E}_p^d, \|\cdot\|_q \right)^{\frac{1}{{\sigma_{\mathbb{K}}(d)}}} \, \varepsilon
&\stackrel{(\ref{eq:verynicerelationsmmmhe})}{\leq} 2^{\frac{2m+1}{\sigma_{\mathbb{K}}(d)}} e_{2m} \left(T_{p,q, \mu} \right) \\
& \stackrel{(\ref{eq:boundonentnbfinal})}{\leq}  2^{\frac{1}{{\sigma_{\mathbb{K}}(d)}}} \, C_0 \, {\sigma_{\mathbb{K}}(d)}^{\left(\frac{1}{q}-\frac{1}{p}\right)} \, \bar \mu_d
\leq \kplus \, {\sigma_{\mathbb{K}}(d)}^{\left(\frac{1}{q}-\frac{1}{p}\right)} \, \bar \mu_d,
\end{align*}
where we set $\kplus \coloneqq 2 \, C_0$ and used $2^{\frac{1}{{\sigma_{\mathbb{K}}(d)}}} \leq 2$ in the last inequality.
This concludes the proof.

\subsection{Proof of Theorem~\ref{thm: scaling metric entropy infinite ellipsoids exp}}\label{sec:proof-main}

We first introduce, for all $\varepsilon > 0$, the {effective dimension} of the ellipsoid $\mathcal{E}_p$ with respect to $ \|\cdot\|_q $ as
\begin{equation}\label{eq: definition of underbar n pt1mg5}
d_{\varepsilon} \coloneqq \min \left\{ k \in \mathbb{N}^* \mid   
 N \left(\varepsilon ; \mathcal{E}_p, \|\cdot\|_q \right)^{\frac{1}{{\sigma_{\mathbb{K}}(k)}}} \, \mu_k
\leq  \kplus \,  {\sigma_{\mathbb{K}}(k)}^{(\frac{1}{q}-\frac{1}{p})} \, \bar \mu_k  \right\},
\end{equation}
where $\kplus$ is the constant in the statement of Theorem~\ref{thm: metric entropy of finite ellipsoids}. 
It is formally established in Lemma \ref{lem:firstpropeffdim} (see Appendix~\ref{sec:prooffirstpropeffdim}) that the effective dimension is well-defined for all $\varepsilon>0$ and approaches infinity as $\varepsilon \to 0$.
Taking the logarithm in (\ref{eq: definition of underbar n pt1mg5}) yields
\begin{align}
\ & \sigma_{\mathbb{K}}(d_{\varepsilon} - 1) \Bigg\{\log \left(\kplus\right)  + \left(\frac{1}{q}-\frac{1}{p}\right)\log\left(\sigma_{\mathbb{K}}(d_{\varepsilon} - 1)\right) \label{eq:thisistheboundonlogNv1}\\
& \hspace{3.3cm} +  \frac{1}{d_{\varepsilon} - 1}\sum_{n=1}^{d_{\varepsilon} - 1} \log \left \{ \frac{\mu_{n}}{\mu_{d_{\varepsilon}-1}} \right \}\Bigg\} \nonumber \\
< \ & {H \left(\varepsilon ; \mathcal{E}_p, \|\cdot\|_q  \right)} \nonumber  \\
\leq \ & \sigma_{\mathbb{K}}(d_{\varepsilon}) \Bigg\{\log \left(\kplus\right)   + \left(\frac{1}{q}-\frac{1}{p}\right)\log\left(\sigma_{\mathbb{K}}(d_{\varepsilon})\right) +  \frac{1}{d_{\varepsilon}}\sum_{n=1}^{d_{\varepsilon}} \log \left \{ \frac{\mu_{n}}{\mu_{d_{\varepsilon}}} \right \}\Bigg\}. \label{eq:thisistheboundonlogNv2}
\end{align}
Using 
\begin{equation*}
\delta(d)
= \frac{1}{d}\sum_{n=1}^{d} \left(\psi({d}) -\psi(n)\right)
\stackrel{(\ref{eq: that's some big boy exp})}{=} \frac{1}{d}\sum_{n=1}^{d} \log \left \{ \frac{\mu_{n}}{\mu_{d}} \right \},
\end{equation*}
and identifying the term $\log \left(\kplus\right)$ as $O_{d_\varepsilon\to \infty}\left(1\right)$, we can rewrite the upper bound in (\ref{eq:thisistheboundonlogNv1})--(\ref{eq:thisistheboundonlogNv2}) according to
\begin{equation}\label{eq:oooiunbbbcjjpj1}
 {H \left(\varepsilon ; \mathcal{E}_p, \|\cdot\|_q\right)}
\leq  \sigma_{\mathbb{K}}\left(d_\varepsilon\right)  
\left\{\delta\left(d_\varepsilon\right)  
+ \left(\frac{1}{q}-\frac{1}{p}\right) \log \left( \sigma_{\mathbb{K}}\left(d_\varepsilon\right) \right)
 + O_{d_\varepsilon\to \infty}\left(1\right) \right\}.
\end{equation}
The lower bound in (\ref{eq:thisistheboundonlogNv1})--(\ref{eq:thisistheboundonlogNv2}) can equivalently be expressed as
\begin{align*}
& \ \sigma_{\mathbb{K}}\left(d_\varepsilon-1\right)  
\left\{\delta\left(d_\varepsilon-1\right)  
+ \left(\frac{1}{q}-\frac{1}{p}\right) \log \left(\sigma_{\mathbb{K}}\left(d_\varepsilon-1\right)\right)
 + O_{d_\varepsilon\to \infty}\left(1\right) \right\} \\
 =& \ \sigma_{\mathbb{K}}\left(d_\varepsilon\right)  
\left\{\delta\left(d_\varepsilon\right)  
+ \left(\frac{1}{q}-\frac{1}{p}\right) \log \left(\sigma_{\mathbb{K}}\left(d_\varepsilon\right)\right) - \zeta\left(d_\varepsilon\right)\, (1-1/d_\varepsilon)
 + O_{d_\varepsilon\to \infty}\left(1\right) 
 \right\},
\end{align*}
where we used 
\begin{equation*}
\sigma_{\mathbb{K}}\left(d_\varepsilon-1\right) \log \left(\sigma_{\mathbb{K}}\left(d_\varepsilon-1\right)\right) 
= \sigma_{\mathbb{K}}\left(d_\varepsilon\right) \left( \log \left(\sigma_{\mathbb{K}}\left(d_\varepsilon\right)\right) + O_{d_\varepsilon\to \infty}\left(1\right)\right)
\end{equation*} 
together with the identity
\begin{equation*}
\frac{\sigma_{\mathbb{K}}(d )\, \delta(d ) 
- \sigma_{\mathbb{K}}(d -1)\, \delta(d-1)}{\sigma_{\mathbb{K}}(d )}
= \zeta(d)\, (1-1/d), 
\quad \text{for all } d\geq 2,
\end{equation*}
established in Lemma \ref{lem:qqqszracxref} (see Appendix~\ref{sec:proofqqqszracxref}).
With
\begin{equation}\label{eqlkjhgfdfghjghvbn}
1 = O_{d\to\infty}\left(\zeta(d)\right),
\quad \text{and therefore} \quad
d = O_{d\to\infty}\left(d \, \zeta(d)\right),
\end{equation}
which follows from Lemma \ref{lem:zetaisod} (see Appendix~\ref{sec:proofqqqszracxref}), we can further rewrite the lower bound in (\ref{eq:thisistheboundonlogNv1})--(\ref{eq:thisistheboundonlogNv2}) according to
\begin{equation}\label{eq:oooiunbbbcjjpj4}
\sigma_{\mathbb{K}}  \left(d_\varepsilon\right)  
\left\{ \delta\left(d_\varepsilon\right)   
+ \left(\frac{1}{q}-\frac{1}{p}\right) \log \left(\sigma_{\mathbb{K}}\left(d_\varepsilon\right)\right) 
 +  O_{d_\varepsilon\to \infty}\left(\zeta\left(d_\varepsilon\right)\right)  \right\}. 
\end{equation}
Combining the upper bound (\ref{eq:oooiunbbbcjjpj1}) with the lower bound (\ref{eq:oooiunbbbcjjpj4}) and using once more (\ref{eqlkjhgfdfghjghvbn}), we 
obtain the desired final result according to
%now get the asymptotic scaling of the metric entropy according to
\begin{align*}
H \left(\varepsilon ; \mathcal{E}_p, \|\cdot\|_q \right) 
= & \, \sigma_{\mathbb{K}}\left(d_\varepsilon\right)  
\left\{\delta\left(d_\varepsilon\right)  
+ \left(\frac{1}{q}-\frac{1}{p}\right) \log \left(\sigma_{\mathbb{K}}\left(d_\varepsilon\right)\right)
 + O_{d_\varepsilon\to \infty}\left(\zeta\left(d_\varepsilon\right)\right) \right\}.
\end{align*}
It remains to establish that $d_\varepsilon$ can, indeed, be expressed as in (\ref{eq:ppppaaatutyyrieie2}).
%of the effective dimension $d_\varepsilon$ on $\varepsilon$.
This will be accomplished with the help of the following two lemmata.

\begin{lemma}\label{lemlkjhgfdaa}
There exists a positive real constant $\cfracminus \leq 1$ independent of $\varepsilon$ such that 
\begin{equation*}
\varepsilon 
\geq \cfracminus \, \mu_{d_{\varepsilon}},
\quad \text{for all } \varepsilon>0,
\end{equation*}
where $d_\varepsilon$ is as defined in (\ref{eq: definition of underbar n pt1mg5}).
\end{lemma}

\begin{proof}[Proof.]
See Appendix~\ref{sec:prooflemlkjhgfdaa}.
\end{proof}

\begin{lemma}\label{lemlkjhgfdaa2}
There exist a positive real number $\varepsilon^*$ and a constant $\cfracplus \geq 1$ such that, for all $\varepsilon \in (0, \varepsilon^*)$, 
\begin{equation*}\label{eq: intermediary bound ellipsoids 1 deouf5}
\varepsilon 
\leq \cfracplus \, \mu_{d_{\varepsilon}-1},
\end{equation*}
where $d_\varepsilon$ is as defined in (\ref{eq: definition of underbar n pt1mg5}).
\end{lemma}

\begin{proof}[Proof.]
See Appendix~\ref{sec:prooflemlkjhgfdaa2}.
\end{proof}

\noindent
Let $\cfracminus$, $\cfracplus$, and $\varepsilon^*>0$ be as in Lemmata \ref{lemlkjhgfdaa} and \ref{lemlkjhgfdaa2}.
Choose $\varepsilon \in (0, \varepsilon^*)$ small enough for
\begin{equation}\label{eqlmkppojdjkd}
\log \left(\varepsilon^{-1}\right)+\log \left(c_0\right) +\log \left(\cfracminus\right) > 0
\quad \text{and} \quad 
d_\varepsilon-1 > t^*
\end{equation}
to hold.
It further follows from Lemmata \ref{lemlkjhgfdaa} and \ref{lemlkjhgfdaa2} that
\begin{equation}\label{eq: lim sup et inf bg50}
\log \left(\mu_{d_{\varepsilon}}\right) + \log(\cfracminus)
\leq \log \left(\varepsilon\right) 
\leq \log \left(\mu_{d_{\varepsilon}-1}\right) + \log(\cfracplus).
\end{equation}
On the other hand we have from (\ref{eq: that's some big boy exp})
\begin{equation}\label{eq: lim sup et inf bg5}
\log \left(\mu_{d}\right)
= -\psi(d) + \log \left(c_0\right), 
\quad \text{for all } d\in \mathbb{N}^*.
\end{equation}
Combining (\ref{eq: lim sup et inf bg50}) and (\ref{eq: lim sup et inf bg5}) yields 
\begin{equation*}
\psi(d_{\varepsilon}-1) - \log(\cfracplus)  
\leq \log \left(\varepsilon^{-1}\right) + \log \left( c_0 \right)
\leq \psi(d_{\varepsilon}) - \log(\cfracminus),
\end{equation*}
or, equivalently, 
\begin{equation}\label{eq: maybe almost the end0}
\begin{cases}
d_{\varepsilon} \geq   \psi^{(-1)}  \left[\log \left(\varepsilon^{-1}\right) + \log \left( c_0 \right) + \log(\cfracminus) \right], \text{ and } 
 \\[2mm]
 d_{\varepsilon} \leq   \psi^{(-1)} \left[\log \left(\varepsilon^{-1}\right) + \log \left(c_0 \right) + \log(\cfracplus)  \right] +1,
\end{cases}
\end{equation}
where we used that the second part of (\ref{eqlmkppojdjkd}) together with Lemma \ref{lem:invertdecaratefct} ensures that $\psi$ can, indeed, be inverted,
and the first part of (\ref{eqlmkppojdjkd}) guarantees that both bounds are well defined.
Using the subadditivity of $\psi^{(-1)}$ (cf. Lemma \ref{lem:subadditivitypsi} in Appendix~\ref{sec:proofqqqszracxref}) and recalling that $\cfracminus \leq 1 \leq \cfracplus$, we get from (\ref{eq: maybe almost the end0})
\begin{equation*}
\begin{cases}
d_{\varepsilon} \geq   \psi^{(-1)}  \left[\log \left(\varepsilon^{-1}\right) + \log \left( c_0 \right)\right] - \psi^{(-1)}  \left[-\log(\cfracminus) \right], \text{ and } 
 \\[2mm]
 d_{\varepsilon} \leq   \psi^{(-1)} \left[\log \left(\varepsilon^{-1}\right)+ \log \left(c_0 \right) \right] + \psi^{(-1)} \left[\log(\cfracplus)  \right] +1,
\end{cases}
\end{equation*}
which, in turn, allows us to finish the proof by concluding that
\begin{equation*}\label{eq: maybe almost the end1}
d_{\varepsilon}
=\psi^{(-1)} \left[\log \left(\varepsilon^{-1}\right)+ \log \left(c_0\right)\right] +O_{\varepsilon \to 0}(1).
\end{equation*}

\subsection{Proof of Corollary~\ref{cor: Complex ellipsoid with exp0}}\label{sec:proofscor67}

First, note that $\psi(t)=ct$ trivially satisfies the defining properties of decay rate functions with the choice $t^*=0$ (cf. Definition \ref{def:expdecellips}).
We can therefore apply Theorem~\ref{thm: scaling metric entropy infinite ellipsoids exp} to get %an asymptotic characterization of metric entropy according to
\begin{equation}\label{eq:jjjjjerzazkk0}
\small
H \left(\varepsilon ; \mathcal{E}_p, \|\cdot\|_q \right) 
=  \, \sigma_{\mathbb{C}}\left(d_\varepsilon\right)  
\left\{\delta\left(d_\varepsilon\right)  
+ \left(\frac{1}{q}-\frac{1}{p}\right) \log \left(\sigma_{\mathbb{C}}\left(d_\varepsilon\right) \right)
 + O_{d_\varepsilon\to \infty}\left(\zeta\left(d_\varepsilon\right)\right) \right\},
\end{equation}
with $d_{\varepsilon}$ satisfying
\begin{equation} \label{eq:jjjjjerzazkk1}
d_{\varepsilon}
= \psi^{(-1)} \left[\log \left(\varepsilon^{-1}\right) +\log \left(c_0\right)\right] + O_{\varepsilon\to 0}(1) 
= \frac{\log \left(\varepsilon^{-1}\right)}{c} + O_{\varepsilon\to 0}(1).
\end{equation}
With Definition \ref{def:deltapsiexpdec}, we obtain the $\psi$-average function according to
\begin{equation}\label{eq:jjjjjerzazkk3}
    \delta(d)
    = \frac{1}{d}\sum_{n=1}^{d} \left( cd - cn \right)
    = c \, \frac{d-1}{2},
    \quad \text{ for all } d \in \mathbb{N}^*, 
\end{equation}
together with the $\psi$-difference function
\begin{equation}\label{eq:jjjjjerzazkk4}
\zeta (d)
= \psi(d) - \psi(d-1)
= c,
\quad \text{ for all } d \geq 2.
\end{equation}
Finally, observe that  $\sigma_{\mathbb{C}}(d)=\alpha \, d$, for all $d \in \mathbb{N}^*$.
Using (\ref{eq:jjjjjerzazkk1}), (\ref{eq:jjjjjerzazkk3}), and (\ref{eq:jjjjjerzazkk4}) in (\ref{eq:jjjjjerzazkk0}), the desired result follows according to
\begin{align*}
H \left(\varepsilon ; \mathcal{E}_p, \|\cdot\|_q \right) 
= & \, \alpha \, d_\varepsilon 
\left\{\frac{c}{2} d_\varepsilon
+ \left(\frac{1}{q}-\frac{1}{p}\right) \log \left(d_\varepsilon\right)
 + O_{d_\varepsilon\to \infty}\left(1\right) \right\} \\
 = & \, \frac{\alpha\, c}{2} \, d_\varepsilon^2 + \alpha \,  \left(\frac{1}{q}-\frac{1}{p}\right) \, d_\varepsilon \log \left(d_\varepsilon\right) + O_{d_\varepsilon\to \infty}\left(d_\varepsilon\right) \\
 = & \, \frac{\alpha \, \log^2 \left(\varepsilon^{-1}\right)}{2\, c} \\ 
 + &  \, \frac{\alpha}{c} \left(\frac{1}{q}-\frac{1}{p}\right) \log \left(\varepsilon^{-1}\right) \log^{(2)} \left(\varepsilon^{-1}\right) + O_{\varepsilon\to 0}\left(\log \left(\varepsilon^{-1}\right)\right).
\end{align*}

\subsection{Proofs of Corollaries~\ref{cor: Complex ellipsoid with exp} and \ref{cor: Complex ellipsoid with exp2}}\label{sec:proofcor89}

First, note that $\psi(t) = ct(\log (t) -c')$ verifies the properties of decay rate functions with the choice
$t^* \coloneqq 2^{c'}$. We can therefore apply
Theorem~\ref{thm: scaling metric entropy infinite ellipsoids exp} with $\mathbb{K}=\mathbb{C}$ to get
\begin{align}\label{eq:kiuygfiiijeqjkzefe2}
H \left(\varepsilon ; \mathcal{E}_{p}, \|\cdot\|_q \right)
= & \, \alpha \,  d_\varepsilon \left\{\delta\left(d_\varepsilon\right)   + O_{\varepsilon\to 0}\left(\log\left(d_\varepsilon\right)+\zeta\left(d_\varepsilon\right)\right)\right\},
\end{align}
where
\begin{equation}\label{eq:lkjhgfdxccs}
d_{\varepsilon}
=\psi^{(-1)} \left[\log \left(\varepsilon^{-1}\right)+ \log(c_0)\right] + O_{\varepsilon \to 0} (1)
=\psi^{(-1)} \left[\log \left(\varepsilon^{-1}\right)\right] + O_{\varepsilon \to 0} (1).
\end{equation}
Here, we made use of the sublinearity of $\psi^{(-1)}$ to rid ourselves of the $\log(c_0)$ term.

The following two lemmata are needed in the study of the asymptotic behavior of the right-hand-side in (\ref{eq:kiuygfiiijeqjkzefe2}).

\begin{lemma}[Asymptotic behavior of the $\psi$-difference function]\label{lem:jhgfdtttgdjhk}
Let $c,c' > 0$ and consider the function $\psi \colon t \mapsto ct(\log (t) - c')$. 
The $\psi$-difference function $\zeta$ scales according to
\begin{equation*}
\zeta (d)
= O_{d\to \infty}\left(\log(d)\right).
\end{equation*}
\end{lemma}

\begin{proof}[Proof.]
The result follows directly from
\begin{align*}
\zeta(d)
= \psi(d)- \psi(d-1)
&= c \, d \log(d) - c \,(d-1)\log(d-1) -c \,c' \\
&= c \,d \log(d) - c \,d \log(d-1) + c \,\log(d-1) - c \,c'\\
&= - c \,d \log(1- 1/d) + c \,\log(d-1) -c \,c'\\
&= c \,\log(d) + O_{d\to \infty}(1).
\end{align*}
\end{proof}

\begin{lemma}[Asymptotic behavior of the $\psi$-average function]\label{lem:kljhugaqff}
Let $c,c' > 0$ and consider the function $\psi \colon t \mapsto ct(\log (t) - c')$. 
The $\psi$-average function $\delta$ scales according to
\begin{equation*}
    \delta(d)
    = \frac{cd\log(d)}{2} + \frac{c  \, d \, (\ln^{-1}(2)-2\, c')}{4} + O_{d\to \infty}\left(\log(d)\right).
\end{equation*}
\end{lemma}

\begin{proof}[Proof.]
See Appendix~\ref{sec:proofkljhugaqff}.
\end{proof}

\noindent
Using Lemma \ref{lem:jhgfdtttgdjhk}, we can now express (\ref{eq:kiuygfiiijeqjkzefe2}) %can first be simplified using Lemma \ref{lem:jhgfdtttgdjhk} 
according to 
\begin{equation}\label{eq:kiuygfiiijeqjkzefe}
H \left(\varepsilon ; \mathcal{E}_{p}, \|\cdot\|_q \right)
=  \, \alpha \,  d_\varepsilon \left\{\delta\left(d_\varepsilon\right)   + O_{\varepsilon\to 0}\left(\log\left(d_\varepsilon\right)\right)\right\}.
\end{equation}
Furthermore, direct application of Lemma \ref{lem:kljhugaqff} yields
\begin{equation*}
    \delta(d_\varepsilon)
    = \frac{c \, d_\varepsilon\log(d_\varepsilon) }{2} + \frac{c  \, d_\varepsilon \, (\ln^{-1}(2)-2\, c')}{4}
    + O_{\varepsilon\to 0}\left(\log(d_\varepsilon) \right),
\end{equation*}
so that (\ref{eq:kiuygfiiijeqjkzefe}) becomes
\begin{align}
H \left(\varepsilon ; \mathcal{E}_{p}, \|\cdot\|_q \right)
= &   \frac{\alpha  c  d_\varepsilon^2 \log(d_\varepsilon)}{2} \!  
 + \! \frac{\alpha  c d_\varepsilon^2 (\ln^{-1}(2)-2 c')}{4} 
 + O_{\varepsilon\to 0}\left(d_\varepsilon \log\left(d_\varepsilon\right)\right) \label{eq:ppooouuurrr1}\\
=& \, \frac{\alpha \, c}{2} \, d_\varepsilon^2 \left[  \log(d_\varepsilon) + \frac{\ln^{-1}(2)-2\, c'}{2}  \right] \left[1+O_{\varepsilon\to 0}\left(\frac{1}{d_\varepsilon}\right)\right]. \label{eq:ppooouuurrr2}
\end{align}
Inserting (\ref{eq:lkjhgfdxccs}) in (\ref{eq:ppooouuurrr1})--(\ref{eq:ppooouuurrr2}), we get
\begin{align}
H \left(\varepsilon ; \mathcal{E}_{p}, \|\cdot\|_q \right)
=& \,\frac{\alpha \, c}{2} \, \left(\psi^{(-1)}\left[\log \left(\varepsilon^{-1}\right) \right]+O_{\varepsilon\to 0}\left(1\right)\right)^2  \label{eq:lkoijaopjzaozs}  \\
& \, \times  \left[\log\left(\psi^{(-1)}\left[\log \left(\varepsilon^{-1}\right) \right]+O_{\varepsilon\to 0}\left(1\right)\right)
+ \frac{\ln^{-1}(2)-2\, c'}{2} \right] \nonumber\\
& \, \times \left[1 
+ O_{\varepsilon\to 0}\left(\frac{1}{\psi^{(-1)}\left[ \log\left(\varepsilon^{-1}\right) \right]}\right)\right] \nonumber\\
=& \, \frac{\alpha \, c}{2} \, \left(\psi^{(-1)}\left[\log \left(\varepsilon^{-1}\right) \right]\right)^2  \nonumber \\
& \, \times  \left[\log\left(\psi^{(-1)}\left[\log \left(\varepsilon^{-1}\right) \right]\right)
+ \frac{\ln^{-1}(2)-2\, c'}{2} \right] \nonumber\\
& \, \times \left[1 
+ O_{\varepsilon\to 0}\left(\frac{1}{\log\left(\psi^{(-1)}\left[ \log\left(\varepsilon^{-1}\right) \right] \right)}\right)\right]. \label{eq:lkoijaopjzaozs2}
\end{align}
Next, we characterize the inverse of the decay rate function $\psi(t) = ct(\log (t) -c')$.

\begin{lemma}\label{lem:kjhgvhjuhuhygatf}
Let $c,c' > 0$. The inverse of the function ${\psi \colon t \mapsto ct(\log(t) - c')}$ is given by
\begin{equation*}
\psi^{(-1)}\colon u \longmapsto \exp\left\{c'  \ln (2) + W\left( \frac{u \ln (2)}{c}e^{-c'  \ln (2) }\right)\right\},
\end{equation*}
where $W$ denotes the Lambert $W$-function.
\end{lemma}

\begin{proof}[Proof.]
See Appendix~\ref{sec:proofkjhgvhjuhuhygatf}.
\end{proof}

\noindent
Using %the expression for $\psi^{(-1)}$ derived in 
Lemma \ref{lem:kjhgvhjuhuhygatf} in (\ref{eq:lkoijaopjzaozs})--(\ref{eq:lkoijaopjzaozs2}), we obtain %the desired result
\begin{align*}
H \left(\varepsilon ; \mathcal{E}_{p}, \|\cdot\|_q \right)
=&\, \frac{\alpha \, c \,  2^{2 c'-1}}{\ln (2)} \,\exp\left\{2 W\left( \frac{\ln \left(\varepsilon^{-1}\right)}{2^{c'}c}\right)\right\}  
\left[ W\left( \frac{\ln \left(\varepsilon^{-1}\right)}{2^{c'}c}\right) +  \frac{1}{2}\right] \\
& \times \left[1
+ O_{\varepsilon\to 0}\left(\exp\left\{-W\left( \frac{\ln \left(\varepsilon^{-1}\right)}{2^{c'}c}\right)\right\}\right)\right],
\end{align*}
which concludes the proof of Corollary~\ref{cor: Complex ellipsoid with exp}.

For the proof of Corollary~\ref{cor: Complex ellipsoid with exp2}, we first need a characterization of the asymptotic behavior of $\psi^{(-1)}$.

\begin{lemma}\label{lem:lkjodihuzehchehjz2}
Let $c,c' > 0$. The inverse function of $\psi \colon t \mapsto ct(\log(t) - c')$  satisfies the relation
\begin{equation*}
\psi^{(-1)}(u)^2 \log\left(\psi^{(-1)}(u)\right)
=\frac{u^2}{c^2\log (u)} 
\left[1 +  \frac{\log^{(2)}(u)}{\log(u)} + o_{u\to \infty}\left(\frac{\log^{(2)}(u)}{\log(u)}\right)\right].
\end{equation*}
\end{lemma}

\begin{proof}[Proof.]
See Appendix~\ref{sec:prooflkjodihuzehchehjz2}.
\end{proof}

\noindent
Now, starting from (\ref{eq:lkoijaopjzaozs})--(\ref{eq:lkoijaopjzaozs2}) and using Lemma \ref{lem:lkjodihuzehchehjz2} with $u=\log\left(\varepsilon^{-1}\right)$ together with the observation  
\begin{align*}
 O_{\varepsilon\to 0}\left(\frac{1}{\log\left(\psi^{(-1)}\left[\log \left(\varepsilon^{-1}\right) \right] \right)}\right)
=& \,  O_{\varepsilon\to 0}\left(\frac{1}{\log^{(2)} \left(\varepsilon^{-1}\right)}\right) \\
=& \,  o_{\varepsilon\to 0}\left(\frac{\log^{(3)}\left(\varepsilon^{-1}\right)}{\log^{(2)} \left(\varepsilon^{-1}\right)}\right),
\end{align*}
which employs Lemma \ref{lem:lkjodihuzehchehjz}, we obtain %the desired result according to
\begin{align*}
H \left(\varepsilon ; \mathcal{E}_{p}, \|\cdot\|_q \right)
=\frac{\alpha \, \log^2\left(\varepsilon^{-1}\right)}{2\, c \log^{(2)} \left(\varepsilon^{-1}\right)} 
\left[1 +  \frac{\log^{(3)}\left(\varepsilon^{-1}\right)}{\log^{(2)}\left(\varepsilon^{-1}\right)} + o_{\varepsilon\to 0}\left(\frac{\log^{(3)}\left(\varepsilon^{-1}\right)}{\log^{(2)} \left(\varepsilon^{-1}\right)}\right)\right].
\end{align*}
This concludes the proof of Corollary~\ref{cor: Complex ellipsoid with exp2}.

\subsection{Proofs of Lemmata \ref{lem:poiuhgfyyyyyyy} and \ref{lem: Bounds on analytic norm}}

\subsubsection{Proof of Lemma \ref{lem:poiuhgfyyyyyyy}}\label{sec:prooffouriercoefs}

We first observe that, for all $x\in \mathbb{R}$, 
\begin{align*}
\lvert f_1(x) - f_2(x)  \rvert
&= \left \lvert \sum_{k=-\infty}^{\infty} a^{(1)}_k e^{ikx}- \sum_{k=-\infty}^{\infty} a^{(2)}_k e^{ikx} \right \rvert \\
&\leq   \sum_{k=-\infty}^{\infty} \left \lvert a^{(1)}_k - a^{(2)}_k\right \rvert  
= \left\| a^{(1)}-a^{(2)} \right\|_{\ell^1(\mathbb{Z})}.
\end{align*}
Upon taking the supremum over all $x\in \mathbb{R}$, we hence get the desired upper bound 
\begin{equation*}
d_{2 \pi}(f_1, f_2)
\leq \left\| a^{(1)}-a^{(2)} \right\|_{\ell^1(\mathbb{Z})}.
\end{equation*}
The sought lower bound is obtained through Parseval's identity according to
\begin{align*}
\left\| a^{(1)}-a^{(2)}\right\|_{\ell^2(\mathbb{Z})}^2
&= \sum_{k=-\infty}^{\infty} \left \lvert a^{(1)}_k - a^{(2)}_k\right \rvert^2 \\
&= \frac{1}{2\pi} \int_0^{2 \pi} \lvert f_1(x) - f_2(x)  \rvert^2 \, d x \\
&\leq \sup_{x \in \mathbb{R}} \, \lvert f_1(x) - f_2(x)  \rvert^2
=d_{2 \pi}(f_1, f_2)^2.
\end{align*}
This concludes the proof.

\subsubsection{Proof of Lemma \ref{lem: Bounds on analytic norm}}\label{sec:prooftaylorcoefs}

We start by rewriting the metric $d_r$  according to 
\begin{align*}
d_r(f_1, f_2)
& =  \sup_{z \in D(0;r)} \left\lvert f_1(z)- f_2(z)\right\rvert \\
& =  \sup_{z \in D(0;r)} \left\lvert\tilde f_1\left(\frac{z}{r}\right)- \tilde f_2\left(\frac{z}{r}\right)\right\rvert \\
& =  \sup_{z \in D(0;1)} \left\lvert\tilde f_1\left(z\right)- \tilde f_2\left(z\right)\right\rvert,
\end{align*}
where $\tilde f_1$ and $\tilde f_2$ are analytic functions defined on the unit disk $D(0 ; 1)$ via their power series expansions
\begin{equation*}
\tilde f_1(z) = \sum_{k=0}^\infty \tilde a_{k}^{(1)} z ^{k}
\quad \text{and} \quad 
\tilde f_2(z) = \sum_{k=0}^\infty \tilde a_{k}^{(2)} z ^{k}.
\end{equation*}
Applying \cite[Theorem 3.1]{hutterMetricEntropyLimits2022}, we obtain 
\begin{equation}\label{eq: thuisbkjadjokezhfbe}
d_r(f_1, f_2) 
= \sup_{\|x\|_{\ell^2}=1} \left\|\left(\tilde a^{(1)}-\tilde a^{(2)}\right) \star x \, \right\|_{\ell^2},
\end{equation}
where $\star$ denotes the convolution operator (for sequences). The desired upper bound now follows directly from (\ref{eq: thuisbkjadjokezhfbe}) by 
Young's convolution inequality according to 
\begin{align*}
d_r(f_1, f_2) 
&= \sup_{\|x\|_{\ell^2}=1} \left\|\left(\tilde a^{(1)}-\tilde a^{(2)}\right) \star x \right\|_{\ell^2} \\
&\leq \sup_{\|x\|_{\ell^2}=1} \left\|x \right\|_{\ell^2} \left\|\tilde a^{(1)}-\tilde a^{(2)} \right\|_{\ell^1}
= \left\|\tilde a^{(1)}-\tilde a^{(2)} \right\|_{\ell^1}.
\end{align*}
The sought lower bound follows from (\ref{eq: thuisbkjadjokezhfbe}) by particularizing the choice of $x$ to 
\begin{equation*}
x_0 
\coloneqq (1, 0, 0, \dots) \in \ell^2,
\end{equation*}
so that 
\begin{align*}
d_r(f_1, f_2) 
&= \sup_{\|x\|_{\ell^2}=1} \left\|\left(\tilde a^{(1)}-\tilde a^{(2)}\right) \star x \right\|_{\ell^2} \\
&\geq \left\|\left(\tilde a^{(1)}-\tilde a^{(2)}\right) \star x_0 \right\|_{\ell^2}
= \left\|\tilde a^{(1)} - \tilde a^{(2)}\right\|_{\ell^2}.
\end{align*}
This concludes the proof.

\subsection{Proofs of Theorem~\ref{thm:exptypecompanlyticfctme} and Corollary~\ref{cor:exptypecompanlyticfctme}}\label{sec:proofthm14cor15ss}

We start by developing material that is pertinent to both proofs. Let $\{a_k\}_{k\in\mathbb{N}}$ be the sequence of Taylor series coefficients of %$f$, i.e., we have
\begin{equation*}
f(z)
=\sum_{k=0}^\infty a_k z^k,
\quad \text{for all } z\in \mathbb{C}.
\end{equation*}
We start by defining the embedding
\begin{equation*}
\iota \colon f \in \mathcal{F}_{\text{exp}}^{A,C} \mapsto \{a_k\}_{k\in\mathbb{N}} \in \ell^{\infty}(\mathbb{N}),
\end{equation*}
where $\{a_k\}_{k\in\mathbb{N}} \in \ell^{\infty}(\mathbb{N})$ follows from arguments similar to those employed in the proof of Lemma \ref{lem:firstpartinclusionexpfct}. 
We wish to circumscribe and inscribe $\iota(\mathcal{F}_{\text{exp}}^{A,C})$ with ellipsoids.

%namely $\mathcal{E}_{\infty}(\{\mu_k\}_{k\in\mathbb{N}})$ with semi-axes 

\begin{lemma}\label{lem:firstpartinclusionexpfct}
Let $A$ and $C$ be positive real constants. It holds that %We have % and let $\{\mu_k\}_{k\in\mathbb{N}}$ be as in (\ref{eq:poijhvddsmskj}).
%Then, we have %the inclusion
\begin{equation*}
\iota \left(\mathcal{F}_{\text{exp}}^{A,C} \right)
\subseteq \mathcal{E}_{\infty}(\{\mu_k\}_{k\in\mathbb{N}}),
\end{equation*}
where
\begin{equation}\label{eq:poijhvddsmskj}
\begin{cases}
    \mu_0
    =  C, \\
    \mu_k
    = C \exp \left\{-k \left(\ln (k) - 1 - \ln (A) \right)\right\},\quad \text{for } k\geq 1.
\end{cases}
\end{equation}

\end{lemma}

\begin{proof}[Proof.]
See Appendix~\ref{sec:prooffirstpartinclusionexpfct}.
\end{proof}

\noindent
%We proceed to inscribe $\iota \left(\mathcal{F}_{\text{exp}}^{A,C} \right)$ by an ellipsoid.
%In Lemma \ref{lem:firstpartinclusionexpfct}, we found an ellipsoid circumscribing our class of interest. 
%We now find an inscribed ellipsoid.
%To this end, we first set
%introduce the constants
%\begin{equation}\label{eq:stairwaytoheaven3}
%    \tilde A \coloneqq  \frac{A}{2}
%    \quad \text{and} \quad 
%    \tilde C \coloneqq C \left[\sup_{k \in \mathbb{N}^*} \frac{\sqrt{2 \pi k}\, e^{\frac{1}{12k}}}{2^k}\right]^{-1},
%\end{equation}
%where $\tilde C$ is well-defined owing to Lemma \ref{lem:studyvariationscompfctbis}.
%We now relate $\iota(\mathcal{F}_{\text{exp}}^{A,C})$ to the ellipsoid $\mathcal{E}_{\infty}(\{\tilde \mu_k\}_{k\in\mathbb{N}})$ of semi-axes 

\begin{lemma}\label{lem:secdpartinclusionexpfct}
Let $A$ and $C$ be positive real constants and set
\begin{equation}\label{eq:stairwaytoheaven3}
    \tilde A \coloneqq  \frac{A}{2}
    \quad \text{and} \quad 
    \tilde C \coloneqq C \left[\sup_{k \in \mathbb{N}^*} \frac{\sqrt{2 \pi k}\, e^{\frac{1}{12k}}}{2^k}\right]^{-1}.
\end{equation}
Then, 
%$\tilde A$, $\tilde C$, and $\{\tilde \mu_k\}_{k\in\mathbb{N}}$ be as in (\ref{eq:stairwaytoheaven3}) and (\ref{eq:stairwaytoheaven2}).
%Then, we have %the inclusion
\begin{equation*}
\mathcal{E}_{\infty}(\{\tilde \mu_k\}_{k\in\mathbb{N}})
\subseteq \iota \left(\mathcal{F}_{\text{exp}}^{A,C}\right),
\end{equation*}
where
\begin{equation}\label{eq:stairwaytoheaven2}
\begin{cases}
    \tilde \mu_0
    =  \tilde C,\\
    \tilde \mu_k
    = \tilde C \exp \left\{-k \left(\ln (k) - 1 - \ln( \tilde A) \right)\right\}, \quad \text{for }k\geq 1.
\end{cases}
\end{equation}

\end{lemma}

\begin{proof}[Proof.]
See Appendix~\ref{sec:proofsecdpartinclusionexpfct}.
\end{proof}

\noindent
Combining Lemmata \ref{lem:firstpartinclusionexpfct} and \ref{lem:secdpartinclusionexpfct} 
%taken together establish the desired circumscription-inscription relation
%according to
%that the class $\mathcal{F}_{\text{exp}}^{A,C}$ relates to infinite-dimensional ellipsoids according to
%\begin{equation*}
%\mathcal{E}_{\infty}(\{\tilde \mu_k\}_{k\in\mathbb{N}})
%\subseteq \iota(\mathcal{F}_{\text{exp}}^{A,C})
%\subseteq \mathcal{E}_{\infty}(\{\mu_k\}_{k\in\mathbb{N}}).
%\end{equation*}
%Combining these inclusions with Lemma \ref{lem: Bounds on analytic norm}, yields
with Lemma \ref{lem: Bounds on analytic norm}, yields
\begin{align}
    H \left(\varepsilon ; \mathcal{E}_{\infty}(\{\tilde \mu_k\}_{k\in\mathbb{N}}), \|\cdot \|_2 \right)
    &\leq H \left(\varepsilon ; \mathcal{F}_{\text{exp}}^{A,C}, d_{1} \right) \label{eq: ljsjgirlssdnf1}\\
    &\leq H \left(\varepsilon ; \mathcal{E}_{\infty}(\{\mu_k\}_{k\in\mathbb{N}}), \|\cdot \|_1 \right). \label{eq: ljsjgirlssdnf2}
\end{align}
We are now ready to proceed to the proof of Theorem~\ref{thm:exptypecompanlyticfctme}.

\subsubsection{Proof of Theorem~\ref{thm:exptypecompanlyticfctme}}

Observe that, from (\ref{eq:stairwaytoheaven2}) and under the convention $0\log(0)=0$, we have 
\begin{equation*}
\tilde \mu_k 
= \tilde C \exp \left\{-k \left(\ln (k) - 1 - \ln (\tilde A) \right)\right\}
= \tilde C \exp \left\{-\ln (2) \, \psi(k)\right\},
\end{equation*}
for all $k\in \mathbb{N}$, with
\begin{equation*}
\psi(k)= c\, k \,  \left(\log (k) - c' \right),
\quad \text{where} \,\,
c= 1
\text{ and }
c'= \log \left(e \tilde A \right).
\end{equation*}
Application of Corollary~\ref{cor: Complex ellipsoid with exp} with $c= 1$ and $c'= \log \left(e \tilde A \right)$ now yields
\begin{align}
& H \left(\varepsilon ; \mathcal{E}_{\infty}(\{\tilde \mu_k\}_{k\in\mathbb{N}}), \|\cdot\|_2 \right) \label{eq: jsqlfhgytgfvdbcz 111}\\
& = \left(e \tilde A\right)^2\,\frac{\exp\left\{2 \beta_1(\varepsilon) \right\}}{\ln(2)} \left[\beta_1(\varepsilon) + \frac{1}{2}\right] \left[1
+ O_{\beta_1(\varepsilon)\to \infty}\left(\exp\left\{- \beta_1(\varepsilon) \right\}\right)\right], \label{eq: jsqlfhgytgfvdbcz 121}
\end{align}
where
\begin{equation*}
\beta_1(\varepsilon) 
\coloneqq  W\left( \frac{\ln \left(\varepsilon^{-1}\right)}{e \,  \tilde A}\right).
\end{equation*}
Likewise, upon application of Corollary~\ref{cor: Complex ellipsoid with exp} with 
%\begin{equation*}
$c= 1$ and
%\quad \text{and} \quad 
$c'= \log (e  A )$,
%\end{equation*}
we obtain
\begin{align}
& H \left(\varepsilon ; \mathcal{E}_{\infty}(\{ \mu_k\}_{k\in\mathbb{N}}), \|\cdot\|_1 \right) \label{eq: jsqlfhgytgfvdbcz 211}\\
& = \left(e  A\right)^2\,\frac{\exp\left\{2 \beta_2(\varepsilon) \right\}}{\ln(2)} \left[\beta_2(\varepsilon) + \frac{1}{2}\right] \left[1
+ O_{\beta_2(\varepsilon)\to \infty}\left(\exp\left\{- \beta_2(\varepsilon) \right\}\right)\right], \label{eq: jsqlfhgytgfvdbcz 221}
\end{align}
where
\begin{equation*}
\beta_2(\varepsilon) 
\coloneqq  W\left( \frac{\ln \left(\varepsilon^{-1}\right)}{e\, A}\right).
\end{equation*}
Using (\ref{eq: jsqlfhgytgfvdbcz 111})--(\ref{eq: jsqlfhgytgfvdbcz 121}) and (\ref{eq: jsqlfhgytgfvdbcz 211})--(\ref{eq: jsqlfhgytgfvdbcz 221}) in (\ref{eq: ljsjgirlssdnf1})--(\ref{eq: ljsjgirlssdnf2}), yields
\begin{align*}
    &  \left(e \tilde A\right)^2\,\frac{\exp\left\{2 \beta_1(\varepsilon) \right\}}{\ln(2)}  \left[\beta_1(\varepsilon)+ \frac{1}{2}\right] \left[1
+ O_{\beta_1(\varepsilon)\to \infty}\left(\exp\left\{- \beta_1(\varepsilon) \right\}\right)\right] \\
    & \leq   H \left(\varepsilon ; \mathcal{F}_{\text{exp}}^{A,C}, d_{1} \right)  \\
    & \leq \left(e  A\right)^2\,\frac{\exp\left\{2 \beta_2(\varepsilon) \right\}}{\ln(2)} \left[\beta_2(\varepsilon) + \frac{1}{2}\right] \left[1
+ O_{\beta_2(\varepsilon)\to \infty}\left(\exp\left\{- \beta_2(\varepsilon) \right\}\right)\right],
\end{align*}
thereby concluding the proof.

\subsubsection{Proof of Corollary~\ref{cor:exptypecompanlyticfctme}}
Upon application of Corollary~\ref{cor: Complex ellipsoid with exp2} with 
$c = 1$ and  
$c'= \log (e  \tilde A )$,
we obtain
\begin{align}
&  H \left(\varepsilon ; \mathcal{E}_{\infty}(\{\tilde \mu_k\}_{k\in\mathbb{N}}), \|\cdot\|_2 \right) \label{eq: jsqlfhgytgfvdbcz 11}\\
& = \frac{\log^{2}\left(\varepsilon^{-1}\right)}{{\log^{(2)} \left(\varepsilon^{-1}\right)}} \left(1+ \frac{\log^{(3)} \left(\varepsilon^{-1}\right)}{\log^{(2)} \left(\varepsilon^{-1}\right)} + o_{\varepsilon\to 0}\left(\frac{\log^{(3)} \left(\varepsilon^{-1}\right)}{\log^{(2)} \left(\varepsilon^{-1}\right)}\right)\right).\label{eq: jsqlfhgytgfvdbcz 12}
\end{align}
Likewise, Corollary~\ref{cor: Complex ellipsoid with exp2} with $c=1$ and $c'= \log (e  A )$,
yields
\begin{align}
&  H \left(\varepsilon ; \mathcal{E}_{\infty}(\{\mu_k\}_{k\in\mathbb{N}}), \|\cdot\|_1 \right) \label{eq: jsqlfhgytgfvdbcz 21}\\
& =  \frac{\log^{2}\left(\varepsilon^{-1}\right)}{{\log^{(2)} \left(\varepsilon^{-1}\right)}} \left(1+ \frac{\log^{(3)} \left(\varepsilon^{-1}\right)}{\log^{(2)} \left(\varepsilon^{-1}\right)} + o_{\varepsilon\to 0}\left(\frac{\log^{(3)} \left(\varepsilon^{-1}\right)}{\log^{(2)} \left(\varepsilon^{-1}\right)}\right)\right).\label{eq: jsqlfhgytgfvdbcz 22}
\end{align}
Using (\ref{eq: jsqlfhgytgfvdbcz 11})--(\ref{eq: jsqlfhgytgfvdbcz 12}) and (\ref{eq: jsqlfhgytgfvdbcz 21})--(\ref{eq: jsqlfhgytgfvdbcz 22}) in (\ref{eq: ljsjgirlssdnf1})--(\ref{eq: ljsjgirlssdnf2}), results in %delivers the asymptotic behavior  
\begin{equation*}
H \left(\varepsilon ; \mathcal{F}_{\text{exp}}^{A,C}, d_{1} \right)
= \frac{\log^{2}\left(\varepsilon^{-1}\right)}{{\log^{(2)} \left(\varepsilon^{-1}\right)}} \left(1+ \frac{\log^{(3)} \left(\varepsilon^{-1}\right)}{\log^{(2)} \left(\varepsilon^{-1}\right)} + o_{\varepsilon\to 0}\left(\frac{\log^{(3)} \left(\varepsilon^{-1}\right)}{\log^{(2)} \left(\varepsilon^{-1}\right)}\right)\right),
\end{equation*}
thereby concluding the proof.

\subsection{Proofs of Auxiliary Results}

\subsubsection{Proof of Lemma \ref{lemlkjhgfdaa}}\label{sec:prooflemlkjhgfdaa}

We argue that covering the ellipsoid $\mathcal{E}_p$ by balls of radius $\varepsilon$ essentially reduces to covering the corresponding finite-dimensional ellipsoid $\mathcal{E}_p^{d_{\varepsilon}}$ obtained by retaining the first $d_{\varepsilon}$ semi-axes of $\mathcal{E}_p$.
More precisely, we combine the inequality 
\begin{equation*}
N \left(\varepsilon ; \mathcal{E}_p, \|\cdot\|_q\right)
\geq N \left(\varepsilon ; \mathcal{E}_p^{d_{\varepsilon}}, \|\cdot\|_q\right)
\end{equation*}
with Theorem~\ref{thm: metric entropy of finite ellipsoidsLB} to obtain
\begin{equation*}
N \left(\varepsilon ; \mathcal{E}_p, \|\cdot\|_q \right)^{\frac{1}{{\sigma_{\mathbb{K}}(d_{\varepsilon})}}} \, \varepsilon
\geq \kminus \,  {\sigma_{\mathbb{K}}(d_\varepsilon)}^{\left(\frac{1}{q}-\frac{1}{p}\right)} \, \bar \mu_{d_\varepsilon}.
\end{equation*}
Now, using the definition of $d_{\varepsilon}$ in (\ref{eq: definition of underbar n pt1mg5}), this yields 
\begin{equation*}
\varepsilon 
\geq {\kminus}\, {\kplus}^{-1} \, \mu_{d_{\varepsilon}},
\end{equation*}
which, upon setting $\cfracminus \coloneqq {\kminus}\, {\kplus}^{-1} $ and noting that (\ref{eq:obviousbutneeded}) then implies $\cfracminus \leq 1$, concludes the proof.

\subsubsection{Proof of Lemma \ref{lemlkjhgfdaa2}}\label{sec:prooflemlkjhgfdaa2}

The proof is effected by reducing the infinite-dimensional covering problem to a finite-dimensional one followed by application of Theorem~\ref{thm: metric entropy of finite ellipsoids}.

\begin{lemma}\label{lem:lemboundtriangleineqlike}
Let $p, q \in [1, \infty]$ and $\rho>0$ be real numbers, and let $d\geq1$ be an integer.
 Let $\mathcal{E}_p$ be an infinite-dimensional ellipsoid with exponentially decaying semi-axes $\{\mu_n\}_{n \in \mathbb{N}^*}$, and let $\mathcal{E}_p^d$ be the $d$-dimensional ellipsoid obtained by retaining the first $d$ semi-axes of $\mathcal{E}_p$.
 Then, there exist an integer $d^*$ and a constant $K\geq 1$ not depending on $d$, such that, for all $d\geq d^*$, it holds that
 \begin{equation*}
N \left(\rho; \mathcal{E}_p^d, \|\cdot\|_q \right) 
\geq N \left(\bar \rho; \mathcal{E}_p, \|\cdot\|_q \right),
\ \ \text{with } \bar \rho \coloneqq 
\begin{cases}
    \left(\rho^q + K\mu_{d+1}^q \right)^{1/q}, \quad & \text{if } 1\leq q < \infty, \\
 \ \max\{\rho,\mu_{d+1}\} , \quad & \text{if } q= \infty.
    \end{cases}
 \end{equation*}
\end{lemma}

\begin{proof}[Proof.]
See Appendix~\ref{sec:prooflemboundtriangleineqlike}.
\end{proof}

\noindent
Now, let $d^*$ and $K\geq 1$ be as in  Lemma \ref{lem:lemboundtriangleineqlike} and let $\varepsilon^*_1>0$ be such that $d_{\varepsilon^*_1}-1\geq d^*$.
We can further assume, without loss of generality, that 
\begin{equation}\label{eq:jhuyhftdrtghjq}
\varepsilon 
> K^{1/q} \, \mu_{d_{\varepsilon}},
\end{equation}
as, otherwise, the statement of Lemma \ref{lemlkjhgfdaa2} follows trivially by setting ${\cfracplus \coloneqq K^{1/q}}\geq 1$.
Next, we define the radius 
\begin{equation}\label{eq:defrhosneww}
\rho 
\coloneqq 
\begin{cases}
 \left(\varepsilon^q - K \mu_{d_{\varepsilon}}^q \right)^{1/q}, \quad & \text{if } 1\leq q < \infty, \\
 \ \varepsilon, \quad & \text{if } q= \infty,
\end{cases}
\end{equation}
which is positive by assumption (\ref{eq:jhuyhftdrtghjq}).
As a direct consequence of (\ref{eq:defrhosneww}), we obtain
\begin{equation*}
    \varepsilon = \bar  \rho \coloneqq 
    \begin{cases}
    \left(\rho^q + K \mu_{d_{\varepsilon}}^q \right)^{1/q}, \quad & \text{if } 1\leq q < \infty, \\
 \ \varepsilon, \quad & \text{if } q= \infty,
    \end{cases}
\end{equation*}
so that $    N \left(\varepsilon ; \mathcal{E}_p, \|\cdot\|_q \right)
= N \left(\bar \rho; \mathcal{E}_p, \|\cdot\|_q \right)$.
Applying Lemma \ref{lem:lemboundtriangleineqlike} with $d=d_{\varepsilon}-1$, then yields
\begin{equation}\label{eq:applicationoflemboundtriangleineqlike}
N \left(\varepsilon ; \mathcal{E}_p, \|\cdot\|_q \right)
\leq N \left(\rho ; \mathcal{E}_p^{d_{\varepsilon}-1}, \|\cdot\|_q \right).
\end{equation}
Our strategy now consists of first upper-bounding $\rho$ and then using the definition (\ref{eq:defrhosneww}) to obtain a corresponding bound on $\varepsilon$.
Specifically, we employ Theorem~\ref{thm: metric entropy of finite ellipsoids} to obtain
\begin{equation}\label{eq: kljhgvwww}
\rho 
\leq \frac{\kplus \,  {\sigma_{\mathbb{K}}(d_{\varepsilon} - 1)}^{\left(\frac{1}{q}-\frac{1}{p}\right)} \, \bar \mu_{d_{\varepsilon} - 1}}{N \left(\rho ; \mathcal{E}_p^{d_{\varepsilon} - 1}, \|\cdot\|_q\right)^{\frac{1}{{\sigma_{\mathbb{K}}(d_{\varepsilon} - 1)}}}}.
\end{equation}
The assumption $\rho \leq 2\mu_{d_{\varepsilon} - 1}$ required by Theorem~\ref{thm: metric entropy of finite ellipsoids} will be verified below, in (\ref{eq: znkjlevfbhsdfs}).
Further, by (\ref{eq:applicationoflemboundtriangleineqlike})
and the fact that ${H \left(\varepsilon ; \mathcal{E}^{d_\varepsilon}_p, \|\cdot\|_q\right)}$ grows superlinearly in $d_\varepsilon$, there exists $\varepsilon^*_2>0$ such that the hypothesis (\ref{eq:thisisassuptionond}) needed in Theorem~\ref{thm: metric entropy of finite ellipsoids}, that is,
\begin{equation*}
H \left(\rho ; \mathcal{E}^{d_\varepsilon}_p, \|\cdot\|_q \right) \geq 2 \, {\sigma_{\mathbb{K}}(d_{\varepsilon})},
\quad \text{for all } \rho\in [0, 2 \mu_{d_{\varepsilon}}],
\end{equation*}
is satisfied for all $\varepsilon<\varepsilon^*_2$. 
Upon application of (\ref{eq:applicationoflemboundtriangleineqlike}), the bound (\ref{eq: kljhgvwww}) becomes
\begin{equation}\label{eq: kljhgvwwwbis}
\rho 
\leq \frac{\kplus \,  {\sigma_{\mathbb{K}}(d_{\varepsilon} - 1)}^{\left(\frac{1}{q}-\frac{1}{p}\right)} \, \bar \mu_{d_{\varepsilon} - 1}}{N \left(\varepsilon ; \mathcal{E}_p,  \|\cdot\|_q\right)^{\frac{1}{{\sigma_{\mathbb{K}}(d_{\varepsilon} - 1)}}}}.
\end{equation}
Moreover, we have 
\begin{equation}\label{eq: kljhgvwwwter}
\frac{\kplus \,  {\sigma_{\mathbb{K}}(d_{\varepsilon} - 1)}^{\left(\frac{1}{q}-\frac{1}{p}\right)} \, \bar \mu_{d_{\varepsilon} - 1}}{N \left(\varepsilon ; \mathcal{E}_p,  \|\cdot\|_q \right)^{\frac{1}{{\sigma_{\mathbb{K}}(d_{\varepsilon} - 1)}}}}
\leq \mu_{d_{\varepsilon} - 1}
\end{equation}
as a direct consequence of the definition (\ref{eq: definition of underbar n pt1mg5}) of $d_{\varepsilon}$.
Combining (\ref{eq: kljhgvwwwbis}) and (\ref{eq: kljhgvwwwter}), we now get an upper bound on the radius $\rho$ according to
\begin{equation}\label{eq: znkjlevfbhsdfs}
\rho
\stackrel{(\ref{eq: kljhgvwwwbis})}{\leq}  \frac{\kplus \,  {\sigma_{\mathbb{K}}(d_{\varepsilon} - 1)}^{\left(\frac{1}{q}-\frac{1}{p}\right)} \, \bar \mu_{d_{\varepsilon} - 1}}{N \left(\varepsilon ; \mathcal{E}_p,  \|\cdot\|_q\right)^{\frac{1}{{\sigma_{\mathbb{K}}(d_{\varepsilon} - 1)}}}} 
\stackrel{(\ref{eq: kljhgvwwwter})}{\leq }\mu_{d_{\varepsilon} - 1}. 
\end{equation}
For all $\varepsilon< \varepsilon^*\coloneqq \min\{\varepsilon^*_1, \varepsilon^*_2\}$, combining (\ref{eq: znkjlevfbhsdfs}) with (\ref{eq:defrhosneww}) now yields
\begin{equation*}
\varepsilon 
=
\begin{cases}
\left(\rho^q+ K\mu_{d_{\varepsilon}}^q\right)^{1/q}
\leq \left(\mu_{d_{\varepsilon} - 1}^q+ K\mu_{d_{\varepsilon}}^q\right)^{1/q}
\leq \cfracplus \mu_{d_{\varepsilon} - 1}, \quad & \text{if } 1\leq q < \infty, \\
 \ \rho \leq \mu_{d_{\varepsilon} - 1}, \quad & \text{if } q= \infty,
\end{cases}
\end{equation*}
where we have used that the semi-axes are non-increasing and we set ${\cfracplus\coloneqq (1+K)^{1/q} \geq 1}$.
This concludes the proof.

\subsubsection{Proof of Lemma \ref{lem:lemmainclusionofellipsoidsfouryeah}}\label{sec:prooflemmainclusionofellipsoidsfouryeah}

We first define the sequences
\begin{equation}\label{eq:poiuyhgddsdmsldnfbfjheeee}
\mu_n \coloneqq M e^{-s n} \quad \text{and} \quad 
\tilde \mu_{n} \coloneqq  M e^{-s \lfloor n/ 2 \rfloor}, 
\quad \text{for all } n \in \mathbb{N},
\end{equation}
and fix $f \in \mathcal{A}_s(M)$ with Fourier series coefficients $a_k$.
Then, we note that 
\begin{align*}
\sum_{n=1}^\infty \, \left\lvert \frac{\tilde a_n}{\tilde \mu_n} \right\rvert
&= \left\lvert \frac{\tilde a_1}{\tilde \mu_1} \right\rvert + \sum_{n=1}^\infty \, \left[\left\lvert \frac{\tilde a_{2n}}{\tilde \mu_{2n}} \right\rvert + \left\lvert \frac{\tilde a_{2n+1}}{\tilde \mu_{2n+1}} \right\rvert\right] \\
&= \left\lvert \frac{a_0}{ \mu_0} \right\rvert + \sum_{n=1}^\infty \, \left[\left\lvert \frac{a_{n}}{ \mu_{n}} \right\rvert + \left\lvert \frac{ a_{-n}}{ \mu_{n}} \right\rvert\right] 
= \sum_{k\in \mathbb{Z}} \, \left\lvert \frac{a_k}{\mu_{\lvert k \rvert}} \right\rvert
= \left \| \left\{\frac{a_k}{ \mu_{\lvert k \rvert}} \right\}_{k \in \mathbb{Z}} \right \|_{\ell^1(\mathbb{Z})},
\end{align*}
where $\{{\tilde{a}_n}\}_{n \in \mathbb{N}^*}$ is as defined in (\ref{eq:deftildea}). 
Moreover, $\tilde \mu_n \geq \mu_n^{(1)}$ readily implies
\begin{equation}\label{eq:lkjhgfcvbnbvcvc}
\left \| \left\{{\tilde a_n} \right\}_{n \in \mathbb{N}^*} \right\|_{1, \mu^{(1)}}
\geq \sum_{n=1}^\infty \, \left\lvert \frac{\tilde a_n}{\tilde \mu_n} \right\rvert
=\left \| \left\{\frac{a_k}{ \mu_{\lvert k \rvert}} \right\}_{k \in \mathbb{Z}} \right \|_{\ell^1(\mathbb{Z})}.
\end{equation}
The right-hand-side in (\ref{eq:lkjhgfcvbnbvcvc}) can be dealt with using the following lemma.

\begin{lemma}\label{lem:poiuhgfyyyyyyy3}
Let $M$ and $s$ be positive real numbers, and let $f$ be a $2\pi$-periodic function with Fourier series coefficients $\left\{{a_k}\right\}_{k \in \mathbb{Z}}$ satisfying
\begin{equation*}
\left \| \left\{\frac{a_k}{ \mu_{\lvert k \rvert}} \right\}_{k \in \mathbb{Z}} \right \|_{\ell^1(\mathbb{Z})}
\leq 1,
\end{equation*}
with the sequence $\left\{{\mu_n}\right\}_{n \in \mathbb{N}}$ as defined in (\ref{eq:poiuyhgddsdmsldnfbfjheeee}).
Then, $f \in \mathcal{A}_s(M)$.
\end{lemma}

\begin{proof}[Proof.]
See Appendix~\ref{sec:proofpoiuhgfyyyyyyy3}.
\end{proof}

\noindent
Combining Lemma \ref{lem:poiuhgfyyyyyyy3} with (\ref{eq:lkjhgfcvbnbvcvc}), it follows that 
\begin{equation}\label{eq:aqsdxwxcvvbdssszpnfn0}
\mathcal{E}_{1}\left(\left\{\mu_{n}^{(1)}\right\}_{n \in\mathbb{N}^*}\right)
\subseteq \iota(\mathcal{A}_s(M)).
\end{equation}

Next, upon observing that $\sqrt{2} \, \tilde \mu_n \leq \mu_n^{(2)}$, for all $n \in \mathbb{N}^*$, we readily obtain
\begin{equation}\label{eq:lkjhgfcvbnbvcvc2}
\left \| \left\{{\tilde a_n} \right\}_{n \in \mathbb{N}^*} \right\|_{2, \mu^{(2)}}
\leq \left(\sum_{n=1}^\infty \, \left\lvert \frac{\tilde a_n}{\sqrt{2} \, \tilde \mu_n} \right\rvert^2\right)^{1/2}
=\left \| \left\{\frac{a_k}{\sqrt{2} \, \mu_{\lvert k \rvert}} \right\}_{k \in \mathbb{Z}} \right \|_{\ell^2(\mathbb{Z})}.
\end{equation}
Next, we show that the right-most expression in (\ref{eq:lkjhgfcvbnbvcvc2}) can be upper-bounded by $1$.

\begin{lemma}\label{lem:poiuhgfyyyyyyy4}
Let $M$ and $s$ be positive real numbers, and consider $f \in \mathcal{A}_s(M) $ with Fourier series coefficients $\left\{{a_k}\right\}_{k \in \mathbb{Z}}$.
Then, we have
\begin{equation*}
\left \| \left\{\frac{a_k}{\sqrt{2} \, \mu_{\lvert k \rvert}} \right\}_{k \in \mathbb{Z}} \right \|_{\ell^2(\mathbb{Z})}
\leq 1,
\end{equation*}
where the sequence $\left\{{\mu_n}\right\}_{n \in \mathbb{N}}$ has been defined in (\ref{eq:poiuyhgddsdmsldnfbfjheeee}).
\end{lemma}

\begin{proof}[Proof.]
See Appendix~\ref{sec:proofpoiuhgfyyyyyyy4}.
\end{proof}

Using Lemma \ref{lem:poiuhgfyyyyyyy4} together with (\ref{eq:lkjhgfcvbnbvcvc2}), we can conclude that
\begin{equation}\label{eq:aqsdxwxcvvbdssszpnfn}
\iota(\mathcal{A}_s(M)) \subseteq \mathcal{E}_{2}\left(\left\{\mu_{n}^{(2)}\right\}_{n \in\mathbb{N}^*}\right).
\end{equation}
The proof is finalized by combining (\ref{eq:aqsdxwxcvvbdssszpnfn0}) and (\ref{eq:aqsdxwxcvvbdssszpnfn}).

\subsubsection{Proof of Lemma \ref{lem:Comparison with ellipsoidbddiskup}}\label{sec:proofComparison with ellipsoidbddiskup}

Let $f$ be analytic on $D(0;r')$ with Taylor series coefficients $\left\{a_{k} \right\}_{k \in \mathbb{N}}$, and set $\tilde a_k \coloneqq a_k r^k$, for all $k\in\mathbb{N}$. 
First, note that
\begin{equation}\label{eq:poihgfddbvbjqw}
d_{r'} (f, 0)
= \sup_{z \in D(0;r')} \left\lvert f(z) \right\rvert.
\end{equation}
Application of Lemma \ref{lem: Bounds on analytic norm} yields
\begin{align*}
M \left\|\left\{\frac{\tilde a_{k}}{\mu_k} \right\}_{k \in \mathbb{N}}\right\|_{\ell^2}
= \left\|\left\{a_{k} r'^{k}\right\}_{k \in \mathbb{N}}\right\|_{\ell^2}
&\leq d_{r'} (f, 0) \\
&\leq \left\|\left\{a_{k} r'^{k}\right\}_{k \in \mathbb{N}}\right\|_{\ell^1}
= M \left\|\left\{\frac{\tilde a_{k}}{\mu_k} \right\}_{k \in \mathbb{N}}\right\|_{\ell^1}.
\end{align*}
For $f \in \mathcal{A}(r' ; M)$, we hence get
\begin{equation*}
\left\|\left\{\frac{\tilde a_{k}}{\mu_k} \right\}_{k \in \mathbb{N}}\right\|_{\ell^2}
\leq \frac{d_{r'} (f, 0)}{M} 
\leq 1,
\end{equation*}
which, in turn, implies the inclusion relation
\begin{equation}\label{eq:thisit1ll}
\iota \left(\mathcal{A}(r' ; M) \right)
\subseteq \mathcal{E}_{2}(\{\mu_k\}_{k\in\mathbb{N}}).
\end{equation}

Conversely, assuming that $\{\tilde a_{k} \}_{k \in \mathbb{N}} \in \mathcal{E}_{1}(\{\mu_k\}_{k\in\mathbb{N}})$, we have
\begin{equation*}
d_{r'} (f, 0) 
\leq M \left\|\left\{\frac{\tilde a_{k}}{\mu_k} \right\}_{k \in \mathbb{N}}\right\|_{\ell^1}
\leq M,
\end{equation*}
which implies $f\in\mathcal{A}(r' ; M)$ and hence yields 
\begin{equation}\label{eq:thisit2ll}
\mathcal{E}_{1}(\{\mu_k\}_{k\in\mathbb{N}})
\subseteq \iota \left(\mathcal{A}(r' ; M) \right).
\end{equation}
Combining (\ref{eq:thisit1ll}) and (\ref{eq:thisit2ll}) establishes the desired result.

\subsubsection{Proof of Lemma \ref{lem: Asymptotic scaling of volume ratio}}\label{sec:proofAsymptotic scaling of volume ratio}

We start by noting that for $\mathbb{K}=\mathbb{R}$,
it follows from the usual volume formula for unit balls in Euclidean spaces (see e.g. \cite[Eq. (1)]{kempkaVolumesUnitBalls2017}) that 
\begin{equation}\label{eq: ratio volumes 1/d}
\left(V_{p,q}^{{\mathbb{R}}, d}\right)^{\frac{1}{{\sigma_{\mathbb{R}}(d)}}}
= \frac{\Gamma(1/p+1)}{\Gamma(1/q+1)}\left(\frac{\Gamma(d/q+1)}{\Gamma(d/p+1)}\right)^{\frac{1}{d}},
\end{equation}
where $\Gamma$ denotes the Euler gamma function.
By the Stirling formula, we have 
\begin{equation*}
\Gamma(d/q+1) = \sqrt{\frac{2\pi d}{q}}\left(\frac{d}{qe}\right)^{d/q}\left(1+O_{d\to\infty}(1/d)\right),
\end{equation*}
and 
\begin{equation*}
\Gamma(d/p+1) = \sqrt{\frac{2\pi d}{p}}\left(\frac{d}{pe}\right)^{d/p}\left(1+O_{d\to\infty}(1/d)\right).
\end{equation*}
Taking ratios yields
\begin{equation}\label{eq:kljhgfdqesrfetfgte}
\frac{\Gamma(d/q+1)}{\Gamma(d/p+1)}
= \frac{p^{d/p+1/2}}{q^{d/q+1/2}} \left(\frac{d}{e}\right)^{d(1/q-1/p)} \left(1+O_{d\to\infty}(1/d)\right).
\end{equation}
Inserting (\ref{eq:kljhgfdqesrfetfgte}) into (\ref{eq: ratio volumes 1/d}) gives
\begin{equation}\label{eq: ratio volumes 1/d3}
\left(V_{p,q}^{{\mathbb{R}}, d}\right)^{\frac{1}{{\sigma_{\mathbb{R}}(d)}}}
=  \frac{\Gamma(1/p+1) \, p^{1/p}}{\Gamma(1/q+1) \, q^{1/q} \, e^{(1/q-1/p)} } {d}^{(1/q-1/p)} \left(1+O_{d\to\infty}(1/d)\right).
\end{equation}
Taking the logarithm on both sides of (\ref{eq: ratio volumes 1/d3}) yields the desired result in the case $\mathbb{K}=\mathbb{R}$.
Likewise, for $\mathbb{K}=\mathbb{C}$, we use the corresponding volume formula (see e.g. \cite[Proposition~3.2.1]{edmundsFunctionSpacesEntropy1996}) to get 
\begin{equation}\label{eq: ratio volumes 1/d2}
\left(V_{p,q}^{{\mathbb{C}}, d}\right)^{\frac{1}{{\sigma_{\mathbb{C}}(d)}}}
= \left( \frac{\Gamma(2/p+1)}{\Gamma(2/q+1)}\right)^{\frac{1}{2}}  \left(\frac{\Gamma(2d/q+1)}{\Gamma(2d/p+1)}\right)^{\frac{1}{2d}} .
\end{equation}
Following the same steps as in the case $\mathbb{K}=\mathbb{R}$ then yields the desired result.

\subsubsection{Proof of Lemma \ref{lem:kljhugaqff}}\label{sec:proofkljhugaqff}

We first note that
\begin{align}
    \delta(d) 
    =& \, \frac{1}{d}\sum_{n=1}^{d} \left(\psi(d)- \psi(n) \right) \label{eq:asympbhdlt1}\\
    =& \, \frac{c}{d}\sum_{n=1}^{d} \left[d(\log (d) -c')- n(\log (n)-c') \right]\nonumber\\
    =&  \, c \, d\log (d) - \frac{c}{d}\sum_{n=1}^{d} n\log (n) -  \frac{c \, c' \, d}{2} + O_{d\to \infty}\left(1\right).\label{eq:asympbhdlt2}
\end{align}
Sum-integral comparisons now yield the lower bound 
\begin{align*}
\sum_{n=1}^{d} n\log (n ) 
\geq \int_{1}^d t \log (t) \, dt 
= & \, \frac{1}{2}\left[d^2\log(d) - \frac{d^2}{2\ln(2)} + \frac{1}{2\ln(2)}\right]\\
=& \, \frac{d^2\log(d)}{2} - \frac{d^2}{4\ln(2)} + O_{d\to \infty}\left(1\right)
\end{align*}
and the upper bound 
\begin{align*}
\sum_{n=1}^{d} n \log(n)  
\leq \int_{2}^{d+1} t \log (t) \, dt 
=& \,  \frac{1}{2}\left[(d+1)^2\log(d+1) - \frac{(d+1)^2}{2\ln(2)} - \frac{2}{\ln(2)} \right]\\
=& \, \frac{d^2\log(d)}{2}- \frac{d^2}{4\ln(2)} + O_{d\to \infty}\left(d \log(d) \right).
\end{align*}
Combining these bounds, we obtain
%lower and upper bounds just derived, we get
\begin{equation*}
\sum_{n=1}^{d} n \log(n) 
= \frac{d^2\log(d)}{2} - \frac{d^2}{4 \ln(2)} + O_{d\to \infty}\left(d\log(d)\right),
\end{equation*}
which, when inserted into (\ref{eq:asympbhdlt1})--(\ref{eq:asympbhdlt2})
yields the desired result
\begin{equation*}
    \delta(d)
    = \frac{cd\log(d)}{2} + \frac{c  \, d \, (\ln^{-1}(2)-2\, c')}{4} + O_{d\to \infty}\left(\log(d)\right).
\end{equation*}

\subsubsection{Proof of Lemma \ref{lem:kjhgvhjuhuhygatf}}\label{sec:proofkjhgvhjuhuhygatf}

Starting from the expression 
\begin{equation*}
\left(\psi\circ \psi^{-1}\right)(u) = u,
\quad \text{for all } u \in (0, \infty) \text{ and with } \psi(t) = ct(\log (t)-c'),
\end{equation*}
it follows that 
\begin{equation}\label{eq:kjhgvhjuhuhygatf}
u 
= c\, \psi^{(-1)}(u)\left(\log \left(\psi^{(-1)}(u)\right)-c'\right).
\end{equation}
Setting 
\begin{equation}\label{eq:kjhgvhjuhuhygatf2}
x \coloneqq \ln \left(\psi^{(-1)}(u)\right),
\end{equation}
we can rewrite (\ref{eq:kjhgvhjuhuhygatf}) according to
\begin{equation}\label{eq:equtobeinvlambt}
u
= c\, e^{x} \left(\frac{x}{\ln (2)}-c'\right), 
\quad \text{or equivalently,}
\quad x=
\frac{u \ln (2)}{c}e^{-x} + c'  \ln (2).
\end{equation}
Next, applying Lemma \ref{lem:equsolLambert} with 
\begin{equation*}
a= \frac{u \ln (2)}{c}
\quad \text{and}
\quad
b= c'  \ln (2),
\end{equation*}
%it follows that 
the solution of (\ref{eq:equtobeinvlambt}) can be expressed as
\begin{equation*}
x
= c'  \ln (2) + W\left( \frac{u \ln (2)}{c}e^{-c'  \ln (2) }\right),
\end{equation*}
which, when combined with the definition (\ref{eq:kjhgvhjuhuhygatf2}), yields the desired result 
\begin{equation*}
\psi^{(-1)}(u)
= \exp\left\{c'  \ln (2) + W\left( \frac{u \ln (2)}{c}e^{-c'  \ln (2) }\right)\right\}.
\end{equation*}

\subsubsection{Proof of Lemma \ref{lem:lkjodihuzehchehjz2}}\label{sec:prooflkjodihuzehchehjz2}

We start with a result on the asymptotic behavior of $\psi^{(-1)}$.

\begin{lemma}\label{lem:lkjodihuzehchehjz}
Let $c,c' > 0$. The inverse of  $\psi \colon t \mapsto ct(\log (t)-c')$ satisfies
\begin{equation*}
\psi^{(-1)}(u)
=\frac{u}{c\log\left(u \right)}+ \frac{u\log^{(2)}(u)}{c\log^2 (u)} + o_{u\to \infty } \left(\frac{u\log^{(2)}(u)}{\log^2 (u)}\right).
\end{equation*}
\end{lemma}

\begin{proof}[Proof.]
See Appendix~\ref{sec:prooflkjodihuzehchehjz}.
\end{proof}

\noindent
We note that by Lemma \ref{lem:lkjodihuzehchehjz}
\begin{align}
\psi^{(-1)}(u)^2 \label{eq:nbvnsapldjrijifrjv1}
&=\left[\frac{u}{c\log (u)} \right]^2\left[1 + \frac{\log^{(2)}(u)}{\log(u)} + o_{u\to \infty}\left(\frac{\log^{(2)}(u)}{\log(u)}\right)\right]^2 \\
&=\left[\frac{u}{c\log (u)} \right]^2
\left[1 + 2 \, \frac{\log^{(2)}(u)}{\log(u)} + o_{u\to \infty}\left(\frac{\log^{(2)}(u)}{\log(u)}\right)\right],\label{eq:nbvnsapldjrijifrjv2}
\end{align}
as well as
\begin{align}
\log\left(\psi^{(-1)}(u)\right)
&=\log \left[\frac{u}{c\log (u)} \right]+\log \left[1 + \frac{\log^{(2)}(u)}{\log(u)} + o_{u\to \infty}\left(\frac{\log^{(2)}(u)}{\log(u)}\right)\right]\label{eq:nbvnsapldjrijifrjv3}\\
&=\log (u) - \log^{(2)}(u) - \log(c)
+ \frac{\log^{(2)}(u)}{\log(u)} + o_{u\to \infty}\left(\frac{\log^{(2)}(u)}{\log(u)}\right)\nonumber \\
&=\log (u) \left[1-\frac{\log^{(2)}(u)}{\log(u)} + o_{u\to \infty}\left(\frac{\log^{(2)}(u)}{\log(u)}\right)\right].\label{eq:nbvnsapldjrijifrjv4}
\end{align}
We further observe that 
\begin{align}
&\left[1 + 2 \, \frac{\log^{(2)}(u)}{\log(u)} + o_{u\to \infty}\left(\frac{\log^{(2)}(u)}{\log(u)}\right)\right]
\left[1-\frac{\log^{(2)}(u)}{\log(u)} + o_{u\to \infty}\left(\frac{\log^{(2)}(u)}{\log(u)}\right)\right]\label{eq:nbvnsapldjrijifrjv5}\\
& = 1 +  \frac{\log^{(2)}(u)}{\log(u)} + o_{u\to \infty}\left(\frac{\log^{(2)}(u)}{\log(u)}\right).\label{eq:nbvnsapldjrijifrjv6}
\end{align}
Combining (\ref{eq:nbvnsapldjrijifrjv1})--(\ref{eq:nbvnsapldjrijifrjv2}), (\ref{eq:nbvnsapldjrijifrjv3})--(\ref{eq:nbvnsapldjrijifrjv4}), and (\ref{eq:nbvnsapldjrijifrjv5})--(\ref{eq:nbvnsapldjrijifrjv6}), we obtain
\begin{align*}
\psi^{(-1)}(u)^2 \log\left(\psi^{(-1)}(u)\right)
=& \left[\frac{u}{c\log (u)} \right]^2
\log (u)
\left[1 +  \frac{\log^{(2)}(u)}{\log(u)} + o_{u\to \infty}\left(\frac{\log^{(2)}(u)}{\log(u)}\right)\right] \\
=& \frac{u^2}{c^2\log (u)} 
\left[1 +  \frac{\log^{(2)}(u)}{\log(u)} + o_{u\to \infty}\left(\frac{\log^{(2)}(u)}{\log(u)}\right)\right],
\end{align*}
which finalizes the proof.

\subsubsection{Proof of Lemma \ref{lem:firstpartinclusionexpfct}}\label{sec:prooffirstpartinclusionexpfct}

Given $f \in \mathcal{F}_{\text{exp}}^{A,C}$ with Taylor series coefficients $\{a_k\}_{k\in\mathbb{N}}$, we need to show that %its Taylor series coefficients satisfy
\begin{equation*}
\begin{cases}
    \left\lvert a_0\right\rvert
    \leq  C, \\
    \left\lvert a_k\right\rvert
    \leq C \exp \left\{-k \left(\ln (k) - 1 - \ln (A) \right)\right\}, \quad \text{for }k\geq 1.
\end{cases}
\end{equation*}
From (\ref{eq: popoloihj}) it follows that  
\begin{equation}\label{eq: ufheuivf2}
    \sup_{z \in D(0;R)} \lvert f(z) \rvert \leq C e^{AR}, 
    \quad \forall R>0.
\end{equation}

\begin{lemma}[Cauchy's estimate] \cite[Theorem 10.26]{rudinRealComplexAnalysis1987}\label{lem:cauchyestimate}
Let  $r$ and $M$ be positive real numbers and let
$f$ be analytic on $D(0 ; r)$ with Taylor series coefficients $\{a_k\}_{k \in \mathbb{N}}$.
If $\lvert f( z) \rvert \leq M$, for all $z \in D(0 ;   r)$, then
\begin{equation}\label{eq: bound an coefs}
  \lvert a_{k} \rvert  \leq \frac{M}{r^{k}},  
  \quad \text{for all } \, k \in \mathbb{N}.
\end{equation}
\end{lemma}

\noindent
Applying Cauchy's estimate (\ref{eq: bound an coefs}) together with (\ref{eq: ufheuivf2}), we get
\begin{equation*}
    \lvert a_k\rvert  \leq \frac{C e^{AR}}{R^k}, 
    \quad \forall R>0.
\end{equation*}
As this bound holds for all values of $R>0$, it holds in particular for the value minimizing the upper bound, or, more specifically, we have
\begin{equation}\label{eq:fksdjoiehfdjvwwzoporirieieze}
    \left\lvert a_k\right\rvert
    \leq \inf_{R > 0} C \exp\left\{AR - k\ln (R)\right\}.
\end{equation}
For $k=0$, we readily get 
\begin{equation*}
\left\lvert a_0\right\rvert
    \leq  C,
\end{equation*}
as desired.
It remains to consider the case $k\geq 1$.
A straightforward calculation reveals that the infimum in (\ref{eq:fksdjoiehfdjvwwzoporirieieze}) is attained at
\begin{equation*}
    R_k
    \coloneqq \frac{k}{A}>0,
\end{equation*}
which, upon insertion into (\ref{eq:fksdjoiehfdjvwwzoporirieieze}), yields
\begin{align}
    \left\lvert a_k\right\rvert
    &\leq C \exp\left\{A R_k - k\ln (R_k)\right\} \\
    &= C \exp \left\{-k \left(\ln( k) - 1 - \ln (A) \right)\right\}.
\end{align}
This concludes the proof.

\subsubsection{Proof of Lemma \ref{lem:secdpartinclusionexpfct}}\label{sec:proofsecdpartinclusionexpfct}

For the proof of Lemma \ref{lem:secdpartinclusionexpfct}, we shall need the following auxiliary result.
\begin{lemma}\label{lem:studyvariationscompfctbis}
It holds that
\begin{equation*}
1
\leq \sup_{k \in \mathbb{N}^*} \frac{\sqrt{2 \pi k}\, e^{\frac{1}{12k}}}{2^k}   < \infty.
\end{equation*}
\end{lemma}

\begin{proof}[Proof.]
The function
\begin{equation*}
     g\colon t 
     \mapsto \frac{\sqrt{2 \pi t}\, e^{\frac{1}{12t}}}{2^t}  
\end{equation*}
is continuous on $[1, \infty)$ and satisfies
\begin{equation*}
\lim_{t \to \infty} g(t)=0.
\end{equation*}
Therefore, $g$ is bounded on $[1, \infty)$, i.e., 
\begin{equation*}
\sup_{t\in [1, \infty)} g(t) < \infty.
\end{equation*}
This obviously implies 
\begin{equation*}
\sup_{k \in \mathbb{N}^*} g(k) < \infty.
\end{equation*}
We finally observe that 
\begin{equation*}
\sup_{k \in \mathbb{N}^*} g(k)
\geq g(1) 
= \sqrt{ \frac{\pi}{2} }\, e^{\frac{1}{12}}
\geq 1, 
\end{equation*}
which concludes the proof.
\end{proof}

\noindent
We first note that $\tilde C$ is well-defined owing to Lemma \ref{lem:studyvariationscompfctbis}.
Next, consider a sequence $\{a_k\}_{k\in\mathbb{N}}$ satisfying
\begin{equation}\label{eq: epapapafbgfur}
\begin{cases}
    \left\lvert a_0\right\rvert
    \leq  {\tilde C},\\
    \left\lvert a_k\right\rvert
    \leq {\tilde C} \exp \left\{-k \left(\ln (k) - 1 - \ln \left({\tilde A}\right) \right)\right\},  \quad \text{for }k\geq 1.
\end{cases}
\end{equation}
Under the definition
\begin{equation*}
f(z)
= \sum_{k=0}^\infty a_k z^k,
\end{equation*}
we need to show that
\begin{equation*}
    \lvert f(z) \rvert \leq C e^{A\lvert z \rvert},
    \quad \text{for all } z\in \mathbb{C}.
\end{equation*}
To this end, fix $z\in \mathbb{C}$ and start with the following chain of inequalities
\begin{align}
\left\lvert f(z)\right\rvert
    = \left\lvert\sum_{k=0}^\infty a_k z^k \right\rvert 
    &\leq \sum_{k=0}^\infty \lvert a_k \rvert \lvert z \rvert^k  \label{eq: mloivbkimpp1} \\
    & \stackrel{(\ref{eq: epapapafbgfur})}{\leq}{\tilde C}+\sum_{k=1}^\infty {\tilde C} \exp \left\{-k \left(\ln (k) - 1 - \ln \left({\tilde A}\right) \right)\right\} \lvert z \rvert^k \nonumber\\
    & = {\tilde C}  \left(1+\sum_{k=1}^\infty \frac{e^k {\tilde A}^k\lvert z \rvert^k}{k^k}\right). \label{eq: mloivbkimpp3}
\end{align}
Applying Stirling's inequality, we obtain
\begin{align}\label{eq:stirling}
    \frac{e^k {\tilde A}^k\lvert z \rvert^k}{k^k}
    & \leq \frac{A^k\lvert z \rvert^k}{k!}\frac{\sqrt{2 \pi k}\, e^{\frac{1}{12k}}}{2^k},
    \quad \text{for all } k\geq 1.
\end{align}
Using (\ref{eq:stirling}) in (\ref{eq: mloivbkimpp1})--(\ref{eq: mloivbkimpp3}), we get
\begin{align*}
    \left\lvert f(z)\right\rvert
    &\leq {\tilde C}  \left(1+\sum_{k=1}^\infty \frac{A^k\lvert z \rvert^k}{k!}\frac{\sqrt{2 \pi k}\, e^{\frac{1}{12k}}}{2^k} \right) \\
    &\leq {\tilde C}  \left(1+ \sup_{k \in \mathbb{N}^*} \frac{\sqrt{2 \pi {k}}\, e^{\frac{1}{12{k}}}}{2^{k}} \sum_{k=1}^\infty \frac{A^k\lvert z \rvert^k}{k!} \right) \\
    &\leq {\tilde C} \left[\sup_{k \in \mathbb{N}^*} \frac{\sqrt{2 \pi {k}}\, e^{\frac{1}{12{k}}}}{2^{k}}\right]  \sum_{k=0}^\infty \frac{A^k\lvert z \rvert^k}{k!} 
    = C e^{A\lvert z \rvert},
\end{align*}
where the last inequality relies on the property
\begin{equation*}
\sup_{k \in \mathbb{N}^*} \frac{\sqrt{2 \pi {k}}\, e^{\frac{1}{12{k}}}}{2^{k}} 
\geq 1
\end{equation*}
established in Lemma \ref{lem:studyvariationscompfctbis}.
This concludes the proof.

\subsubsection{Proof of Lemma \ref{lem:lemboundtriangleineqlike}}\label{sec:prooflemboundtriangleineqlike}

Let $\{x^1, \dots, x^N\}$ be a $\rho$-covering of $\mathcal{E}_p^d$ with respect to the $\|\cdot\|_q$-norm.
We introduce the vector obtained by completing the coordinates of $x^i$ with infinitely many zeros as
\begin{equation*}
\bar x^i
\coloneqq (x^i_1, \dots, x^i_d, 0, \dots), \quad \text{for all } i \in \{1, \dots, N\}.
\end{equation*}
The proof is then effected by showing that $\{\bar x^1, \dots, \bar x^N\}$ is a $\bar \rho$-covering of $\mathcal{E}_p$.
To this end, take $\bar x\in \mathcal{E}_p$ and define the vector
\begin{equation*}
x
\coloneqq (\bar x_1, \dots, \bar x_d)
\end{equation*}
obtained by retaining the first $d$ components of $\bar x$.
By definition, $x \in \mathcal{E}_p^d$. As $\{x^1, \dots, x^N\}$ is a $\rho$-covering of $\mathcal{E}_p^d$ in the $\|\cdot\|_q$-norm, there exists an index $j\in \{1, \dots, N\}$ such that 
\begin{equation}\label{eqlkiudizerufhhjhjs0}
\left\|x-x^j\right\|_q\leq \rho.
\end{equation}
%where $q\in[1,\infty]$.
Now, consider the case $q \in [1, \infty)$, and observe that 
\begin{align}\label{eqlkiudizerufhhjhjs}
\left\|\bar x- \bar x^j\right\|_q^q
= \sum_{n=1}^d \left\lvert\bar x_n - \bar x_n^j \right\rvert^q +\sum_{n=d+1}^\infty \left\lvert\bar x_n\right\rvert^q.
\end{align}
The first term on the right-hand-side of (\ref{eqlkiudizerufhhjhjs}) equals $\|x-x^j\|_q^q$ and can hence be upper-bounded by $\rho^q$ owing to (\ref{eqlkiudizerufhhjhjs0}).
We next observe that, thanks to $\bar x \in \mathcal{E}_p$, $\lvert\bar x_n \rvert \leq \mu_n$, for $n \in \mathbb{N}^*$.
This immediately implies 
\begin{equation}\label{eq:firstboundonresidualsum}
\sum_{n=d+1}^\infty \left\lvert\bar x_n\right\rvert^q
\leq \sum_{n=d+1}^\infty \mu_n^q
= \mu_{d+1}^q \sum_{n=d+1}^\infty \left(\frac{\mu_n}{\mu_{d+1}}\right)^q.
\end{equation}
We will need to further upper-bound the right-hand-side in (\ref{eq:firstboundonresidualsum}), which will be effected through the following result.

\begin{lemma}\label{lem:decayrelsemiaxesratiol}
Let $\psi$ be a decay rate function.
Then, for every sequence $\{\mu_n\}_{n\in \mathbb{N}^*}$ of $\psi$-exponentially decaying semi-axes, there exist a positive real number $c$ and an $n^*\in \mathbb{N}^*$ such that, for all pairs of integers $n$ and $m$ satisfying $n\geq m \geq n^*$, it holds that
\begin{equation*}
\frac{\mu_n}{\mu_{m}} 
\leq 2^{c(m-n)}.
\end{equation*}
\end{lemma}

\begin{proof}[Proof.]
See Appendix~\ref{sec:proofdecayrelsemiaxesratiol}.
\end{proof}

\noindent
With $n^*$ according to Lemma \ref{lem:decayrelsemiaxesratiol}, setting $d^*\coloneqq n^*-1$, it follows from Lemma \ref{lem:decayrelsemiaxesratiol} that
\begin{equation}\label{eq:secondboundonresidualsum}
\sum_{n=d+1}^\infty \left(\frac{\mu_n}{\mu_{d+1}}\right)^q
\leq \sum_{n=d+1}^\infty  2^{cq(d+1-n)}
= \sum_{n=0}^\infty 2^{-cqn}
= \frac{2^{cq}}{2^{cq}-1},
\end{equation}
for all $d\geq d^*$.
Upon defining the constant 
\begin{equation*}
K \coloneqq \frac{2^{cq}}{2^{cq}-1}
\geq 1,
\end{equation*}
we get, from (\ref{eq:firstboundonresidualsum}) and (\ref{eq:secondboundonresidualsum}), the bound 
\begin{equation*}
\sum_{n=d+1}^\infty \left\lvert\bar x_n\right\rvert^q
\leq K \mu_{d+1}^q.
\end{equation*}
Consequently, (\ref{eqlkiudizerufhhjhjs}) becomes 
\begin{align*}
\left\|\bar x- \bar x^j\right\|_q^q
\leq \rho^q + K \mu_{d+1}^q 
=\bar \rho^q, 
\end{align*}
which shows that $\{\bar x^1, \dots, \bar x^N\}$ is a $\bar \rho$-covering of $\mathcal{E}_p$.
Finally noting that $K$ does not depend on $d$ as required, the result is established for $q\in [1, \infty)$.
For $q=\infty$, we have 
\begin{align}\label{eqlkiudizerufhhjhjsinf}
\left\|\bar x- \bar x^j\right\|_q
= \sup_{n\in \mathbb{N}^*} \left\lvert\bar x_n - \bar x_n^i\right\rvert
= \max\left\{\max_{n=1,\dots, d} \left\lvert\bar x_n - \bar x_n^i\right\rvert, \sup_{n\geq d+1}\left\lvert\bar x_n\right\rvert\right\}.
\end{align}
As $\max_{n=1,\dots, d} \left\lvert\bar x_n - \bar x_n^i\right\rvert \leq \rho$ thanks to (\ref{eqlkiudizerufhhjhjs0}) and $\sup_{n\geq d+1}\left\lvert\bar x_n\right\rvert \leq \mu_{d+1}$, it follows that $\{\bar x^1, \dots, \bar x^N\}$ is a $\max\left\{\rho, \mu_{d+1}\right\}$-covering of $\mathcal{E}_p$.

\subsubsection{Proof of Lemma \ref{lem:poiuhgfyyyyyyy3}}\label{sec:proofpoiuhgfyyyyyyy3}

Consider the analytic extension of $f$ to $\mathcal{S}$ (recall Definition~\ref{def:perfctonstrip}), i.e.,
\begin{equation*}
g \colon z \in \mathcal{S} \mapsto \sum_{k=-\infty}^{\infty} a_k e^{ikz},
\end{equation*}
and write $z=x+iy$, with $\lvert y \rvert < s$.
Then, it follows from 
\begin{equation*}
\sum_{k=-\infty}^{\infty} \left \lvert  a_k e^{ik(x+iy)}\right \rvert
\leq \sum_{k=-\infty}^{\infty} \left \lvert  a_k \right \rvert e^{s \lvert k\rvert}
= M \left \| \left\{\frac{a_k}{ \mu_{\lvert k \rvert}} \right\}_{k \in \mathbb{Z}} \right \|_{\ell^1(\mathbb{Z})}
\leq M,
\end{equation*}
that the infinite series defining $g$ is absolutely and uniformly convergent on $\mathcal{S}$, which implies analyticity of $g$ on $\mathcal{S}$.
In particular, as $g=f$ on $\mathbb{R}\subseteq \mathcal{S}$, we have that $g$ is the analytic continuation of $f$ to $\mathcal{S}$.
It remains to prove that $\sup_{z \in \mathcal{S}} \lvert g(z)\rvert\leq M$.
To this end, fix $s'\in \mathbb{R}$ such that $\lvert s' \rvert < s$ and apply Lemma \ref{lem:poiuhgfyyyyyyy2} which yields
\begin{align}
\lvert g(x+i s') \rvert
= \left \lvert \sum_{k=-\infty}^{\infty} a_k e^{ikx - ks'} \right \rvert
&\leq \sum_{k=-\infty}^{\infty} \left \lvert  a_k \right \rvert e^{s \lvert k\rvert} \label{eq:kjuhgfqqqed} \\
&= M \left \| \left\{\frac{a_k}{ \mu_{\lvert k \rvert}} \right\}_{k \in \mathbb{Z}} \right \|_{\ell^1(\mathbb{Z})}
\leq M.\label{eq:kjuhgfqqqed2}
\end{align}
Taking the supremum in (\ref{eq:kjuhgfqqqed})--(\ref{eq:kjuhgfqqqed2}), according to
\begin{equation*}
\sup_{z \in \mathcal{S}} \,  \lvert g(z) \rvert
=\sup_{x \in \mathbb{R}} \sup_{\lvert s' \rvert<s}  \lvert f(x+i s') \rvert
\leq M,
\end{equation*}
concludes the proof.

\subsubsection{Proof of Lemma \ref{lem:poiuhgfyyyyyyy4}}\label{sec:proofpoiuhgfyyyyyyy4}

Fix $s'\in (0 ; s)$ and define
\begin{equation*}
\mu_n' \coloneqq M e^{-s' n},
\quad \text{for all } n \in \mathbb{N}.
\end{equation*}
We start with the observation that 
\begin{align}
M^2 \left \| \left\{\frac{a_k}{\mu'_{\lvert k \rvert}} \right\}_{k \in \mathbb{Z}} \right \|_{\ell^2(\mathbb{Z})}^2
&= \sum_{k=-\infty}^{\infty} \left \lvert  a_k \, e^{\lvert k\rvert s'} \right \rvert ^2 \label{eqdddfgvvvchbs1} \\
&\leq \sum_{k=-\infty}^{\infty} \left \lvert  a_k \, e^{-k s'} \right \rvert ^2 + \sum_{k=-\infty}^{\infty} \left \lvert  a_k \, e^{k s'} \right \rvert^2. \label{eqdddfgvvvchbs2}
\end{align}
Next, using Parseval's identity combined with Lemma \ref{lem:poiuhgfyyyyyyy2}, we get
\begin{equation}\label{eq:boundonsumsprime}
\sum_{k=-\infty}^{\infty} \left \lvert  a_k \, e^{- k s'} \right \rvert ^2
= \frac{1}{2\pi}  \int_{0}^{2\pi} \left \lvert f(x+is')  \right \rvert ^2 \, dx 
\leq \sup_{z \in \mathcal{S}} \, \lvert f(z) \rvert^2 \leq M^2.
\end{equation}
Therefore, taking the limit $s' \to s$ in (145) shows that the non-decreasing 
sequence $\bigl\{ \sum_{k=-N}^{N} |a_k e^{-ks}|^2 \bigr\}_{N \in \mathbb{N}^*}$ 
is bounded by $M^2$.
%Therefore, upon taking the limit $s'\to s$ in (\ref{eq:boundonsumsprime}), the sequence of partial sums $\{\sum_{k=-N}^{N} \left \lvert  a_k \, e^{- k s} %\right \rvert ^2\}_{N\in \mathbb{N}^*}$
%is non-decreasing and bounded by $M^2$. 
By \cite[Theorem 3.14]{rudin1953principles}, it hence converges to a limit which is also bounded by $M^2$, i.e., we have 
\begin{equation}\label{eq:thisisniceresulthe1}
\sum_{k=-\infty}^{\infty} \left \lvert  a_k \, e^{- k s} \right \rvert ^2
\leq M^2.
\end{equation}
Likewise, one can show that 
\begin{equation}\label{eq:thisisniceresulthe2}
\sum_{k=-\infty}^{\infty} \left \lvert  a_k \, e^{k s} \right \rvert ^2
 \leq M^2.
\end{equation}
Using (\ref{eq:thisisniceresulthe1}) and (\ref{eq:thisisniceresulthe2}) in (\ref{eqdddfgvvvchbs1})--(\ref{eqdddfgvvvchbs2}) hence yields
\begin{equation*}
\left \| \left\{\frac{a_k}{\sqrt{2}\mu_{\lvert k \rvert}} \right\}_{k \in \mathbb{Z}} \right \|_{\ell^2(\mathbb{Z})}
\leq 1,
\end{equation*}
which concludes the proof.

\subsubsection{Proof of Lemma \ref{lem:lkjodihuzehchehjz}}\label{sec:prooflkjodihuzehchehjz}

We start by observing that
\begin{align}
& \ln\left(\frac{u \ln (2) \, e^{-c'  \ln (2) } }{c \ln\left(\frac{u \ln (2) \, e^{-c'  \ln (2) } }{c}\right)}\right) \nonumber \\  
& = -c'\ln (2) + \ln\left(\frac{u}{c \left[\log\left(u\right) + O_{u\to \infty}(1) \right]}\right) \label{eq:apdozdjhchczdhezjj1}\\
& = -c'\ln (2) + \ln\left(\frac{u}{c \log\left(u\right)}\left[1+  O_{u\to \infty}\left(\frac{1}{\log(u)}\right)\right]\right) \nonumber\\
& = -c'\ln (2) + \ln\left(\frac{u}{c \log\left(u\right) }\right) + O_{u\to \infty}\left(\frac{1}{\log(u)}\right).\label{eq:apdozdjhchczdhezjj2}
\end{align}
Furthermore, we have %the asymptotic relation
\begin{align}
\frac{\ln^{(2)}\left(\frac{u \ln (2) \, e^{-c'  \ln (2) } }{c}\right)}{\ln\left(\frac{u \ln (2) \, e^{-c'  \ln (2) } }{c}\right)}
& = \frac{\ln^{(2)}\left(u\right) + O_{u\to \infty}\left(\frac{1}{\ln(u)}\right) }{\ln\left(u\right) + O_{u\to \infty}(1)} \label{eq:apdozdjhchczdhezjj3}\\
& = \frac{\log^{(2)}\left(u\right)}{\log\left(u\right)} + o_{u\to \infty}\left(\frac{\log^{(2)}\left(u\right)}{\log\left(u\right)}\right). \label{eq:apdozdjhchczdhezjj4}
\end{align}
Combining (\ref{eq:apdozdjhchczdhezjj1})--(\ref{eq:apdozdjhchczdhezjj2}) and (\ref{eq:apdozdjhchczdhezjj3})--(\ref{eq:apdozdjhchczdhezjj4}) with (\ref{eq:lbtseriesexp}) in Lemma \ref{lem:lbtseriesexp}, we obtain
\begin{align}
&W\left( \frac{u \ln (2)}{c}e^{-c'  \ln (2) }\right)\label{eq:apdozdjhchczdhezjj5}\\
& = \ln\left(\frac{u \ln (2) \, e^{-c'  \ln (2) } }{c \ln\left(\frac{u \ln (2) \, e^{-c'  \ln (2) } }{c}\right)}\right)  
+ \frac{\ln^{(2)}\left(\frac{u \ln (2) \, e^{-c'  \ln (2) } }{c}\right)}{\ln\left(\frac{u \ln (2) \, e^{-c'  \ln (2) } }{c}\right)} \left(1 + o_{u\to \infty}(1)\right) \nonumber\\
& = -c'\ln (2) + \ln\left(\frac{u}{c \log\left(u\right) }\right) + \frac{\log^{(2)}\left(u\right)}{\log\left(u\right)} + o_{u\to \infty}\left(\frac{\log^{(2)}\left(u\right)}{\log\left(u\right)}\right) \nonumber\\
& = -c'\ln (2) + \ln\left(\frac{u}{c \log\left(u\right) }\right) 
+ \ln\left\{1 + \frac{\log^{(2)}\left(u\right)}{\log\left(u\right)} + o_{u\to \infty}\left(\frac{\log^{(2)}\left(u\right)}{\log\left(u\right)}\right)\right\} \nonumber \\
& = -c'\ln (2) + \ln\left\{\frac{u}{c \log\left(u\right)} + \frac{u\log^{(2)}\left(u\right)}{c\log^2\left(u\right)} + o_{u\to \infty}\left(\frac{u\log^{(2)}\left(u\right)}{\log^2\left(u\right)}\right)\right\}. \label{eq:apdozdjhchczdhezjj6}
\end{align}
Using (\ref{eq:apdozdjhchczdhezjj5})--(\ref{eq:apdozdjhchczdhezjj6}) in the expression for $\psi^{(-1)}$ according to Lemma \ref{lem:kjhgvhjuhuhygatf}, %now 
yields
\begin{align*}
\psi^{(-1)}(u)
=& \,  \exp\left\{c'  \ln (2) + W\left( \frac{u \ln (2)}{c}e^{-c'  \ln (2) }\right)\right\}\\
=& \, \frac{u}{c \log\left(u\right) } + \frac{u\log^{(2)}\left(u\right)}{c\log^2\left(u\right)} + o_{u\to \infty}\left(\frac{u\log^{(2)}\left(u\right)}{\log^2\left(u\right)}\right),
\end{align*}
thereby concluding the proof.

\subsubsection{Proof of Lemma \ref{lem:decayrelsemiaxesratiol}}\label{sec:proofdecayrelsemiaxesratiol}

Let $t^*$ be the parameter of $\psi$ and set
\begin{equation*}
n^* \coloneqq \lceil t^* \rceil +1.
\end{equation*}
Then, from (\ref{eq:assumptionspsidecayfct}), we get
\begin{equation*}
\frac{\psi(m)}{m}
\leq \frac{\psi(n)}{n},
\end{equation*}
which directly implies 
\begin{equation*}
\frac{\psi(m)- \psi(n)}{m}
\leq \psi(n)\left(\frac{1}{n}-\frac{1}{m}\right),
\end{equation*}
or, equivalently,
\begin{equation}\label{eq:joihugfghszxxxx}
\psi(n)-\psi(m) 
\geq m \, \psi(n)\left(\frac{1}{m}-\frac{1}{n}\right).
\end{equation}
Furthermore, we have 
\begin{equation*}
\frac{\psi(n)}{n} 
\geq \frac{\psi(n^*)}{n^*} 
\eqqcolon c > 0,
\end{equation*}
which, when used in (\ref{eq:joihugfghszxxxx}), yields
\begin{equation*}
\psi(n)-\psi(m) 
\geq c\, m \, n\left(\frac{1}{m}-\frac{1}{n}\right)
= c\, \left(n-m\right).
\end{equation*}
Finally, invoking (\ref{eq: that's some big boy exp}) results in 
\begin{equation*}
\frac{\mu_n}{\mu_{m}} 
= \exp\left\{-\ln(2)(\psi(n)-\psi(m))\right\} 
\leq 2^{c(m-n)},
\end{equation*}
which is the desired result.

\subsubsection{Statement and Proof of Lemma \ref{lem:firstpropeffdim}}\label{sec:prooffirstpropeffdim}

\begin{lemma}\label{lem:firstpropeffdim}
Let $p, q \in [1, \infty]$ and let  $\mathcal{E}_p$ be the infinite-dimensional ellipsoid with exponentially decaying semi-axes $\{\mu_n\}_{n \in \mathbb{N}^*}$.
Then, the effective dimension 
\begin{equation*}
d_{\varepsilon} = \min \left\{ k \in \mathbb{N}^* \mid   
 N \left(\varepsilon ; \mathcal{E}_p, \|\cdot\|_q \right)^{\frac{1}{{\sigma_{\mathbb{K}}(k)}}} \, \mu_k
\leq  \kplus \,  {\sigma_{\mathbb{K}}(k)}^{(\frac{1}{q}-\frac{1}{p})} \, \bar \mu_k  \right\},
\end{equation*}
with $\bar \mu_k$ denoting the geometric mean of the first $k$ semi-axes, for all $k\in \mathbb{N}^*$, is well-defined for all $\varepsilon>0$, and satisfies  $\lim_{\varepsilon \to 0} d_\varepsilon=\infty$. Here, $\kplus$ is the constant defined in Theorem~\ref{thm: metric entropy of finite ellipsoids}.
\end{lemma}

\begin{proof}[Proof.]
    By Lemma~\ref{lem:lemboundtriangleineqlike}, for every $\varepsilon>0$ the covering number 
$N(\varepsilon;\mathcal{E}_p,\|\cdot\|_q)$ is bounded above by the covering number of a 
finite-dimensional truncation $\mathcal{E}_p^d$. Since $\mathcal{E}_p^d$ is compact in 
$\mathbb{R}^d$, its covering number is finite, and therefore 
$N(\varepsilon;\mathcal{E}_p,\|\cdot\|_q)<\infty$.
%    The covering number $N \left(\varepsilon ; \mathcal{E}_p, \|\cdot\|_q \right)$ is finite for all $\varepsilon>0$; this can be seen, e.g., as a %consequence of Lemma~\ref{lem:lemboundtriangleineqlike}.
To prove that the effective dimension is well-defined, it hence suffices to show that, for all $\varepsilon>0$, there exists a $k\in \mathbb{N}^*$ such that 
\begin{equation}\label{eq:poihugfaaaaaddddwwqss}
N \left(\varepsilon ; \mathcal{E}_p, \|\cdot\|_q \right)
\leq  \left[\kplus \,  {\sigma_{\mathbb{K}}(k)}^{(\frac{1}{q}-\frac{1}{p})} \, \frac{ \bar \mu_k}{\mu_k}\right]^{\sigma_{\mathbb{K}}(k)},
\end{equation}
which, in turn, is guaranteed by
%It suffices to show that 
\begin{equation*}
\left[\kplus \,  {\sigma_{\mathbb{K}}(k)}^{(\frac{1}{q}-\frac{1}{p})} \, \frac{ \bar \mu_k}{\mu_k}\right]^{\sigma_{\mathbb{K}}(k)} 
\xrightarrow{k \to \infty} \infty.
\end{equation*}
It is therefore enough to establish that 
\begin{equation*}
\log \left(\frac{ \bar \mu_k}{\mu_k}\right) + \left(\frac{1}{q}-\frac{1}{p}\right) \log \left(\sigma_{\mathbb{K}}(k)\right) \xrightarrow{k \to \infty} \infty,
\end{equation*}
which holds as a direct consequence of the identity
\begin{equation*}
\log \left(\frac{ \bar \mu_k}{\mu_k}\right) 
\stackrel{(\ref{eq: that's some big boy exp})}{=} \frac{1}{k} \sum_{n=1}^k \left[\psi(k) - \psi(n)\right]
\geq \kappa (k-1),
\end{equation*}
for some $\kappa>0$ independent of $k$, with the inequality relying on the assumption that $\psi$ is a decay rate function and therefore grows at least linearly.

Finally, $\lim_{\varepsilon \to 0} d_\varepsilon=\infty$ is an immediate consequence of the right-hand side in (\ref{eq:poihugfaaaaaddddwwqss}) being finite for all $k\in \mathbb{N}^*$, while the left-hand side goes to infinity as $\varepsilon$ approaches zero.
\end{proof}

\subsubsection{Statement and Proof of Lemma \ref{lem:poiuhgfyyyyyyy2}}\label{sec:proofpoiuhgfyyyyyyy2}

\begin{lemma}\label{lem:poiuhgfyyyyyyy2}
Let $s$ be a positive real number and let $s'\in \mathbb{R}$ be such that $\lvert s' \rvert < s$.
Consider the $2 \pi$-periodic function $f$ analytic on the strip $\mathcal{S}$ of width $s$ from Definition \ref{def:perfctonstrip}.
Then, the Fourier series coefficients of the function
\begin{equation*}
x \mapsto f(x +is')
\quad \text{are given by } \quad 
\left\{ a_k e^{-ks'} \right\}_{k \in \mathbb{Z}}.
\end{equation*}
\end{lemma}

\begin{proof}[Proof.]
We first define
\begin{equation*}
\gamma \colon t \in [0, 1] \mapsto 
\begin{cases}
8 \pi t, \quad  & t \in [0, 1/4], \\
2 \pi + i s' (4t-1), \quad  &  t \in [1/4, 1/2], \\
8 \pi (3/4-t) + is', \quad  & t \in [1/2, 3/4], \\
4 i s' (1-t), \quad  &  t \in [3/4, 1],
\end{cases}
\end{equation*}
which is a closed contour in $\mathcal{S}$ (see Figure \ref{fig:contourgamma}).

\begin{center}
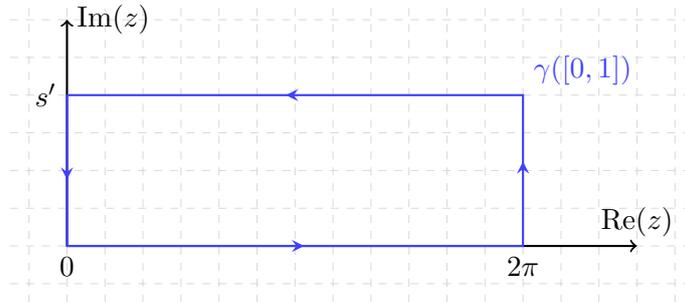

\begin{tikzpicture}[scale=.5, decoration={markings,
    mark=at position .18   with {\arrowreversed[line width=1pt]{stealth}},
    mark=at position .43 with {\arrowreversed[line width=1pt]{stealth}},
    mark=at position .68   with {\arrowreversed[line width=1pt]{stealth}},
    mark=at position .93 with {\arrowreversed[line width=1pt]{stealth}}
  }]

  \draw[help lines, color=gray!30, dashed] (-1.5,-1.5) grid (16.5,6.5);

  \draw[thick, ->] (0,0) -- (15,0) coordinate (xaxis);

  \draw[thick, ->] (0,0) -- (0,6) coordinate (yaxis);

  \node[above] at (xaxis) {$\mathrm{Re}(z)$};

  \node[right]  at (yaxis) {$\mathrm{Im}(z)$};

  \path[draw,blue!75, line width=0.8pt, postaction=decorate] 
        (0,4)   node[left, black] {$s'$}
    --  (12, 4)  node[above right, blue!75] {$\gamma([0, 1])$}
    --  (12, 0)   node[below, black] {$2\pi$}
    --  (0,0)  node[below, black] {$0$} 
    --  (0,4)  ;
\end{tikzpicture}
\captionof{figure}{\small Plot of the contour $\gamma$.}
\label{fig:contourgamma}
\end{center}

\noindent
As the function 
\begin{equation*}
z \mapsto f(z) \, e^{-ikz}
\end{equation*}
is analytic in $\mathcal{S}$ by assumption, we can apply Cauchy's integral theorem on the contour $\gamma$ to get
\begin{equation}\label{eq:uioooooolololo1}
\int_{\gamma([0, 1])} f(z) \, e^{-ikz} \, dz = 0,
\quad \text{for all } k\in \mathbb{Z}.
\end{equation}
Using the $2\pi$-periodicity 
\begin{equation*}
f(z) \, e^{-ikz} 
= f(z+2\pi) \, e^{-ikz}
= f(z+2\pi) \, e^{-ik(z+2\pi)},
\end{equation*}
it follows from (\ref{eq:uioooooolololo1}) that
\begin{equation}\label{eq:uioooooolololo2}
\int_{\gamma([1/4, 1/2])} f(z) \, e^{-ikz} \, dz + \int_{\gamma([3/4, 1])} f(z) \, e^{-ikz} \, dz = 0,
\quad \text{for all } k\in \mathbb{Z}.
\end{equation}
Combining (\ref{eq:uioooooolololo1}) and (\ref{eq:uioooooolololo2}) yields, for all $k\in \mathbb{Z}$,
\begin{align}
a_k = \frac{1}{2\pi}  \int_{0}^{2\pi} f(x) \, e^{-ikx} \, dx 
&= 4\int_{\gamma([0, 1/4])} f(z) \, e^{-ikz} \, dz \label{eq:poiuyttttoiuytre1}\\
&= - 4\int_{\gamma([1/2, 3/4])} f(z) \, e^{-ikz} \, dz \nonumber \\
&= \frac{1}{2\pi}  \int_{0}^{2\pi} f(x+is') \, e^{-ikx+ks'} \, dx.\label{eq:poiuyttttoiuytre2}
\end{align}
Rearranging terms, we obtain
\begin{equation*}
a_k e^{-ks'} = \frac{1}{2\pi}  \int_0^{2\pi}  f(x+is') \, e^{-ikx} \, dx,
\quad \text{for all } k\in \mathbb{Z},
\end{equation*}
which concludes the proof.
\end{proof}

\section{Complements on the Lambert W-function}\label{sec:WLambertappdix}

We recall results on the Lambert $W$-function.

\begin{definition}[Lambert $W$-function]\label{def:Wlambert}
The {Lambert $W$-function} is defined as the unique function satisfying 
\begin{equation*}
W(x)\, e^{W(x)} = x, \quad \text{for all } x \geq 0.
\end{equation*}
\end{definition}

\begin{lemma}\label{lem:equsolLambert}
Let $a$ and $b$ be positive real numbers.
Then, the equation
\begin{equation*}
x = ae^{-x} + b
\end{equation*}
has a unique solution, which is given by 
\begin{equation*}
x = b + W\left(ae^{-b}\right).
\end{equation*}
\end{lemma}

\begin{lemma}\label{lem:lbtseriesexp}
The Lambert $W$-function satisfies the relation
\begin{equation*}
W(x)
= {\ln(x)} - {\ln^{(2)}(x)} + \sum_{n=0}^\infty \sum_{m=1}^\infty \frac{(-1)^n s(n+m, n+1)}{m!} \frac{\left(\ln^{(2)}(x)\right)^m}{\ln^{n+m}(x)},
\end{equation*}
for all $x>0$, where $s(n+m, n+1)$ denotes the Stirling numbers of the first kind.
In particular, one has the asymptotic behavior
\begin{equation}\label{eq:lbtseriesexp}
W(x)
= {\ln(x)} - {\ln^{(2)}(x)} + \frac{\ln^{(2)}(x)}{\ln(x)} \left(1 + o_{x\to\infty}\left(1\right)\right).
\end{equation}
\end{lemma}

\noindent
We refer to \cite{corlessLambertFunction1996} for further material on the Lambert $W$-function and to \cite[Chapter 2.4]{de1981asymptotic} for a proof of Lemma \ref{lem:lbtseriesexp}. Lemma \ref{lem:equsolLambert} can be verified by direct substitution.

\end{document}